\documentclass[10pt]{amsart}
\usepackage{amssymb,times,amsmath,cancel,mathrsfs,multirow,textcomp}
\usepackage{graphicx,color,epsfig}
\usepackage{bbm}

\usepackage[all]{xy}
\usepackage{tikz}
\usetikzlibrary{calc,3d,arrows,decorations.markings}
\tikzset{->-/.style={decoration={  markings,  mark=at position .75 with {\arrow{latex}}},postaction={decorate}}}
\tikzset{-<-/.style={decoration={  markings,  mark=at position .75 with {\arrow{latex reversed}}},postaction={decorate}}}

\def\r{\rightarrow}

\def\ot{\leftarrow}

\def\Id{\id}

\let\cal\mathcal
\def\cA{{\mathtt A}}

\def\cF{{\cal F}}

\def\cI{{\cal I}}

\def\cL{{\cal L}}

\def\cO{{\cal O}}
\def\cP{{\cal P}}

\def\eps{{\epsilon}}

\let\blb\mathbb

\def\H{{\blb M}}

\def\UU{{\blb U}}

\def\QQ{{\blb Q}}

\def \PP{{\blb P}}

\def \VV{{\blb V}}

\newcommand{\koppa}{\raisebox{-2.5pt}{\begin{tikzpicture}\begin{scope}[xscale=.35,yscale=.27]\draw[thick] (0,0)--(-.25,.33)--(-.25,.66)--(0,1)--(.25,.66)--(.25,.33)--(0,0); \draw[thick] (0,.33)--(0,-.16);\end{scope}\end{tikzpicture}}}
\newcommand{\weg}[1]{}
\newcommand{\se}[1]{\begin{equation*}\begin{split}#1\end{split}\end{equation*}}

\newcommand{\C}{\mathbb{C}}
\newcommand{\N}{\mathbb{N}}
\newcommand{\Z}{\mathbb{Z}}
\newcommand{\R}{\mathbb{R}}

\newcommand{\genmu}{\bbmu}

\newcommand{\vtx}[1]{*+[o][F-]{\scriptscriptstyle #1}}

\newcommand{\<}{\langle}
\renewcommand{\>}{\rangle}

\newcommand{\Ext}{\mathsf{Ext}}

\newcommand{\supp}{\mathrm{supp}}

\newcommand{\wfuk}{\mathrm{wfuk}}
\newcommand{\sm}[1]{\left(\begin{smallmatrix}#1\end{smallmatrix}\right)}

\newcommand{\Cone}{\mathrm{Cone}}

\newcommand{\erbij}[1]{}

\newcommand{\NP}{\mathrm{NP}}
\newcommand{\MP}{\mathrm{MP}}
\newcommand{\PT}{\mathrm{PT}}
\newcommand{\PL}{\mathrm{PL}}
\newcommand{\slicing}{\Sigma}
\newcommand{\LH}{\mathrm{LH}}
\newcommand{\SG}{\mathrm{SG}}
\newcommand{\Semi}{\mathrm{Semi}}
\newcommand{\Mst}{\mathcal{M}}

\newtheorem{lemma}{Lemma}[section]
\newtheorem{proposition}[lemma]{Proposition}
\newtheorem{theorem}[lemma]{Theorem}
\newtheorem{corollary}[lemma]{Corollary}

\theoremstyle{definition}

\newtheorem{example}[lemma]{Example}
\newtheorem{definition}[lemma]{Definition}

{

}
\newtheorem{remark}[lemma]{Remark}

\theoremstyle{remark}

\newcommand{\rep}{\ensuremath{\mathsf{rep}}}

\newcommand{\trep}{\ensuremath{\mathsf{trep}}}

\newcommand{\Trop}{\ensuremath{\mathsf{Trop}}}
\newcommand{\Am}{\ensuremath{\mathsf{Am}}}

\newcommand{\Proj}{\ensuremath{\mathsf{Proj}}}

\newcommand{\Mat}{\mathsf{Mat}}

\newcommand{\Hom}{\mathtt{Hom}}

\newcommand{\GL}{\ensuremath{\mathsf{GL}}}
\newcommand{\id}{\mathbf{1}}

\newcommand{\PM}{\cP}

\renewcommand{\H}{\mathtt{H}\,}
\newcommand{\Tw}{\mathtt{Tw}\,}
\newcommand{\Sing}{\mathtt{Sing}\,}

\def\genmu{{\mu}}
\def\Der{{\mathtt {D}}}

\def\MF{{\mathtt {MF}}}

\def\cD{{\mathtt D}}

\def\mf{{\mathtt {mf}}}

\newcommand{\Gtl}{\mathtt{Gtl}}

\newcommand{\Jac}{{\mathtt{J}}}
\newcommand{\wJac}{{\widehat{\mathtt{J}}}}

\newcommand{\RQ}{{\koppa}}

\def\qpol{\mathrm{Q}}

\newcommand{\Mod}{\ensuremath{\mathtt{Mod}}}

\newcommand{\twist}[1]{\mathop{\reflectbox{$#1$}}}
\newcommand{\polq}{\twist{\qpol}}

\renewcommand{\C}{\mathbb{C}}
\newcommand{\kk}{\mathbbm{k}}

\newcommand{\surf}[1]{\bar S({#1})}
\newcommand{\psurf}[1]{\dot{S}({#1})}
\newcommand{\gsurf}[2]{{\vec{S}(#2,#1)}}
\newcommand{\rib}[1]{\mathring{S}({#1})}
\newcommand{\tube}[1]{\dot{S}({#1})}
\newcommand{\tub}[1]{\dot{S}({#1})}
\newcommand{\Msg}{\mathcal{M}^{sing}}
\newcommand{\Mla}{\mathcal{M}^{mf}}
\newcommand{\retr}{\mathtt{retr}}

\title{Strebel Differentials and stable Matrix Factorizations}

\author{Raf Bocklandt}
\address{Raf Bocklandt\\
Korteweg de Vries institute\\
University of Amsterdam (UvA)\\
Science Park 904\\ 
1098 XH Amsterdam\\ 
The Netherlands
}
\email{raf.bocklandt@gmail.com}

\xyoption{all}

\begin{document}
\begin{abstract}
We study the connection between quadratic Strebel differentials on punctured surfaces and 
the construction of moduli spaces of matrix factorizations for dimer models 
using GIT-quotients.
We show that for each consistent dimer model and each nondegenerate stability condition $\theta$ we can find
a Strebel differential for which the horizontal trajectories correspond to the $\theta$-stable 
matrix factorizations and the vertical trajectories correspond to the arrows of the dimer quiver.
We give explicit expressions for the $\theta$-stable matrix factorizations that can be deduced
from these horizontal trajectories. 

Following ideas by Pascaleff and Sybilla \cite{pascaleff2016topological} we show that each nondegenerate stability condition gives rise to a sheaf of curved algebras coming from consistent 
dimer models. The corresponding categories of matrix factorizations can be glued together to
form the category of matrix factorizations of the original dimer.
\end{abstract}

\maketitle
\section{Introduction}

For any punctured Riemann surface one can define an invariant, its wrapped Fukaya category (Abouzaid et al. \cite{abouzaid2013homological}) or 
topological Fukaya category (Haiden-Katzarkov-Kontsevich in \cite{haiden2014flat} or Dyckerhoff-Kapranov \cite{dyckerhoff2013triangulated}). 
As its second name suggests, this invariant only depends on the topology of the surface and intuitively it describes the intersection theory of curves on the surface.

The topological Fukaya category carries a rich structure, it is an $A_\infty$-category,
but unlike many other $A_\infty$-categories its description is fairly straightforward and
one can describe its objects and morphisms in a natural way. This makes it an ideal toy model
to study phenomena related to the concept of Mirror symmetry.

Mirror symmetry for punctured surfaces establishes an equivalence between the topological
Fukaya category (the A-side) and certain categories of matrix factorizations (the B-side). 
These categories of matrix factorizations can be defined over commutative spaces like in Abouzaid et al. \cite{abouzaid2013homological} and Pascaleff and Sibilla \cite{pascaleff2016topological} or 
over noncommutative space like in \cite{bocklandt2016noncommutative}.

In this paper we will study mirror symmetry for punctured surfaces from the point of view of constructing moduli spaces of objects. 
Inspired by work of Bridgeland and Smith \cite{bridgeland2015quadratic} and Haiden et al. \cite{haiden2014flat}, we will explain how to construct moduli space of objects in the topological Fukaya category using Strebel differentials and 
relate this to the classical GIT-construction of moduli spaces of representations of noncommutative algebras. 

The stepping stone between the two sides of our story is the theory of dimer models.
They give a combinatorial description of both the Fukaya category and the category of matrix factorizations. This description can then be used to construct moduli spaces and relate these moduli spaces to tropical geometry.  

The structure of the paper is as follows. We start with an review on quivers and $A_\infty$-structures in section \ref{sectionquivers}. Then we introduce dimer models in section \ref{sectionquivers} and explain how they can be used to describe mirror
symmetry for punctured surfaces in section \ref{sectionmirror}. In section \ref{sectionbands} we show how these dimer models can be used to describe objects in the 
Fukaya category and the category of matrix factorizations. Section \ref{sectionspider} and \ref{sectiontropical} are used to review some concepts that will be 
important in the story: ribbon graphs and spider graphs and tropical and toric geometry.
We will use these concepts to describe moduli spaces in both the A and B-side in section \ref{sectionmoduli} and relate these moduli spaces to the gluing construction of Pascaleff and Sibilla \cite{pascaleff2016topological} in section \ref{sectionglue}. 
We illustrate the theory with an example in section \ref{sectionexample}.

\section{Quivers and $A_\infty$-structures}\label{sectionquivers}

\subsection{Quivers}
A quiver consists of a set of vertices $Q_0$ and a set of arrow $Q_1$ together with two maps $h,t:Q_1 \to Q_0$ that assign
to each arrow its head and tail. A nontrivial path is a sequence of arrows $a_0\dots a_k$ such that $t(a_{i})=h(a_{i+1})$ (so arrows point to the left: $\stackrel{a_0}{\ot}\dots \stackrel{a_k}{\ot}$). A path is cyclic if $h(a_0)=t(a_k)$ and a cyclic path up to cyclic shifts is called a cycle. 
A trivial path is a vertex. The path algebra $\C Q$ is the complex vector space spanned by all paths. The product of two paths is their concatenation if possible and
zero otherwise. The vertices are orthogonal idempotents for the product and they generate a commutative subalgebra $\kk \cong \C^{\oplus Q_0}$.
A path algebra with relations is the quotient of a path algebra by a two-sided ideal that sits in the ideal spanned by all paths of length at least $2$. Such an
algebra can be considered as an algebra over $\kk$.

\subsection{$A_\infty$-algebras}
An $A_\infty$-algebra is a $\Z$-graded $\kk$-bimodule $A$ with a 
set of products $\mu_i : A^{\otimes_\kk i} \to A$ of degree $2-i$ subject to certain generalized associativity laws. For the explicit expressions of these laws and more background about $A_\infty$-algebras we refer to \cite{keller1999introduction,keller2006infinity,kontsevich606241notes}. 
If $\mu_i=0$ for $i\ne 2$ then $A$ is an ordinary graded algebra for the product $\mu_2$ and if $\mu_i=0$ for $i\ne 1,2$, $A$ is a dg-algebra for the product $\mu_2$ and 
the differental $\mu_1$. Note that the definition allows for a map $\mu_0: \kk \to A$. If that map is nonzero we say that the $A_\infty$-algebra is curved and the element $\mu_0(1) \in A$ is its curvature.
A curved $A_\infty$-algebra with $\mu_i=0$ for ${i\ne 0,2}$ is called a Landau-Ginzburg model, it consists of an ordinary graded algebra 
and a degree $2$ element $\ell=\mu_0(1)$, which is central.

If $\mu_0=0$ then the $A_\infty$-algebra is called flat. In this case $\mu_1^2=0$ so $(A, d=\mu_1)$ is a complex and we will denote its homology by $HA$.
If $A$ is flat and $\mu_1=0$ then we say that $A$ is minimal.

Let $A$ be a path algebra with relations and assume it has a $\Z$-grading such that the arrows are homogeneous. If $\mu$ is an $A_\infty$-structure on $A$, 
we say that $\mu$ is compatible with the ordinary algebra structure on $A$ if $\mu_1=0$, $\mu_2$ is the ordinary product, and $\mu_{k}(x_1,\dots,x_k)$ is zero
if $k\ge 3$ and one of the entries is a vertex. 

There is a notion of an $A_\infty$-morphism between two $A_\infty$-algebras $A$ and $B$. This is a set of maps $\cF_i: A^{\otimes_\kk i}\to B$ with additional constraints \cite{keller1999introduction,keller2006infinity,kontsevich606241notes}.
If $A$ and $B$ are both flat then $\cF_1$ will induce a morphism of complexes and we say that $\cF$ is a quasi-isomorphism if $\cF_1:A\to B$ is a quasi-isomorphism. 
\begin{theorem}[Minimal model theorem \cite{kadeishvili1980homology}]
If $A$ is a flat $A_\infty$-algebra then there is an $A_\infty$-structure on $HA$ and a $A_\infty$-quasi-isomorphism $\cF:A \to HA$.
\end{theorem}
The $A_\infty$-algebra $(HA, \mu)$ is called the minimal model of $A$. 

\subsection{Twisted completion}
Analogously to $A_\infty$-algebras we can also define $A_\infty$-categories in such a way that an $A_\infty$-algebra $A$ can be viewed as an $A_\infty$-category with one object for each vertex if we're working with a path algebra with relations.
From an $A_\infty$-algebra $A=\C Q/\cI$ we can define the $A_\infty$-category of twisted objects: $\Tw A$. A \emph{twisted object} \cite{keller1999introduction} is a pair $(M,\delta)$, 
where $M\in \N[\Z^{Q_0}]$ is a formal sum of vertices shifted by elements in $\Z$.
We will write such a sum as $v_1[i_1] \oplus \dots \oplus v_k[i_k]$ where the $v_j$ are vertices and the $i_j$ shifts.
The map $\delta$ is a $k\times k$-matrix with entries $\delta_{st}\in v_{i_s}Av_{i_t}$ 
of degree $i_s-i_t+1$ and subject to the Maurer-Cartan equation: i.e.
\[
 \sum_{n=0}^{\infty}(-1)^{\frac {n(n-1)}2}\mu_n(\delta,\dots,\delta)=0,
\]
where we extended $\genmu_n$ to matrices in the standard way. Note that for this to make sense this infinite sum has to be finite, which can be achieved if $\delta$ is upper triangular
or if all products $\mu_i$ are zero for $i\gg 0$. Therefore if we are not in the latter case we restrict to objects for which $\delta$ is upper triangular.

The homomorphism space between two such objects $(M,\delta)$ and $(M',\delta')$ is given by
\[
\Hom((M,\delta),(M',\delta')) := \bigoplus_{r,s} v'_sA v_r [i_s-i_r]
\]
which we equip with an $\cA_{\infty}$-structure as follows:
\[
 \mu(f_1,\dots,f_n) := \sum_{t=0}^{\infty}\sum_{i_0+\dots+i_n=t} \pm \genmu(\underbrace{\delta,\dots,\delta}_{i_0},f_1,\underbrace{\delta,\dots,\delta}_{i_1},\dots,f_n,\underbrace{\delta,\dots,\delta}_{i_n}).
\]
The $\pm$-sign is calculated by multiplying with a factor $(-1)^{n+t-k}$ for each $\delta$ in the expression on position $k$.
The $\cA_\infty$-category of twisted objects and their homomorphism spaces is denoted by $\Tw A$. Note that the Maurer-Cartan equation implies
that $\mu_0$ is zero, so $\Tw A$ is flat. 

\begin{remark}
If $A$ is an ordinary $\C$-algebra concentrated in degree $0$ 
we can view $\Tw A$ as the dg-category of complexes of finitely generated free $A$-modules.
\end{remark}

\subsection{The derived category}
Because $\Tw A$ is flat, we can construct a category $\Der A$, with the same objects but 
its hom-spaces are the degree zero part of the $\mu_1$-homology of the hom-spaces in $\Tw A$ and the ordinary product is the induced product on the homology. 
We will call this category the derived category of $A$. 
Unlike $\Tw A$, which is an $A_\infty$-category, the derived category is an ordinary category and
it is even triangulated \cite{keller2006infinity}. 

In the light of the minimal model theorem the derived category is minimal model of $\Tw A$ and therefore it is equiped with an additional $A_\infty$-structure $(\Der A,\mu)$. 
Quasi-isomorphic $A_\infty$-algebras will have quasi-equivalent twisted completions and 
therefore also equivalent derived categories.

\begin{remark}
In case that $A$ is an ordinary algebra concentrated in degree zero and every
finitely generated module has a resolution of finitely generated free $A$-modules,
$\Der A$ is equivalent with the bounded derived category of $A$-modules $\mathtt{D}^b\Mod A^{op}$. 
\end{remark}

\begin{remark}
If $A$ is not $\Z$-graded but $\Z/2\Z$-graded or $G$-graded for some other abelian group with a map $\Z\to G$, 
we can adjust our definitions to $G$-graded $A_\infty$-algebras and twisted objects with $G$-shifts. 
The corresponding categories will be denoted by $\Tw_{G}A$ and $\Der_{G}A$. 
\end{remark}

\section{Dimers}\label{sectiondimers}

\subsection{Definition}
A dimer quiver $\qpol$ is a quiver that is embedded in a compact orientable surface, such that the complement of $\qpol$ is a union of polygons bounded by oriented 
cycles of the quiver. These cycles form a set $\qpol_2$ that is split in two $\qpol_2^+\cup \qpol_2^-$ according to their orientation on the surface (anticlockwise vs. clockwise). 
Every arrow is contained in precisely one cycle in $\qpol_2^+$ and one in $\qpol_2^-$. 
Examples of dimers can be found in \ref{lotsofmirrors}.

We denote the compact surface in which $\qpol$ is embedded by $\surf \qpol$. Its Euler characteristic equals 
$\chi(\qpol) = \#\qpol_0 -\#\qpol_1 +\#\qpol_2$. If we remove the vertices from the surface we get a punctured surface, 
which we denote by $\psurf{\qpol}:=   \surf \qpol \setminus \qpol_0$. If we cut out open disks around each vertex we get a surface with boundary which we denote by $\rib{\qpol}$.

\subsection{Perfect matchings and zigzag paths}
A perfect matching of $\qpol$ is a set of arrows $\PM \subset \qpol_1$ such that every cycle in $\qpol_2$ contains exactly one arrow of $\PM$.
Not every dimer admits a perfect matching: a necessary but not sufficient condition is that $\#\qpol_2^+=\#\qpol_2^-$.
Given a perfect matching $\PM$ we can define a degree function $\deg_\PM$ on $\C \qpol$ such that an arrow has degree one if it sits in $\PM$ and zero otherwise.
In this way all cycles in $\qpol_2$ have degree $1$.

The zig ray of an arrow $a \in \qpol_1$ is the infinite path $\dots a_2a_1a_0$ such that $a_{i+1}a_{i}$ is a subpath of a positive cycle if $i$ is even and of a 
negative cycle if $i$ is odd. The zag ray is defined similarly: $a_{i+1}a_{i}$ is a subpath of a positive cycle if $i$ is odd and of a 
negative cycle if $i$ is even. If $\qpol$ is finite the zig and zag rays become cyclic and we call them zigzag cycles.
Every arrow $a$ is contained in two zigzag cycles, one coming from its zig ray and one from its zag ray.

\subsection{Dimer Duality}
Given a dimer $\qpol$ we can construct a new dimer $\polq$. 
Geometricaly we can construct $\polq$ from $\qpol$ by cutting open $|\qpol|$ along the arrows. Then we flip over the polygons that come from negative cycles to their mirror images
and reverse the orientations of the sides (so that there are still oriented clockwise). Finally we glue all polygons together along sides that come from the same arrows. 
In this way we get a new quiver embedded in a new surface with a possibly different topology of the first dimer. 
The new dimer is also called the \emph{untwisted dimer}\cite{feng2008dimer}, the \emph{mirror dimer}\cite{bocklandt2016noncommutative} or the \emph{specular dual}\cite{hanany2012brane}.
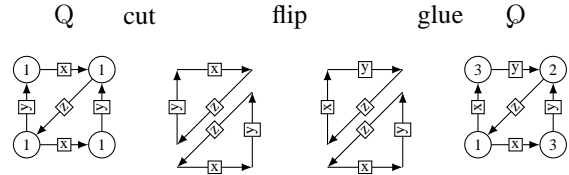
\begin{center}
\begin{tikzpicture}
\begin{scope}
\draw (.5,2) node{$\qpol$};
\draw [-latex,shorten >=5pt] (0,0.3) to node [rectangle,draw,fill=white,sloped,inner sep=1pt] {{\tiny x}} (1,0.3);
\draw [-latex,shorten >=5pt] (1,0.3) to node [rectangle,draw,fill=white,sloped,inner sep=1pt] {{\tiny y}} (1,1.3);
\draw [-latex,shorten >=5pt] (1,1.3) to node [rectangle,draw,fill=white,sloped,inner sep=1pt] {{\tiny z}} (0,0.3);
\draw [-latex,shorten >=5pt] (0,0.3) to node [rectangle,draw,fill=white,sloped,inner sep=1pt] {{\tiny y}} (0,1.3);
\draw [-latex,shorten >=5pt] (0,1.3) to node [rectangle,draw,fill=white,sloped,inner sep=1pt] {{\tiny x}} (1,1.3);
\draw (0,0.3) node[circle,draw,fill=white,minimum size=10pt,inner sep=1pt] {{\tiny1}};
\draw (0,1.3) node[circle,draw,fill=white,minimum size=10pt,inner sep=1pt] {{\tiny1}};
\draw (1,1.3) node[circle,draw,fill=white,minimum size=10pt,inner sep=1pt] {{\tiny1}};
\draw (1,0.3) node[circle,draw,fill=white,minimum size=10pt,inner sep=1pt] {{\tiny1}};
\end{scope}

\begin{scope}[xshift=2cm]
\draw (-.5,2) node{cut};
\draw [-latex] (0,0) to node [rectangle,draw,fill=white,sloped,inner sep=1pt] {{\tiny x}} (1,0);
\draw [-latex] (1,0) to node [rectangle,draw,fill=white,sloped,inner sep=1pt] {{\tiny y}} (1,1);
\draw [-latex] (1,1) to node [rectangle,draw,fill=white,sloped,inner sep=1pt] {{\tiny z}} (0,0);

\draw [-latex] (1,1.3) to node [rectangle,draw,fill=white,sloped,inner sep=1pt] {{\tiny z}} (0,0.3);
\draw [-latex] (0,0.3) to node [rectangle,draw,fill=white,sloped,inner sep=1pt] {{\tiny y}} (0,1.3);
\draw [-latex] (0,1.3) to node [rectangle,draw,fill=white,sloped,inner sep=1pt] {{\tiny x}} (1,1.3);
\end{scope}

\begin{scope}[xshift=4cm]
\draw (-.5,2) node{flip};
\draw [-latex] (0,0) to node [rectangle,draw,fill=white,sloped,inner sep=1pt] {{\tiny x}} (1,0);
\draw [-latex] (1,0) to node [rectangle,draw,fill=white,sloped,inner sep=1pt] {{\tiny y}} (1,1);
\draw [-latex] (1,1) to node [rectangle,draw,fill=white,sloped,inner sep=1pt] {{\tiny z}} (0,0);

\draw [-latex] (1,1.3) to node [rectangle,draw,fill=white,sloped,inner sep=1pt] {{\tiny z}} (0,0.3);
\draw [-latex] (0,0.3) to node [rectangle,draw,fill=white,sloped,inner sep=1pt] {{\tiny x}} (0,1.3);
\draw [-latex] (0,1.3) to node [rectangle,draw,fill=white,sloped,inner sep=1pt] {{\tiny y}} (1,1.3);
\end{scope}

\begin{scope}[xshift=6cm]
\draw (-.5,2) node{glue};
\draw (.5,2) node{$\polq$};
\draw [-latex,shorten >=5pt] (0,0.3) to node [rectangle,draw,fill=white,sloped,inner sep=1pt] {{\tiny x}} (1,0.3);
\draw [-latex,shorten >=5pt] (1,0.3) to node [rectangle,draw,fill=white,sloped,inner sep=1pt] {{\tiny y}} (1,1.3);
\draw [-latex,shorten >=5pt] (1,1.3) to node [rectangle,draw,fill=white,sloped,inner sep=1pt] {{\tiny z}} (0,0.3);
\draw [-latex,shorten >=5pt] (0,0.3) to node [rectangle,draw,fill=white,sloped,inner sep=1pt] {{\tiny x}} (0,1.3);
\draw [-latex,shorten >=5pt] (0,1.3) to node [rectangle,draw,fill=white,sloped,inner sep=1pt] {{\tiny y}} (1,1.3);
\draw (0,0.3) node[circle,draw,fill=white,minimum size=10pt,inner sep=1pt] {{\tiny1}};
\draw (0,1.3) node[circle,draw,fill=white,minimum size=10pt,inner sep=1pt] {{\tiny3}};
\draw (1,1.3) node[circle,draw,fill=white,minimum size=10pt,inner sep=1pt] {{\tiny2}};
\draw (1,0.3) node[circle,draw,fill=white,minimum size=10pt,inner sep=1pt] {{\tiny3}};
\end{scope}

\end{tikzpicture}
\end{center}

From the geometric construction it is easy to see that it forms an involution on the set of dimers. An interesting observation
is that this involution induces a bijection between the set of perfect matchings of $\qpol$ and $\polq$. 
Moreover the vertices of $\polq$ are in one to one correspondence to the zigzag cycles in $\qpol$ and vice versa.

\subsection{Examples}\label{lotsofmirrors}
We illustrate this with some extra examples from \cite{bocklandt2016noncommutative}:
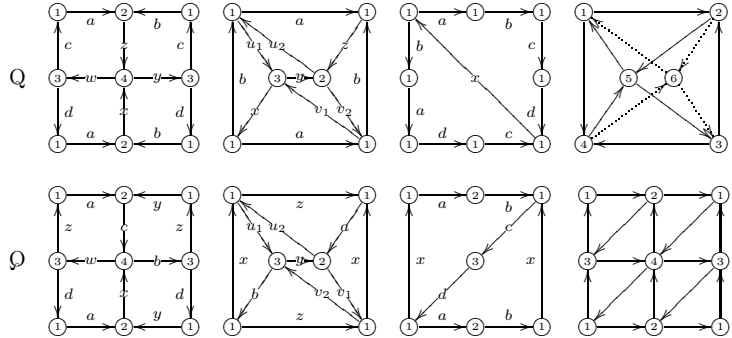
\begin{center}
\resizebox{10cm}{!}{
\begin{tabular}{ccccc}
$\qpol$&$\vcenter{
\xymatrix@C=.75cm@R=.75cm{
\vtx{1}\ar[r]_{a}&\vtx{2}\ar[d]|z&\vtx{1}\ar[l]^b\\
\vtx{3}\ar[u]_c\ar[d]^d&\vtx{4}\ar[l]|w\ar[r]|y&\vtx{3}\ar[u]^c\ar[d]_d\\
\vtx{1}\ar[r]^{a}&\vtx{2}\ar[u]|x&\vtx{1}\ar[l]_b
}}$
&
$\vcenter{
\xymatrix@C=.4cm@R=.75cm{
\vtx{1}\ar[rrr]_a\ar[dr]|{u_1}&&&\vtx{1}\ar[ld]|z\\
&\vtx{3}\ar[r]|y\ar[ld]|x&\vtx{2}\ar[ull]|{u_2}\ar[dr]|{v_2}&\\
\vtx{1}\ar[rrr]^a\ar[uu]_b&&&\vtx{1}\ar[ull]|{v_1}\ar[uu]^b
}}$
&
$\vcenter{
\xymatrix@C=.75cm@R=.75cm{
\vtx{1}\ar[r]_{a}\ar[d]^{b}&\vtx{1}\ar[r]_b&\vtx{1}\ar[d]_c\\
\vtx{1}\ar[d]^a&&\vtx{1}\ar[d]_d\\	
\vtx{1}\ar[r]^{d}&\vtx{1}\ar[r]^c&\vtx{1}\ar[uull]|x
}}$
&
$\vcenter{
\xymatrix@C=.4cm@R=.75cm{
\vtx{1}\ar[dd]\ar[rrr]&&&\vtx{2}\ar[dll]\ar@{.>}[dl]\\
&\vtx{5}\ar[ul]\ar[drr]&\vtx{6}\ar@{.>}[ull]\ar@{.>}[dr]&\\
\vtx{4}\ar[ur]\ar@{.>}[urr]&&&\vtx{3}\ar[uu]\ar[lll]
}}$
\vspace{.5cm}
\\
$\polq$&
$\vcenter{
\xymatrix@C=.75cm@R=.75cm{
\vtx{1}\ar[r]_{a}&\vtx{2}\ar[d]|c&\vtx{1}\ar[l]^y\\
\vtx{3}\ar[u]_z\ar[d]^d&\vtx{4}\ar[l]|w\ar[r]|b&\vtx{3}\ar[u]^z\ar[d]_d\\
\vtx{1}\ar[r]^{a}&\vtx{2}\ar[u]|x&\vtx{1}\ar[l]_y
}}$
&
$\vcenter{
\xymatrix@C=.4cm@R=.75cm{
\vtx{1}\ar[rrr]_z\ar[dr]|{u_1}&&&\vtx{1}\ar[ld]|a\\
&\vtx{3}\ar[r]|y\ar[ld]|b&\vtx{2}\ar[ull]|{u_2}\ar[dr]|{v_1}&\\
\vtx{1}\ar[rrr]^z\ar[uu]_x&&&\vtx{1}\ar[ull]|{v_2}\ar[uu]^x
}}$
&
$\vcenter{
\xymatrix@C=.75cm@R=.75cm{
\vtx{1}\ar[r]_{a}&\vtx{2}\ar[r]_b&\vtx{1}\ar[ld]|c\\
&\vtx{3}\ar[ld]|d&\\	
\vtx{1}\ar[uu]_x\ar[r]^{a}&\vtx{2}\ar[r]^b&\vtx{1}\ar[uu]^x
}}$
&
$\vcenter{
\xymatrix@C=.75cm@R=.75cm{
\vtx{1}\ar[r]&\vtx{2}\ar[r]\ar[ld]&\vtx{1}\ar[ld]\\
\vtx{3}\ar[r]\ar[u]&\vtx{4}\ar[r]\ar[u]\ar[ld]&\vtx{3}\ar[u]\ar[ld]\\
\vtx{1}\ar[r]\ar[u]&\vtx{2}\ar[r]\ar[u]&\vtx{1}\ar[u]
}}$
\end{tabular}}
\end{center}
On the top row the first two quivers are embedded in a torus, the third in a surface with genus $2$ and the fourth in a sphere.
The mirrors on the bottom row are all embedded in a torus.
Note that the first $2$ are isomorphic to their mirrors, but in a nontrivial way.

\subsection{Consistency}
A dimer is called \emph{zigzag consistent} if in the universal cover the zig and the zag ray of an arrow $a$ do not have arrows in common apart from $a$. In the example above the first and third from the $\qpol$-row are consistent.
The second is not consistent because the zig and zag ray from $x$ both contain $z$.
In the $\polq$-row the first and fourth are consistent. 

\begin{remark}
Many different versions of consistency can be found in the literature \cite{broomhead2012dimer, gulotta2008properly, ishii2010note, bocklandt2016dimer,mozgovoy2009crepant,davison2011consistency}. 
For dimers on the torus these are equivalent to the notion of zigzag consistency. For more information about these equivalences we refer to \cite{bocklandt2016dimer,ishii2010note}. 
\end{remark}

Zigzag consistent dimers have nice properties, especially when they are embedded in a torus. 
Suppose that $\surf{\qpol}$ is torus and fix two cyclic paths $p_X,p_Y$ in $\qpol$ that span the homology of the torus. We can assign to each 
perfect matching a lattice point $$(\deg_\PM p_X,\deg_\PM p_Y) \in \Z^2.$$ 
These lattice points span a convex polygon which is called the \emph{matching polygon} $\MP(\qpol)$.
Note that it depends on the choice of $p_X,p_Y$ but different choices will result in polygons that are equal up to an affine integral transformation.

\begin{theorem}\label{matchingconsistent}
If $\qpol$ is a zigzag consistent dimer on a torus then
\begin{enumerate}
 \item On each lattice point of the matching polygon there is at least one perfect matching.
 \item On each corner of the matching polygon there is exactly one perfect matching.
 \item There is a one-to-one correspondence between the zigzag cycles and the outward pointing normal vectors of the zigzag polygon.
\end{enumerate}
\end{theorem}

This theorem can also be translated to some properties of the mirror dimer
\begin{theorem}\label{genusmirror}
If $\qpol$ is a zigzag consistent dimer on a torus then
\begin{enumerate}
\item The genus of $\surf{\polq}$ is equal to the internal lattice points of the matching polygon.
\item The number of punctures $\psurf{\polq}$ is equal to the number of bondary lattice points of the matching polygon.
\end{enumerate}
\end{theorem}

\begin{remark}
These two theorems are known by experts and appear in many different disguises in the literature e.g \cite{ishii2009dimer, mozgovoy2010noncommutative}. For a proof of these we refer to \cite{bocklandt2016dimer}.
\end{remark}

\subsection{Example}
The suspended pinchpoint \cite{franco2006brane}[section 4.1] is an example of a dimer model on the torus. Below we show the dimer $\qpol$ and its mirror $\polq$ (see also \cite{bocklandt2016dimer}). 

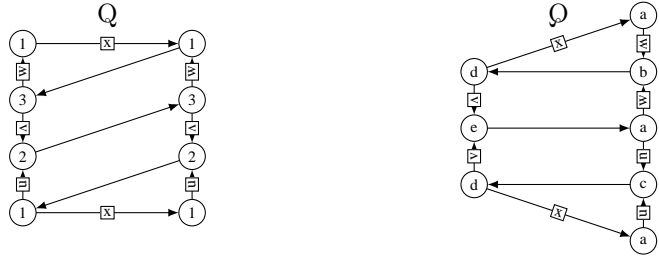
\begin{center}
\begin{tikzpicture}
\begin{scope}[scale=.75]
\draw (1.5,3.5) node{$\qpol$};
\draw [-latex,shorten >=5pt] (0,0) to node [rectangle,draw,fill=white,sloped,inner sep=1pt] {{\tiny x}} (3,0);
\draw [-latex,shorten >=5pt] (3,0) to node [rectangle,draw,fill=white,sloped,inner sep=1pt] {{\tiny u}} (3,1);
\draw [-latex,shorten >=5pt] (3,1) -- (0,0);
\draw [-latex,shorten >=5pt] (0,0) to node [rectangle,draw,fill=white,sloped,inner sep=1pt] {{\tiny u}} (0,1);
\draw [-latex,shorten >=5pt] (0,1) -- (3,2);
\draw [-latex,shorten >=5pt] (3,2) to node [rectangle,draw,fill=white,sloped,inner sep=1pt] {{\tiny w}} (3,3);
\draw [-latex,shorten >=5pt] (3,2) to node [rectangle,draw,fill=white,sloped,inner sep=1pt] {{\tiny v}} (3,1);
\draw [-latex,shorten >=5pt] (0,2) to node [rectangle,draw,fill=white,sloped,inner sep=1pt] {{\tiny v}} (0,1);
\draw [-latex,shorten >=5pt] (0,3) to node [rectangle,draw,fill=white,sloped,inner sep=1pt] {{\tiny x}} (3,3);
\draw [-latex,shorten >=5pt] (3,3) -- (0,2);
\draw [-latex,shorten >=5pt] (0,2) to node [rectangle,draw,fill=white,sloped,inner sep=1pt] {{\tiny w}} (0,3);
\draw (0,0) node[circle,draw,fill=white,minimum size=10pt,inner sep=1pt] {{\tiny1}};
\draw (0,1) node[circle,draw,fill=white,minimum size=10pt,inner sep=1pt] {{\tiny2}};
\draw (0,2) node[circle,draw,fill=white,minimum size=10pt,inner sep=1pt] {{\tiny3}};
\draw (0,3) node[circle,draw,fill=white,minimum size=10pt,inner sep=1pt] {{\tiny1}};
\draw (3,0) node[circle,draw,fill=white,minimum size=10pt,inner sep=1pt] {{\tiny1}};
\draw (3,1) node[circle,draw,fill=white,minimum size=10pt,inner sep=1pt] {{\tiny2}};
\draw (3,2) node[circle,draw,fill=white,minimum size=10pt,inner sep=1pt] {{\tiny3}};
\draw (3,3) node[circle,draw,fill=white,minimum size=10pt,inner sep=1pt] {{\tiny1}};
\end{scope}

\begin{scope}[xshift=6cm]
\begin{scope}[scale=.75]
\draw (1.5,3.5) node{$\polq$};
\draw [-latex,shorten >=5pt] (0,.5) to node [rectangle,draw,fill=white,sloped,inner sep=1pt] {{\tiny x}} (3,-.5);
\draw [-latex,shorten >=5pt] (3,-.5) to node [rectangle,draw,fill=white,sloped,inner sep=1pt] {{\tiny u}} (3,.5);
\draw [-latex,shorten >=5pt] (3,.5) -- (0,.5);
\draw [-latex,shorten >=5pt] (0,.5) to node [rectangle,draw,fill=white,sloped,inner sep=1pt] {{\tiny v}} (0,1.5);
\draw [-latex,shorten >=5pt] (0,2.5) to node [rectangle,draw,fill=white,sloped,inner sep=1pt] {{\tiny v}} (0,1.5);
\draw [-latex,shorten >=5pt] (3,1.5) to node [rectangle,draw,fill=white,sloped,inner sep=1pt] {{\tiny u}} (3,.5);
\draw [-latex,shorten >=5pt] (3,1.5) to node [rectangle,draw,fill=white,sloped,inner sep=1pt] {{\tiny w}} (3,2.5);
\draw [-latex,shorten >=5pt] (0,1.5) -- (3,1.5);
\draw [-latex,shorten >=5pt] (3,2.5) -- (0,2.5);
\draw [-latex,shorten >=5pt] (0,2.5) to node [rectangle,draw,fill=white,sloped,inner sep=1pt] {{\tiny x}} (3,3.5);
\draw [-latex,shorten >=5pt] (3,3.5) to node [rectangle,draw,fill=white,sloped,inner sep=1pt] {{\tiny w}} (3,2.5);
\draw (0,.5) node[circle,draw,fill=white,minimum size=10pt,inner sep=1pt] {{\tiny d}};
\draw (0,2.5) node[circle,draw,fill=white,minimum size=10pt,inner sep=1pt] {{\tiny d}};
\draw (0,1.5) node[circle,draw,fill=white,minimum size=10pt,inner sep=1pt] {{\tiny e}};
\draw (3,.5) node[circle,draw,fill=white,minimum size=10pt,inner sep=1pt] {{\tiny c}};
\draw (3,2.5) node[circle,draw,fill=white,minimum size=10pt,inner sep=1pt] {{\tiny b}};
\draw (3,1.5) node[circle,draw,fill=white,minimum size=10pt,inner sep=1pt] {{\tiny a}};
\draw (3,-.5) node[circle,draw,fill=white,minimum size=10pt,inner sep=1pt] {{\tiny a}};
\draw (3,3.5) node[circle,draw,fill=white,minimum size=10pt,inner sep=1pt] {{\tiny a}};
\end{scope}
\end{scope}
\end{tikzpicture}
\end{center}

There are five zigzag paths in $\qpol$ with homology classes normal to the five line segments on the boundary. 
There are six perfect matchings, one for each corner and two on the lattice point that is not a corner.
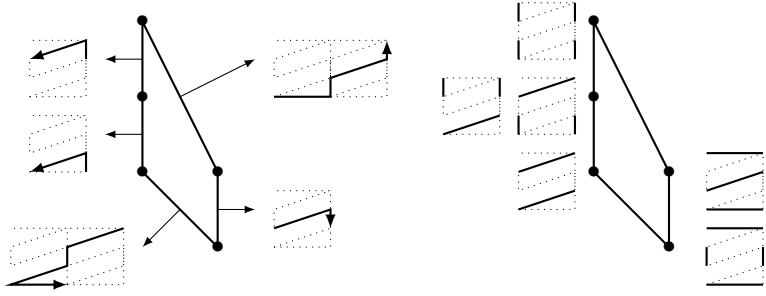
\begin{center}
\begin{tikzpicture}
\draw[thick](0,0) -- (0,2) -- (1,0) -- (1,-1) -- (0,0); 
\draw (0,0) node{$\bullet$};
\draw (0,2) node{$\bullet$};
\draw (0,1) node{$\bullet$};
\draw (1,0) node{$\bullet$};
\draw (1,-1) node{$\bullet$};
\draw[-latex] (0,.5)--(-.5,.5);
\draw[-latex] (0,1.5)--(-.5,1.5);
\draw[-latex] (.5,1)--(1.5,1.5);
\draw[-latex] (1,-.5)--(1.5,-.5);
\draw[-latex] (.5,-.5)--(0,-1);
\begin{scope}[xshift=-1.5cm]
\begin{scope}[scale=.25]
\draw [dotted] (0,0) -- (3,0) -- (3,3) -- (0,3);
\draw [dotted] (0,0) -- (3,1) -- (3,2) -- (0,1)--(0,2)--(3,3);
\draw [thick,-latex] (3,0) -- (3,1)--(0,0);
\end{scope}
\end{scope}
\begin{scope}[xshift=-1.5cm,yshift=1cm]
\begin{scope}[scale=.25]
\draw [dotted] (0,0) -- (3,0) -- (3,3) -- (0,3);
\draw [dotted] (0,0) -- (3,1) -- (3,2) -- (0,1)--(0,2)--(3,3);
\draw [thick,-latex] (3,2) -- (3,3)--(0,2);
\end{scope}
\end{scope}
\begin{scope}[xshift=1.75cm,yshift=-1cm]
\begin{scope}[scale=.25]
\draw [dotted] (0,0) -- (3,0) -- (3,3) -- (0,3);
\draw [dotted] (0,0) -- (3,1) -- (3,2) -- (0,1)--(0,2)--(3,3);
\draw [thick,-latex] (0,1) -- (3,2)--(3,1);
\end{scope}
\end{scope}
\begin{scope}[xshift=-1.75cm,yshift=-1.5cm]
\begin{scope}[scale=.25]
\draw [dotted] (0,0) -- (3,0) -- (3,3) -- (0,3);
\draw [dotted] (0,0) -- (3,1) -- (3,2) -- (0,1)--(0,2)--(3,3);
\draw [dotted] (3,0) -- (6,0) -- (6,3) -- (3,3);
\draw [dotted] (3,0) -- (6,1) -- (6,2) -- (3,1)--(3,2)--(6,3);
\draw [thick,-latex] (6,3) -- (3,2)--(3,1)--(0,0)--(3,0);
\end{scope}
\end{scope}
\begin{scope}[xshift=1.75cm,yshift=1cm]
\begin{scope}[scale=.25]
\draw [dotted] (0,0) -- (3,0) -- (3,3) -- (0,3);
\draw [dotted] (0,0) -- (3,1) -- (3,2) -- (0,1)--(0,2)--(3,3);
\draw [dotted] (3,0) -- (6,0) -- (6,3) -- (3,3);
\draw [dotted] (3,0) -- (6,1) -- (6,2) -- (3,1)--(3,2)--(6,3);
\draw [thick,-latex] (0,0) -- (3,0)--(3,1)--(6,2)--(6,3);
\end{scope}
\end{scope}

\begin{scope}[xshift=6cm]
\draw[thick](0,0) -- (0,2) -- (1,0) -- (1,-1) -- (0,0); 
\draw (0,0) node{$\bullet$};
\draw (0,2) node{$\bullet$};
\draw (0,1) node{$\bullet$};
\draw (1,0) node{$\bullet$};
\draw (1,-1) node{$\bullet$};

\begin{scope}[xshift=-1cm,yshift=1.5 cm]
\begin{scope}[scale=.25]
\draw [dotted] (0,0) -- (3,0) -- (3,3) -- (0,3);
\draw [dotted] (0,0) -- (3,1) -- (3,2) -- (0,1)--(0,2)--(3,3);
\draw [thick] (0,0) -- (0,1);
\draw [thick] (3,0) -- (3,1);
\draw [thick] (0,2) -- (0,3);
\draw [thick] (3,2) -- (3,3);
\end{scope}
\end{scope}

\begin{scope}[xshift=-2cm,yshift=.5 cm]
\begin{scope}[scale=.25]
\draw [dotted] (0,0) -- (3,0) -- (3,3) -- (0,3);
\draw [dotted] (0,0) -- (3,1) -- (3,2) -- (0,1)--(0,2)--(3,3);
\draw [thick] (0,0) -- (3,1);
\draw [thick] (0,2) -- (0,3);
\draw [thick] (3,2) -- (3,3);
\end{scope}
\end{scope}

\begin{scope}[xshift=-1cm,yshift=.5 cm]
\begin{scope}[scale=.25]
\draw [dotted] (0,0) -- (3,0) -- (3,3) -- (0,3);
\draw [dotted] (0,0) -- (3,1) -- (3,2) -- (0,1)--(0,2)--(3,3);
\draw [thick] (0,0) -- (0,1);
\draw [thick] (3,0) -- (3,1);
\draw [thick] (0,2) -- (3,3);
\end{scope}
\end{scope}

\begin{scope}[xshift=-1cm,yshift=-.5cm]
\begin{scope}[scale=.25]
\draw [dotted] (0,0) -- (3,0) -- (3,3) -- (0,3);
\draw [dotted] (0,0) -- (3,1) -- (3,2) -- (0,1)--(0,2)--(3,3);
\draw [thick] (0,0) -- (3,1);
\draw [thick] (0,2) -- (3,3);
\end{scope}
\end{scope}

\begin{scope}[xshift=1.5cm,yshift=-1.5cm]
\begin{scope}[scale=.25]
\draw [dotted] (0,0) -- (3,0) -- (3,3) -- (0,3);
\draw [dotted] (0,0) -- (3,1) -- (3,2) -- (0,1)--(0,2)--(3,3);
\draw [thick] (0,1) -- (0,2);
\draw [thick] (3,1) -- (3,2);
\draw [thick] (0,0) -- (3,0);
\draw [thick] (0,3) -- (3,3);
\end{scope}
\end{scope}

\begin{scope}[xshift=1.5cm,yshift=-.5cm]
\begin{scope}[scale=.25]
\draw [dotted] (0,0) -- (3,0) -- (3,3) -- (0,3);
\draw [dotted] (0,0) -- (3,1) -- (3,2) -- (0,1)--(0,2)--(3,3);
\draw [thick] (0,1) -- (3,2);
\draw [thick] (0,0) -- (3,0);
\draw [thick] (0,3) -- (3,3);
\end{scope}
\end{scope}
\end{scope}
\end{tikzpicture}
\end{center}
As there are no internal lattice points and 5 lattice points on the boundary, the mirror dimer describes a sphere with 5 punctures.
This can also be seen by gluing the dimer on the right together.

\section{Mirror Symmetry for dimers}\label{sectionmirror}
\subsection{The A-model}

For a dimer $\qpol$ we define a second quiver $\RQ$: its vertices are the midpoints of the arrows of $\qpol$ and its arrows are angle arcs inside
the polygons that connect these arrows in a clockwise fashion. 

From this new quiver we can define an $A_\infty$-algebra $\Gtl(\qpol)$.
As an ordinary algebra $A =\Gtl(\qpol)$ is the path algebra $\C \RQ$ modulo the relations of the form $\alpha\beta$ where $\alpha$ and $\beta$ are two consecutive
angle arcs in a cycle $c \in \qpol_2$. 
\begin{example}\label{torus1}
Below is an example for a torus with one marked point.
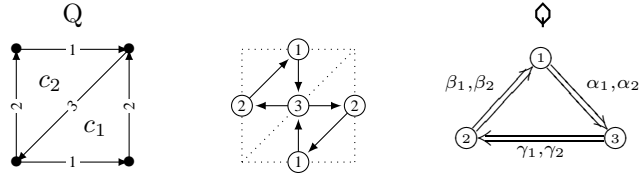
\begin{center}
\begin{tikzpicture}
\begin{scope}[scale=1.5]
\draw (.5,1.6) node{$\qpol$};
\draw [-latex] (0,0.3) to node [fill=white,sloped,inner sep=1pt] {{\tiny 1}} (1,0.3);
\draw [-latex] (1,0.3) to node [fill=white,sloped,inner sep=1pt] {{\tiny 2}} (1,1.3);
\draw [-latex] (1,1.3) to node [fill=white,sloped,inner sep=1pt] {{\tiny 3}} (0,0.3);
\draw [-latex] (0,0.3) to node [fill=white,sloped,inner sep=1pt] {{\tiny 2}} (0,1.3);
\draw [-latex] (0,1.3) to node [fill=white,sloped,inner sep=1pt] {{\tiny 1}} (1,1.3);
\draw (0,0.3) node {{$\bullet$}};
\draw (0,1.3) node {{$\bullet$}};
\draw (1,1.3) node {{$\bullet$}};
\draw (1,0.3) node {{$\bullet$}};
\draw (.7,0.6) node {{$c_1$}};
\draw (.3,1) node {{$c_2$}};
\end{scope}

\begin{scope}[xshift=3cm]
\begin{scope}[scale=1.5]
\draw [-latex,shorten >=5pt] (.5,0.8) -- (0,0.8);
\draw [-latex,shorten >=5pt] (.5,0.8) -- (1,0.8);
\draw [-latex,shorten >=5pt] (.5,.3)--(.5,0.8);
\draw [-latex,shorten >=5pt] (.5,1.3)--(.5,0.8);
\draw [-latex,shorten >=5pt] (0,0.8) -- (.5,1.3);
\draw [-latex,shorten >=5pt] (1,0.8) -- (.5,0.3);
\draw [dotted] (0,0.3) -- (1,0.3);
\draw [dotted] (1,0.3) -- (1,1.3);
\draw [dotted] (1,1.3) --  (0,0.3);
\draw [dotted] (0,0.3) -- (0,1.3);
\draw [dotted] (0,1.3) -- (1,1.3);
\draw (1,0.8) node [circle,draw,fill=white,sloped,inner sep=1pt] {{\tiny 2}};
\draw (0,0.8) node [circle,draw,fill=white,sloped,inner sep=1pt] {{\tiny 2}};
\draw (.5,0.8) node [circle,draw,fill=white,sloped,inner sep=1pt] {{\tiny 3}};
\draw (0.5,0.3) node [circle,draw,fill=white,sloped,inner sep=1pt] {{\tiny 1}};
\draw (0.5,1.3) node [circle,draw,fill=white,sloped,inner sep=1pt] {{\tiny 1}};
\end{scope}
\end{scope}

\begin{scope}[xshift=6.5cm]
\draw (.5,2.4) node{$\RQ$};
\draw (.5,1.2) node{$\xymatrix@R=.75cm@C=.66cm{&\vtx{1}\ar@2[rd]^{\alpha_1,\alpha_2}&\\\vtx{2}\ar@2[ur]^{\beta_1,\beta_2}&&\vtx{3}\ar@2[ll]^{\gamma_1,\gamma_2}}$};
\end{scope}
\end{tikzpicture}
\[
 \Gtl(\qpol) \cong \frac{\C\RQ}{\<\alpha_i\beta_i, \beta_i\gamma_i, \gamma_i\alpha_i| i=1,2\>}
\]
\end{center}
\end{example}

The algebra $\Gtl(\qpol)$ comes with a natural $A_\infty$-structure \cite{bocklandt2016noncommutative}.
The higher products $\genmu:A^{\otimes k}\to A$ are defined inductively: let $\rho_1,\dots,\rho_k$ be any sequence of paths 
and $\beta_1\dots \beta_l$ a cycle of angle arrows that goes around a cycle of $\qpol_2$ with $h(\beta_1)=t(\rho_i)$. 
We set 
\[
\text{[IR] : }\genmu(\rho_1,\dots, \rho_i\beta_{1},\beta_{2},\dots, \beta_{l-1},\beta_{l}\rho_{i+1},\dots, \rho_k) := \pm \genmu(\rho_1,\dots,\rho_k).
\]
For the sign convention we refer to \cite{bocklandt2016noncommutative}. Pictorially this gives rise to the following diagram:
\[
\genmu\left(
\vcenter{\xymatrix@=.3cm{&\ar[dr]&\\
\ar[ur]&&\ar[lldd]_(.8){\rho_{i}\beta_1}\\
&&\\
\dots&&\dots\ar[lluu]_(.2){\beta_l\rho_{i+1}}
}}
\right)=\pm 
\genmu\left(
\vcenter{\xymatrix@=.3cm{
&\ar[ld]_{\rho_{i}}&\\
\dots&&\dots\ar[lu]_{\rho_{i+1}}
}}
\right).
\]
We set $\genmu(\sigma_1,\dots, \sigma_k)=0$ if $k>2$ and we cannot perform any reduction of the form above. For 
$k=2$ we use the ordinary product. $A$ has a natural $\Z_2$-grading by assigning degree $1$ to each angle arrow in $\RQ$. 

\begin{lemma}\label{muhomotopy}
Let $\rho_1,\dots,\rho_k$ be a composable sequence of nonzero angle paths in $\RQ$ with $k>2$ such that
$\rho_i\rho_{i+1}=0$ and $\rho_1\dots\rho_k$ is a contractible cycle on $\psurf{\qpol}$.
For each angle $\beta$ with $\rho_k\beta\ne 0$ in $\Gtl(\qpol)$ we have
\[
 \mu_k(\rho_1,\dots,\rho_k\beta) =(-1)^{|\beta|} \beta
\]
and for each angle $\beta$ with $\beta\rho_1\ne 0$  in $\Gtl(\qpol)$ we have
\[
 \mu_k(\beta\rho_1,\dots,\rho_k) =(-1)^{|\beta|} \beta
\]
All other $\mu_k$ with $k>2$ are zero. 
\end{lemma}
\begin{proof}
This follows easily from Lemma 10.9 in \cite{bocklandt2016noncommutative}. This expression can also be found in Haiden et al. \cite{haiden2014flat} where it is part of the
definition of the $A_\infty$-structure.
\end{proof}
\begin{corollary}\label{subzero}
If $\mu(\rho_1,\dots,\rho_k)$ is nonzero then $\mu(\rho_i,\dots,\rho_j)$ is zero for allow
proper subsequences $\rho_i,\dots,\rho_j$.
\end{corollary}
\begin{proof}
This from the induction hypotesis: if the product 
$\mu(\rho_i,\dots,\rho_j)$ were nonzero
we can reduce $\rho_i,\dots,\rho_j$ to a single angle. 
After reducing the original product, that single angle composes with $\rho_{i-1}$ or $\rho_{j+1}$ to something nonzero.
Which contradicts the fact that the original is nonzero. 
\end{proof}

If $a$ is an arrow in $\qpol$, we can see $a$ as an oriented 1-dimensional submanifold in the punctured surface $\psurf \qpol$ and if we chose a symplectic structure
on $\psurf \qpol$, we can see $a$ as an embedded Lagrangian submanifold. Such a submanifold can be seen as an object in the wrapped Fukaya category of this surface.
For a precise definition of this category we refer to  \cite{abouzaid2013homological}, but its main objects are immersions $\gamma: I \to \psurf \qpol$ where $I=(0,1)$ or $\R/\Z$. These objects are 
called open or closed Lagrangians depending on $I$. For the open Lagrangians, we also demand that the limit points are punctures of the surface.

\begin{proposition}\cite{bocklandt2016noncommutative}
The category $\Der_{\Z_2} \Gtl(\qpol)$ is equivalent to the derived version of the wrapped Fukaya category of the punctured surface  $\psurf \qpol$ in the sense of Abouzaid et Al. \cite{abouzaid2013homological}.
Under this isomorphism $a[0] \in \Der_{\Z_2} \Gtl(\qpol)$ corresponds to the embedded Lagrangian submanifold represented by $a$.
\end{proposition}

A \emph{graded surface} $\vec S=(S,V)$ consists of an oriented surface $S$ equiped with a vector field $V$. For a graded surface we can define the 
notion of a graded Lagrangian submanifold $\cL=(\gamma,\theta)$. 
This is an immersed curve $\gamma: I \to S$ together with a map $\theta: I \to \R$ which specifies the angle between $V_{\gamma(t)}$ and $\frac{d\gamma}{dt}(t)$.
Note that while every open Lagrangian can be given a grading, this is not true for every closed Lagrangian. 
To shift a graded Lagrangian we reverse its orientation and add $\pi$ to the map $\theta$: $\cL[1] = (\gamma(1-t), \theta+\pi)$. 

For a graded surface with boundary we can define its topological Fukaya category. This category is $\Z$-graded instead of $\Z_2$-graded. For its general construction we refer to \cite{haiden2014flat}.
Although the definition above uses a metric to determine the angles, the actuall Fukaya category does not do not depend on the choice of the metric 
because the actual angles are not important only how many multiples of $\pi$ they differ. 

Given a perfect matching on $\qpol$ we can put a vector field on $\psurf \qpol$ in the following way: we fill each polygon with integral curves
that start at the head of the arrow in the perfect matching and run to its tail, as illustrated in the picture. 
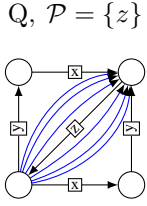
\begin{center}
\begin{tikzpicture}
\begin{scope}[scale=1.5]
\draw [draw=blue,-latex,shorten >=5pt] (0,0) to[out=5, in=265] (1,1) ;
\draw [draw=blue,-latex,shorten >=5pt] (0,0) to[out=85, in=185] (1,1) ;
\draw [draw=blue,-latex,shorten >=5pt] (0,0) to[out=15, in=255] (1,1) ;
\draw [draw=blue,-latex,shorten >=5pt] (0,0) to[out=30, in=240] (1,1) ;
\draw [draw=blue,-latex,shorten >=5pt] (0,0) to[out=60, in=210] (1,1) ;
\draw [draw=blue,-latex,shorten >=5pt] (0,0) to[out=75, in=195] (1,1) ;
\draw (.5,1.5) node{$\qpol,~\PM=\{z\}$};
\draw [-latex,shorten >=5pt] (0,0) to node [rectangle,draw,fill=white,sloped,inner sep=1pt] {{\tiny x}} (1,0);
\draw [-latex,shorten >=5pt] (1,0) to node [rectangle,draw,fill=white,sloped,inner sep=1pt] {{\tiny y}} (1,1);
\draw [-latex,shorten >=5pt] (1,1) to node [rectangle,draw,fill=white,sloped,inner sep=1pt] {{\tiny z}} (0,0);
\draw [-latex,shorten >=5pt] (0,0) to node [rectangle,draw,fill=white,sloped,inner sep=1pt] {{\tiny y}} (0,1);
\draw [-latex,shorten >=5pt] (0,1) to node [rectangle,draw,fill=white,sloped,inner sep=1pt] {{\tiny x}} (1,1);
\draw (0,0) node[circle,draw,fill=white,minimum size=10pt,inner sep=1pt] {{}};
\draw (0,1) node[circle,draw,fill=white,minimum size=10pt,inner sep=1pt] {{}};
\draw (1,1) node[circle,draw,fill=white,minimum size=10pt,inner sep=1pt] {{}};
\draw (1,0) node[circle,draw,fill=white,minimum size=10pt,inner sep=1pt] {{}};
\end{scope}
\end{tikzpicture} 
\end{center}
In this way we get a graded surface $\gsurf{\PM}{\qpol}$. Each arrow not in the perfect matching corresponds to a curve parallel to the vector field, so we can give it a grading by putting $\theta=0$. Each arrow in the perfect matching runs opposite to
the vector field so we can grade it with $\theta=\pi$.

We can also use the perfect matching to put a $\Z$-grading on $\Gtl \qpol$: we put the degree of an angle arrow $-1$ if it arrives in an arrow of the perfect matching and $+1$ otherwise. In this way the degree of a cycle of angle arrows that goes around a face in $\qpol_2$ is $k-2$, where $k$ is the
length of the cycle. In the induction step the order of the multiplication also goes down by $k-2$, so the $A_\infty$-structure is compatible with this grading.
We denote this graded algebra by $\Gtl_\PM(\qpol)$.

\begin{proposition}
The category $\Der_{\Z} \Gtl_\PM(\qpol)$ is equivalent to the derived version of the topological Fukaya category of the graded surface  $\gsurf{\PM}{\qpol}$ in the sense of 
Haiden et Al. \cite{haiden2014flat}. Under this isomorphism $a[0] \in \Der_{\Z} \Gtl_\PM(\qpol)$ corresponds to the embedded Lagrangian submanifold represented by $a$ 
graded by $\theta=0$ if $a\not \in \PM$ and $\theta=\pi$ if $a\in \PM$.
\end{proposition}

\subsection{The B-model}

Given a dimer quiver $\qpol$ we can construct its \emph{Jacobi algebra} $\Jac(\qpol)$. This is the path algebra of the quiver $\qpol$ with relations 
\[
 \Jac(\qpol) := \C \qpol/\<r_a^+-r_a^- | a \in \qpol_1\>
\]
where $r_a^\pm a$ is the unique cycle in $\qpol_2^\pm$ containing $a$. 
These relations state that going back around a cycle to the left of an arrow is the same as going back along the right cycle. 

This algebra has a central element, which is the sum of cycles in $\qpol_2$, one starting in each vertex. 
\[
\ell = \sum_{v \in \qpol_2} c_v \text{ with $h(c_v)=v$ and $c_v \in \qpol_2$} 
\]
Note that the relations in $\Jac(\qpol)$ ensure that this element is central and it is independent of the choices of the $c_v$.

\begin{example}\label{torus1jac}
If we apply this definition to $\qpol$ in example \ref{torus1}, we get $\Jac(\qpol)=\C[X,Y,Z]$ and $\ell=XYZ$.
\end{example}

Under certain conditions, this algebra has very nice properties:
\begin{theorem}\label{consprop}
If $\qpol$ zigzag consistent on a torus then
\begin{enumerate}
 \item $\Jac(\qpol)$ is finitely generated over its center, which is a 3d Gorenstein ring.
 \item $\Jac(\qpol)$ is a 3-Calabi-Yau algebra, which means that it has global dimension 3 and $\Hom^\bullet_{\Jac^e}(\Jac,\Jac^e)\cong \Jac[3]$.
 \item $\Jac(\qpol)$ embeds in $\wJac =\Jac(\qpol)\otimes_{\C[\ell]}\C[\ell,\ell^{-1}]\cong  \Mat_n(\C[X^{\pm 1},Y^{\pm 1},Z^{\pm 1}])$.
\end{enumerate}
\end{theorem}

\begin{remark}\label{historyconsistent}
This result is a summary of many results in the literature. Proofs of these statements for different equivalent notions of consistency can be found in
\cite{mozgovoy2010noncommutative,davison2011consistency,broomhead2012dimer,ishii2010note,bocklandt2012consistency}.
\end{remark}

Given an arrow $a \in \qpol_1$ we set $a^{-1}=r_a^+\ell^{-1}$. This new element satisfies $aa^{-1}=h(a)$ and $a^{-1}a=t(a)$ and we call
it the inverse of the arrow. We also define the inverse of a path by $(a_1\dots a_k)^{-1}=a_k^{-1}\dots a_{1}^{-1}$. 
The algebra $\wJac$ is generated by all arrows of $\qpol$ and their inverses.
Paths consisting of arrows and inverses of arrows are called \emph{weak paths} and $\wJac$ is called the \emph{weak Jacobi algebra}.
If $\Jac$ is graded by a perfect matching $\PM$ we can extend 
this grading to $\wJac$ by putting $\deg_\PM p^{-1}= -\deg_\PM p$.

We will need a few lemmas about the difference between weak and real paths and the existence
of real paths.
\begin{lemma}[Broomhead]\label{zeroonallpm}\cite{broomhead2012dimer}
Let $\qpol$ be a consistent dimer on a torus.
A weak path $p$ in $\wJac$ sits in $\Jac$ if and only if $\deg_{\PM} p\ge 0$ for all perfect matchings.
\end{lemma}

\begin{lemma}[Broomhead]\label{existspath}\cite{broomhead2012dimer}
Let $\qpol$ be a consistent dimer on a torus.
\begin{enumerate}
 \item For each pair of vertices $v,w$ and each homotopy class, there is a path $v\stackrel{p}{\ot}w$ that represents the homotopy class and a perfect matching such that $\deg_{\PM} p= 0$.
 \item For each homology class there is a central element $z$ and a perfect matching $\deg_{\PM} z= 0$ such that $zv$ is a path with that homology class.
 \item
 Two paths represent the same element in $\Jac$ or $\wJac$ if they have the same homotopy class and
 the same degree for at least one perfect matching.
\end{enumerate}
\end{lemma}
The last lemma implies that the algebra $\wJac$ can be identified with $\Mat_n(\C[X^{\pm 1},Y^{\pm 1},Z^{\pm}])$. To do this
fix a vertex $v$ and a perfect matching $\PM$. Choose paths $p_{vw}:v \ot w$ for all other vertices $w$ (and assume $p_{vv}=v$).
The path $p$ is mapped to $X^iY^jZ^kE_{h(p)t(p)}$ where $E_{h(p)t(p)}$ is an elementary matrix, 
$(i,j)$ is the homology class of $p_{vh(a)}pp_{vt(a)}^{-1}$
and $k = \deg_\PM p_{vh(a)}+\deg_\PM p- \deg_\PM p_{vt(a)}^{-1}$.

We can use $\ell$ to turn $\Jac(\qpol)$ into a $\Z_2$-graded Landau-Ginzburg model $\Jac(\qpol,\ell)$ by adding a zeroth multiplication with $\mu_0(1)=\ell$.
We assume that $\Jac(\qpol)$ is $\Z_2$-graded and concentrated in degree $0$. 
The objects in $\Tw \Jac(\qpol,\ell)$ are also known as matrix factorizatons.
A matrix factorization of $\ell$ consists of a pair $(P,d)$ where $P$ is a ($\Z_2$ or $\Z$)-graded finitely generated projective $\Jac(\qpol)$-module and
$d$ is a degree $1$ map such that $d^2=\ell$. Every finitely generated projective $\Jac(\qpol)$-module can be seen as a direct sum of shifts of projective modules
coming from the vertices: $P= v_1\Jac[i_1]\oplus\dots \oplus v_k\Jac[i_k]$ and $d$ corresponds to a matrix with entries of degree $1$. The Maurer-Cartan equation
in this case becomes $\ell - d^2=0$, so $(v_1[i_1]\oplus \dots \oplus v_k[i_k],d)$ is a twisted object.
We will often write the matrix factorization in a diagram
\[
\xymatrix{P_0 \ar@<.5ex>[r]^{d_{10}}&P_1 \ar@<.5ex>[l]^{d_{01}}}
\]
where $P_0$ and $P_1$ are the even and odd part of the matrix projective module.

Every arrow $a \in \qpol_1$ gives rise to a matrix factorization
\[
 M_a := \left( h(a)[1] \oplus t(a), \sm{0&a\\r_a^+&0}\right) =
\xymatrix{t(a)\Jac \ar@<.5ex>[r]^{a}&h(a)\Jac \ar@<.5ex>[l]^{r_a^+}}
\]
If $a$ and $b$ are consecutive arrows in a positive cycle $abu \in \qpol_2^+$ there is a morphism of matrix factorizations $\widehat{ab}$ 
that corresponds to a commutative diagram
\[
\xymatrix{
&t(a)\Jac \ar@<.5ex>[r]^{a}\ar[d]^{-\Id}&h(a)\Jac \ar@<.5ex>[r]^{bu}\ar[d]^{u=\ell(ab)^{-1}}&\\
\ar@<.5ex>[r]^{-b}&h(b)\Jac \ar@<.5ex>[r]^{-ua}&t(b)\Jac 
}
\]
(The minus signs in the bottom row come from the shift $M_b[1]$).
These morphisms generate the endomorphism algebra of $\oplus_{a \in \qpol_1}M_a$ in $\H\Tw_{\Z_2} \Jac(\qpol,\ell)$ and we denote
this algebra by $\mf(\qpol)$.

\begin{theorem}\cite{bocklandt2016noncommutative}
If $\qpol$ is zigzag consistent then $\mf(\qpol)$ and $\Gtl(\polq)$ are $A_\infty$-isomorphic as $A_\infty$-algebras, so
\[
 \Der_{\Z_2} \mf(\qpol) \cong \Der_{\Z_2} \Gtl(\polq).
\]
\end{theorem}

If a perfect matching $\PM$ is specified we can make $\Jac(\qpol)$ $\Z$-graded by giving all arrows in $\PM$ degree $2$ and all 
other arrows degree $0$. In this way $\ell$ has degree $2$. Again we have a matrix factorization $M_a$ for every arrow in $\qpol_1$ and we let
$\mf_\PM(\qpol)$ be the graded endomorphism algebra of $\oplus_{a \in \qpol_1}M_a$ in the category $\Tw_{\Z} \Jac(\qpol,\ell)$.
\begin{theorem}\cite{bocklandt2016noncommutative}
If $\qpol$ is zigzag consistent and $\PM$ a perfect matching then $\mf_\PM(\qpol)$ and $\Gtl_\PM(\polq)$ are quasi-isomorphic as $A_\infty$-algebras, so
\[
 \Der_{\Z} \mf_\PM(\qpol) \cong \Der_{\Z} \Gtl_\PM(\polq).
\]
\end{theorem}

\section{Strings and Bands}\label{sectionbands}

In \cite{haiden2014flat} Haiden et al. described all indecomposable objects in the topological Fukaya category of a graded surface.
They showed that these objects come in two types: strings and bands. The strings correspond to immersed open curves in the surface that run between two punctures, while 
the bands correspond to closed curves. The bands also come equiped with a local system, given by a Jordan matrix that measures the transport around the curve.
In this section we will give a combinatorial description of these objects using the dimer formalism and use this description 
to construct the corresponding matrix factorizations for the $B$-model. We will focus on the band objects, a similar construction can be done for the string objects.
Note that in the $\Z_2$-graded case there are more objects possible and a classification is not known, although there are results in the completed case \cite{burban2010maximal}.

In what follows we assume that $\qpol$ is a consistent dimer on a torus and $\polq$ is its mirror dimer, which is embedded in a surface $S=\psurf{\polq}$.
The dimer  $\qpol$ will be used for the B-model and $\polq$ for the A-model.
\subsection{Definition}
A \emph{garland band} is a sequence $g =(g_0, g_1, \dots, g_{2k-1})$ where $k\ge 1$, $g_{2i} \in \polq_1=\qpol_1$ and $g_{2i+1} \in \polq_2=\qpol_2$.
We also demand that
\begin{itemize}
 \item[B1] the arrows are contained in the cycles around them: $g_{2i} \in g_{2i\pm 1}$,
 \item[B2] consecutive arrows and cycles are different: $g_{i} \ne g_{i+2}$,
\end{itemize}
The indices in the definition are considered in $\Z/{2k\Z}$ and bands are considered upto cyclic shifts by an even offset.
A band is called \emph{primitive} if it there is no nontrivial cyclic permutation that maps $g$ to itself, so $g$ is not the power of a smaller band.

Note that this definition does not depend on whether you are working in $\qpol$ or its mirror, so there is a one to one correspondence between
garland bands in $\qpol$ and its dual $\polq$. 

\subsection{Bands in the A-model}
Given a garland band, we can draw a closed curve on the surface $\psurf{\polq}$ by connecting the centers of the arrows $g_{2i+2}$ and $g_{2i}$ by a line
through the face $g_{2i+1}$. Vice versa if we have a closed curve on $\psurf \polq$ we can isotope it such that it never enters and leaves a face
via the same arrow. After this isotopy we can write down a list of all the arrows and faces this curve meets and this gives us a garland band.
This band is depends uniquely on the isotopy class of the curve in $\psurf \polq$ because the universal cover of $\psurf \polq$ is a tree of faces
glued together along the arrows.

Given a garland $g$ there is a unique path $p_0\dots p_u \in \C \polq$ starting with $p_0=g_0$ and running along all the face cycles.
It is the concatenation of subpaths $p_{i_j}\dots p_{i_{j+1}-1}$ of $g_{2j+1}$ such that $p_{i_j}=g_{2j}$ and $t(p_{i_{j+1}-1}) = h(g_{2j+2})$. 
We will call this path the snake path of $g$. It has the property that two consecutive arrows of a path are either contained in a positive or a negative cycle of $\polq$.
The arrows of $p_i$ for which $p_{i-1}p_{i}$ and $p_{i}p_{i+1}$ are subpaths of different cycles correspond to the arrows in the garland.  
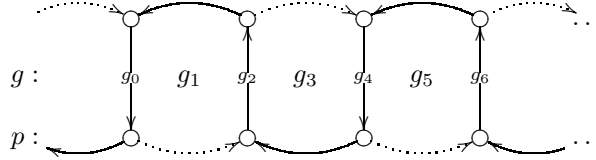
\[
\xymatrix@=.4cm{
\ar@{.>}@/^/[rr]&&\vtx{}\ar[dd]|{g_0}&&\vtx{}\ar@/_/[ll]\ar@{.>}@/^/[rr]&&\vtx{}\ar[dd]|{g_4}&&\vtx{}\ar@/_/[ll]\ar@{.>}@/^/[rr]&&\dots\\
g:&&&g_1&&g_3&&g_5&&&\\
p:&&\vtx{}\ar@/^/[ll]\ar@{.>}@/_/[rr]&&\vtx{}\ar[uu]|{g_2}&&\vtx{}\ar@/^/[ll]\ar@{.>}@/_/[rr]&&\vtx{}\ar[uu]|{g_6}&&\dots \ar@/^/[ll]
}
\]
From the snake path $p=p_0\dots p_r$ and a sequence of invertible square matrices $(\alpha_i)_{0\dots u}$ we can construct a twisted complex
\[
 \left(P =\bigoplus_i p_i[u_i],  \delta= \sum_i  (\alpha_i\otimes\widehat{p_{i-1}p_{i}})^{\pm 1}\right)
\]
where $\widehat{p_{i-1}p_{i}}$ is the angle between the two arrows inside the face of which $p_{i-1}p_{i}$ is a subpath. The exponent of $\widehat{p_{i-1}p_{i}}$ is $\pm 1$ depending on whether this face is in $\polq_2^{\pm}$. It is there to make sure that $\widehat{p_{i-1}p_{i}}^{\pm 1} \in \Gtl(\polq)$, because
in the negative cycles the angle goes in the opposite direction.

The shifts $u_i$ are uniquely determined by $u_0=0$ and the demand that $\deg \delta=1$. If we work with $\Z_2$-gradings all the $u_i$ are zero, 
but if we specified a perfect matching the degrees can only be satisfied when $\sum_i \pm \deg \widehat{p_{i-1}p_{i}}=0$. 
This is precisely when the curve of the garland is graded for the graded surface $\gsurf{\PM}{\polq}$.

The pair $(P,\delta)$ is indeed a twisted complex: if we look at the Maurer-Cartan equation all higher products are zero because 
to apply the reduction step all angles in one face have to be present but this is not possible because the curve of a garland enters and leaves a face by a different arrow. 
Furthermore $\delta^2=0$ because if $\widehat{p_{i-1}p_{i}}^{\pm 1}$ and $\widehat{p_{i}p_{i+1}}^{\pm 1}$ concatenate they must sit in the same cycle and then their product is zero
in $\Gtl(\polq)$.

By applying base changes to the matrices we see that the twisted object only depends on the conjugacy class of the product $\alpha = \prod_i \alpha_i^{\pm 1}$.
We denote the twisted object by $B(g,\alpha)$.

\begin{lemma}
If $\gamma$ is a graded curve in $\gsurf{\PM}{\polq}$ and the corresponding garland is $g$, then $B(g,\alpha)$ corresponds
to the object in the topological Fukaya category associated to (an appropriate shift of) $\gamma$ and $\alpha$ by Haiden et al.  
\end{lemma}
\begin{proof}
This follows from the construction in \cite{haiden2014flat}.
\end{proof}

\subsection{Bands in the B-model}
Now we want to construct the corresponding matrix factorizations of $\Jac(\qpol,\ell)$. The idea is to make a cone over all the matrix factorizations
coming from arrows in the snake path and then simplifying this matrix factorization using a shortening lemma.

\begin{lemma}[Shortening lemma]
Let $(P,d)$ be a matrix factorization and $P_a, P_b$ two graded summands of $P$ such that $d_{ba}:=\Id_b d \Id_a= \phi$ is an isomorphism.
The pair
\[
 (P^{red}, d^{red}) := \left(\frac{P}{P_a\oplus P_b},~ d^{red}_{ij} = d_{ij} - d_{ia}\phi^{-1}d_{bj}\right)
\]
is a matrix factorization quasi-isomorphic to $(P,d)$
\end{lemma}
\begin{proof}
We need to check that for $i,j\not \in \{a,b\}$
\se{
(d^{red})^2_{ij} 
&=\sum_{k\ne a,b} (d_{ik} - d_{ia}\phi^{-1}d_{bk})(d_{kj} - d_{ka}\phi^{-1}d_{bj})\\
&=\sum_{k\ne a,b} (d_{ik}d_{kj} - d_{ia}\phi^{-1}d_{bk}d_{kj} - d_{ik}d_{ka}\phi^{-1}d_{bj} +  d_{ia}\phi^{-1}d_{bk}d_{ka}\phi^{-1}d_{bj})\\
&=(\sum_{k\ne a,b} d_{ik}d_{kj}) - d_{ia}\phi^{-1}(d^2_{bj}-d_{ba}d_{aj}-d_{bb}d_{bj})  \\
&\phantom{=}- (d^2_{ia}-d_{ia}d_{aa}-d_{ib}d_{ba})\phi^{-1}d_{bj}+ d_{ia}\phi^{-1}(d^2_{ab}-d_{ba}d_{aa}-d_{bb}d_{ba})\phi^{-1}d_{bj})\\
&= (\sum_{k\ne a,b} d_{ik}d_{kj}) + d_{ia}d_{aj} + d_{ib}d_{bj} + 0 = d^2_{ij}.
}
In the calculation we used the fact that $d^2_{uv}=0$ if $u\ne v$ and $d_{uv}=0$ if $u=v$.

The quasi-isomorphism $\psi: P_{red}\to P$ is given by
\[
 \psi_{ij} = \begin{cases}
     \delta_{ij}\Id_i &\text{if }i \ne a,b\\
     -\phi^{-1}d_{bj} &\text{if }i = a\\
     0 &\text{if }i = b
             \end{cases}
\text{ and }
 \psi^{-1}_{ij} = \begin{cases}
     \delta_{ij}\Id_i &\text{if }j \ne a,b\\
     0 &\text{if }j = a\\
     -d_{ia}\phi^{-1} &\text{if }j = b
     \end{cases}
\]
A straightforward calculation shows that these maps are indeed quasi-inverses.
\end{proof}
The shortening lemma has some easy consequences.
\begin{lemma}\label{subcycle}
If $a_1\dots a_k$ is a positive cycle in $\qpol_2^+$ then
\[
 M_{a_1\dots a_l} := \left( h(a_1)[1] \oplus t(a), \sm{0&a_1\dots a_l\\a_{l+1}\dots a_k&0}\right) =
\xymatrix{t(a_l)\Jac \ar@<.5ex>[r]^{a_1\dots a_l}&h(a)\Jac \ar@<.5ex>[l]^{a_{l+1}\dots a_k}}
\]
is equivalent to the repeated cone of the morphisms $M_{a_1}\stackrel{\widehat{a_1a_2}}{\to}M_{a_2}\dots M_{a_l}$. 
\end{lemma}
\begin{proof}
For $l=2$ we just apply the shortening lemma:
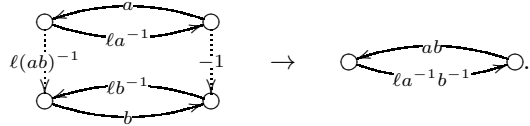
\[
\vcenter{ 
 \xymatrix@C=2cm{\vtx{}\ar@/_/[r]|{\ell a^{-1}}&\vtx{}\ar@/_/[l]|{a}\\
 \vtx{}\ar@/_/[r]|{b}\ar@{<.}[u]|{\ell (ab)^{-1}}&\vtx{}\ar@/_/[l]|{\ell b^{-1}}\ar@{<.}[u]|{-1}\\
}}
\hspace{.5cm}\to\hspace{.5cm}
\vcenter{ 
 \xymatrix@C=2cm{\vtx{}\ar@/_/[r]|{\ell a^{-1}b^{-1}}&\vtx{}\ar@/_/[l]|{ab}}}.
\]
Under the quasi-isomorphism $\Cone(\widehat{ab})\cong M_{ab}$ the morphism
$\widehat{bc}: M_c \to \Cone(\widehat{ab})$ becomes
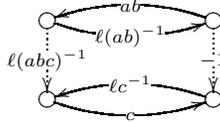
\[
\vcenter{ 
 \xymatrix@C=2cm{\vtx{}\ar@/_/[r]|{\ell (ab)^{-1}}&\vtx{}\ar@/_/[l]|{ab}\\
 \vtx{}\ar@/_/[r]|{c}\ar@{<.}[u]|{\ell (abc)^{-1}}&\vtx{}\ar@/_/[l]|{\ell c^{-1}}\ar@{<.}[u]|{-1}
}}
\]
and hence the proof can be finished by induction.
\end{proof}

If we apply a similar reasoning to more complicated cones, this leads to the following construction:
consider a band $g =(g_0,g_1,\dots,g_{2k})$ in $\qpol$ and assume that $g_1$ is a cycle in $\qpol_2^+$. 
This gives a picture that looks like this:
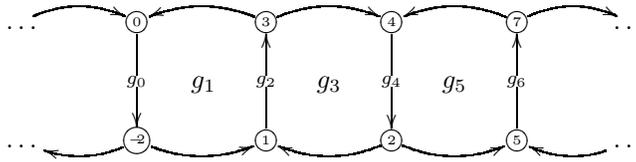
\[
\xymatrix@=.4cm{
\dots\ar@/^/[rr]&&\vtx{0}\ar[dd]|{g_0}&&\vtx{3}\ar@/_/[ll]\ar@/^/[rr]&&\vtx{4}\ar[dd]|{g_4}&&\vtx{7}\ar@/_/[ll]\ar@/^/[rr]&&\dots\\
&&&g_1&&g_3&&g_5&&&\\
\dots&&\vtx{-\!\!2}\ar@/^/[ll]\ar@/_/[rr]&&\vtx{1}\ar[uu]|{g_2}&&\vtx{2}\ar@/^/[ll]\ar@/_/[rr]&&\vtx{5}\ar[uu]|{g_6}&&\dots\ar@/^/[ll]
}
\]
Let $P$ be the direct sum of all the numbered vertices, where the even/odd numbered have even/odd degree.
The map $d:P\to P$ consists of $2$ parts $d=d_a+d_b$, which can be written as a signed sum of paths. 
The first part $d_a$ is the sum of the subpaths of 
$g_i$ that connect the vertices $i$ and $i-1$ (in both directions and only for odd $i$). 
The second part is the sum of all the upper arcs minus the sum off all lower arcs in the picture.
It is easy to check that both $(P,d)$ and $(P,d_a)$ are matrix factorizations.
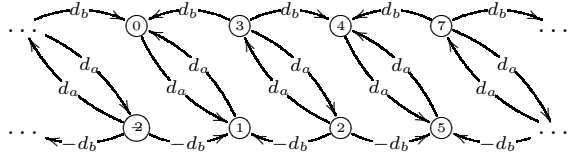
\[
\xymatrix@=.4cm{
\dots\ar@/^/[ddrr]|{d_a}\ar@/^/[rr]|{d_b}&&\vtx{0}\ar@/_/[ddrr]|{d_a}&&\vtx{3}\ar@/^/[ddrr]|{d_a}\ar@/_/[ll]|{d_b}\ar@/^/[rr]|{d_b}&&\vtx{4}\ar@/_/[ddrr]|{d_a}&&\vtx{7}\ar@/_/[ll]|{d_b}\ar@/^/[rr]|{d_b}\ar@/^/[ddrr]|{d_a}&&\dots\\
&&&&&&&&&&\\
\dots&&\vtx{-\!\!\!2}\ar@/^/[uull]|{d_a}\ar@/^/[ll]|{-d_b}\ar@/_/[rr]|{-d_b}&&\vtx{1}\ar@/_/[uull]|{d_a}&&\vtx{2}\ar@/^/[uull]|{d_a}\ar@/^/[ll]|{-d_b}\ar@/_/[rr]|{-d_b}&&\vtx{5}\ar@/_/[uull]|{d_a}&&\dots\ar@/^/[ll]|{-d_b}\ar@/^/[uull]|{d_a}
}
\]

More precisely, $(P,d_a)$ can be seen as a direct sum of $k$ matrix factorzations, one for
each cycle in the band. This matrix factorization can be split as a direct sum of two subfactorizations, the one from the positive cycles and the one from the negative cycles.
The map $d_b$ can be interpreted as a morphism between those two and 
$M(g,1):=(P,d_a+d_b)$ is the cone of this map. 

\begin{proposition}
Under the equivalence $\Der_{\Z_2} \mf(\qpol)\cong \Der_{\Z_2} \Gtl(\polq)$
the matrix factorization $M(g,1)$ corresponds to the object $B(g,1)$.
\end{proposition}
\begin{proof}
By construction the equivalence commutes with taking cones. Therefore, 
to find an object equivalent to $B(g,1)$,
we need to perform the cone construction on the matrix factorizations of the arrows in 
the snake path $p=p_1\dots p_k$ in the mirror dimer according to the morphisms $\widehat{p_ip_{i+1}}$.

The arrows of the snake path $p$ in the mirror dimer do not form a path in $\qpol$ but 
something like this
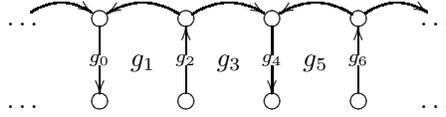
\[
\xymatrix@=.2cm{
\dots\ar@/^/[rr]&&\vtx{}\ar[dd]|{g_0}&&\vtx{}\ar@/_/[ll]\ar@/^/[rr]&&\vtx{}\ar[dd]|{g_4}&&\vtx{}\ar@/_/[ll]\ar@/^/[rr]&&\dots\\
&&&g_1&&g_3&&g_5&&&\\
\dots&&\vtx{}&&\vtx{}\ar[uu]|{g_2}&&\vtx{}&&\vtx{}\ar[uu]|{g_6}&&\dots
}
\]
At the arrows $p_j=g_i$ we have two morphisms leaving or arriving. At all other arrows
there is a morphism arriving and one leaving. To perform the cone operation we proceed in two steps. 
First we do all the cones over morphisms $\widehat{p_ip_{i+1}}$ for which $p_i$ is not one of the $g_j$
or $\widehat{p_ip_{i+1}}$ is contained in a positive cycle.

Repeated application of lemma \ref{subcycle} this shows that the cone after the first step is gives $(P,d_a)$.
The remainder consist of morphisms between consecutive summands of $(P,d_a)$ which are of the form
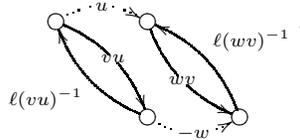
\[
\xymatrix@=.4cm{
\vtx{}\ar@/^/[ddrr]|{vu}\ar@{.>}@/^/[rr]|{u}&&\vtx{}\ar@/_/[ddrr]|{wv}&&\\
&&&&\\
&&\vtx{}\ar@/^/[uull]^{\ell (vu)^{-1}}\ar@{.>}@/_/[rr]|{-w}&&\vtx{}\ar@/_/[uull]_{\ell (wv)^{-1}}
}.
\]
These are precisely the components of $d_b$.
\end{proof}

\begin{remark}
If we have an invertible $n\times n$-matrix $\alpha$ we can also construct a matrix factorization for $B(g,\alpha)$:
\[
 M(g,\alpha) := (P^{\oplus n}, d_a^{\oplus n} + d_b^\alpha)
\]
where $d_b^\alpha$ is equal to $d_b^{\oplus n}$ except for the two paths that go from vertex $-1$ and $-2$ to vertex $0$ and $1$. Those paths are
tensored by $\alpha$ to ensure that the parallel transport around the cycle equals $\alpha$.
\end{remark}

\section{Intermezzo I: Spider graphs and Ribbon graphs}\label{sectionspider}

Spider graphs and ribbon graphs are two combinatorial objects that occur as limit situations of surfaces.

A \emph{graph with legs} $\Lambda$ is a generalization of a graph where we allow edges that are only connected on one side to a node and
the other side goes to infinity. These edges are called the legs or external edges, while the other edges are called the internal edges.
We split the set of edges as $\Lambda_1 = \Lambda_1^{int}\cup \Lambda_1^{ext}$. 

A graph with legs is called weighted
if there is a map $w: \Lambda_1^{int} \to \R_{>0}$. These weights can be thought of as the lengths of the edges. The legs have infinite length, so they don't need
weights. By gluing together line segments of the appropriate length, we can construct a metric space $|\Lambda|$ that is a realization of the graph.

\subsection{Ribbon graphs}
A graph with legs $\Gamma$ becomes a \emph{ribbon graph} \cite{mulase1998ribbon} if we assign to each node a cyclic order on the incoming edges. 
From a ribbon graph we can construct a surface with boundary by substituting each node of valency $n$ by an $n$-gon and
each edge by a strip. We glue the ends of the strips to the polygon of the node in the specified cyclic order. The result is an oriented surface with boundary $\rib \Gamma$ 
and there is a retraction $\pi: \rib \Gamma \to |\Gamma|$, defined up to homotopy.

If $\Gamma$ is a connected ribbon graph without legs we say that it is of type $\Sigma_{g,n}$ if $\rib \Gamma$ is a surface
with genus $g$ and $n$ boundary components.

Vice versa if $\Gamma$ is any graph embedded in an orientable surface, we can give it the structure of a ribbon graph by
using the anticlockwise cyclic order on the surface in each node. The space $\rib \Gamma$ will be homeomorphic to a neighborhood of the graph in the surface.

\erbij{
If $e$ is an internal edge of the ribbon graph $\Gamma$ that is not a loop, we can define its contraction $C_e\Gamma$ as the ribbon graph where $e$ is removed and 
the two nodes of $e$ are identified. The cyclic order in the new node is obtained by gluing
together the linear orders we get in both nodes after deleting $e$ (e.g $x<e<y<$ and $u<e<v<$ gives $y<x<v<u<$). }

Given a dimer quiver we can construct a ribbon graph in the following way. As a graph it is dual to the quiver, so its nodes correspond to the cycles in $\qpol_2$
and its edges correspond to the arrows. We draw this graph on $\psurf{\qpol}$ by putting a node in each face and connecting two node by an edge perpendicular to 
the corresponding arrow. Note that this graph is by partite because the $\qpol_2 =\qpol_2^+\cup \qpol_2^-$. We can also go in the opposite
directing and construct a dimer quiver from a bipartite ribbon graph without legs
\footnote{This is the way how dimer models are often introduced in the literature}.

\subsection{Spider graphs}

A graph with legs becomes a \emph{spider graph} if we assign to each node $v$ a genus $g_v$, which is a nonnegative integer. 
The total genus of a spider graph is the sum of the genera of the nodes plus the genus of the graph itself
\[
 g(\Lambda) = \sum_{v \in \Lambda_0} g_v + \# \Lambda_1^{int} - \# \Lambda_0 + 1.  
\]

From a spider graph we can construct a surface by substituting each node of genus $g$ and valency $n$ by a surface
with genus $g$ and $n$ boundary components, substituting each edge by a cylinder and gluing these surfaces and cylinders together
along boundary components. The resulting surface has the same genus as the spider graph 
and has a hole for each leg of the spider graph. It is called the tubular surface and it comes with a projection map $\pi: \tube{\Lambda}\to \Lambda$
which is its boundary components correspond defined up to homotopy. We say that $\Lambda$ is of type $\Sigma_{g,n}$ if $\tube{\Lambda}$ is a surface with genus $g$ and $n$ holes.

We can also go in the opposite direction. Let $\dot S$ be a surface of genus $g$ and $n$ holes. A slicing $\slicing$ is a set of homotopically different nonintersecting curves on the surface which
includes the $n$ simple curves that go around the holes. The spider graph $\Lambda_\slicing$ has an edge for each curve and a node for each connected component
of the complement of the curves, except the $n$ components that wrap around the holes. The genus of the node is the genus of the corresponding connected component.
It is clear from this construction that $\tube{\Lambda_\slicing}$ is isomorphic to $\dot S$.

\erbij{If $e$ is an internal edge of the spider graph $\Lambda$, we can define its contraction $C_e\Lambda$ as the spider graph where $e$ is removed and 
the two nodes of $e$ are identified. The genus of this new node is the sum of the genera of the two original nodes if they are different, or
one more then the genus of the original node if $e$ was a loop.
}

\section{Intermezzo II: Tropical and toric geometry}\label{sectiontropical}

In this intermezzo we review some well-known facts from tropical geometry and toric geometry and tie this to the notion of spider graphs.
For more information on the former we refer to \cite{mikhalkin2004amoebas}, while for the latter we refer to \cite{fulton1993introduction}.

\subsection{Tropical curves}
A tropical polynomial in two variables is an expression of the form $f = \min_{i\in I}\{a_iX+b_iY+c_i\}$, where $a_i,b_i,c_i \in \Z$. 
The tropical curve $\Trop(f)$ is the set of all point $(X,Y) \in \R^2$ for which the map $f:\R^2 \to f(X,Y)$ is not smooth, or
equivalently the points where the minimum is reached by at least two of the linear functions. In general this set is a union of line segments
and half lines.

To each tropical polynomial we can assign its Newton polygon $\NP(f)\subset \R^2$. This is the convex hull of the points $(a_i,b_i)$.
For each point $(x,y)$ in the tropical curve we define its support as the convex hull of all $(a_i,b_i)$ for which $f(x,y)= a_ix+b_iy+c_i$.
If $(x,y)$ lies on an edge then $\supp(x,y)$ will be a line segment and if $(x,y)$ is a node then $\supp(x,y)$ will be a subpolygon of $\cP(f)$.
The set of all supports will form a subdivision of the Newton polygon $\cF(f)=\{\supp(x,y) | (x,y) \in \Trop(f)\}$. 

From a tropical polynomial, we construct a spider graph $\Lambda(f)$ by 
letting the line segments of $\Trop(f)$ be the internal edges and the half lines be the legs. 
We turn an edge into a $k$-fold edge if there are $k-1$ lattice points in the interior of the support in $\cF(f)$.
The nodes of the spider graph are the endpoints of the line segments. 
We set the genus of a node equal to the number of interior points in the support of the node.

We can also assign weights to all the internal edges that correspond to the affine lengths of
the line segments of the tropical curve. I.e. if $e$ is an edge between the points
$(x_1,y_1)$ and $(x_2,y_2)$ then its affine length is equal to $W_e :=\min_{n,m\in \Z} |(x_1-x_2)m+(y_1-y_2)n|$. 
In other words the vector that connects the two points  is $W_e$ times an elementary affine
line segment. Note that this makes sense because the vector always has a rational direction
because the coefficients $a_i,b_i$ of the tropical polynomial are integers.

A tropical curve is called smooth if the subdivision $\cF(f)$ is as fine as possible: it consists of elementary triangles. For the spider graph
this means that all nodes of the curve are trivalent and have genus zero. Just as in complex geometry a generic tropical curve is smooth.

\subsection{Amoebae}
Given a tropical polynomial we can consider a one parameter family of ordinary polynomials $f_t := \sum_i \alpha_i t^{c_i}X^{a_i}Y^{b_i}$
where the $\alpha_i \in \C$ are some constants. For generic $\alpha_i$ and $t$ this polynomial defines a smooth Riemann surface $$S_t=\{(x,y) \in (\C^*)^2 | f_t(x,y)=0\}.$$
We can look at the image of the map $\Am_t : S_t \to \R^2: (x,y) \mapsto \frac 1{\log |t|}(\log |x|, \log |x|)$. 
This image is called the amoeba of $f_t$ because it looks like a blob with holes and tentacles. If we let $t$ go to infinity the interior of the amoeba becomes thinner and thinner until we end up with a graph. The main theorem in tropical geometry states that $\lim_{t \to \infty }\Am_t = \Trop f$ \cite{mikhalkin2004amoebas}.

If we fix a point $p$ on the tropical curve and a small $\eps>0$ we can look at the preimage of $\Am_t \cap B(p,\eps)$ for $t>\!\!>0$. If $p$ is on the inside
of a $k$-fold edge of the tropical curve, this preimage will be a disjoint union of $k$ cylinders.
If $p$ is a node with valency $v$ and genus $g$ the preimage will be a surface of genus $g$ with $v$ punctures. Therefore $S_t$ can be seen
as the tubular surface of $\Lambda(f)$.

\begin{lemma}\label{genustropical}
Let $f$ be a tropical polynomial and $\Lambda(f)$ be its spider graph 
\begin{enumerate}
 \item The genus of $\Lambda(f)$ equals the number of internal lattice points of the Newton polygon $\NP(f)$.
 \item The number of legs of $\Lambda(f)$ equals the number of boundary lattice points of the Newton polygon $\NP(f)$.
\end{enumerate}
\end{lemma}
\begin{proof}
It is well known that if $S\subset \C^*\times \C^*$  is a smooth curve defined by a polynomial with Newton polygon $\NP$, then the genus
of $S$ is the number of internal lattice points and number of punctures in $S$ is the number of boundary lattice points.
\end{proof}

We end this section with an example. Below we give the Newton polygon, the tropical curve and
the spider graph for the tropical polynomial
\begin{center}
$f = \max \{2x,x+y, -x+y, -x+1, -y, 2x-y, y\}$\\
$f_t = X^2+XY+tX^{-1}Y +tX^{-1}+Y^{-1}+X^2Y^{-1}+Y$\\
~\\
\resizebox{!}{2cm}
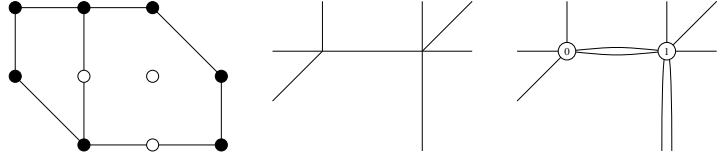
{\begin{tikzpicture} 
\draw (1,1) -- (2,0); 
\draw (0,1) -- (1,1); 
\draw (-1,0) -- (-1,1); 
\draw (0,-1) -- (-1,0); 
\draw (2,-1) -- (0,-1); 
\draw (2,0) -- (2,-1); 
\draw (-1,1) -- (0,1); 
\draw (0,-1) -- (0,1); 
\draw (2,0) node[circle,draw,fill=black,minimum size=5pt,inner sep=1pt] {};
\draw (1,1) node[circle,draw,fill=black,minimum size=5pt,inner sep=1pt] {};
\draw (-1,1) node[circle,draw,fill=black,minimum size=5pt,inner sep=1pt] {};
\draw (-1,0) node[circle,draw,fill=black,minimum size=5pt,inner sep=1pt] {};
\draw (0,-1) node[circle,draw,fill=black,minimum size=5pt,inner sep=1pt] {};
\draw (2,-1) node[circle,draw,fill=black,minimum size=5pt,inner sep=1pt] {};
\draw (0,1) node[circle,draw,fill=black,minimum size=5pt,inner sep=1pt] {};
\draw (0,0) node[circle,draw,fill=white,minimum size=5pt,inner sep=1pt] {};
\draw (1,0) node[circle,draw,fill=white,minimum size=5pt,inner sep=1pt] {};
\draw (1,-1) node[circle,draw,fill=white,minimum size=5pt,inner sep=1pt] {};
\end{tikzpicture}}
\hspace{.5cm}\resizebox{!}{2cm}{\begin{tikzpicture} 
\draw (0,0) -- (1,1); 
\draw (0,0) -- (0,1); 
\draw (-2,0) -- (-3,0); 
\draw (-2,0) -- (-3,-1); 
\draw (0,0) --(0,-2); 
\draw (0,0) -- (1,0); 
\draw (-2,0) -- (-2,1); 
\draw (0,0) --(-2,0); 
\end{tikzpicture}}
\hspace{.5cm}\resizebox{!}{2cm}{\begin{tikzpicture} 
\draw (0,0) -- (1,1); 
\draw (0,0) -- (0,1); 
\draw (-2,0) -- (-3,0); 
\draw (-2,0) -- (-3,-1); 
\draw (0,0) .. controls (1/30,0) and (1/10,0) .. (1/10,-2); 
\draw (0,0) .. controls (-1/30,0) and (-1/10,0) .. (-1/10,-2); 
\draw (0,0) -- (1,0); 
\draw (-2,0) -- (-2,1); 
\draw (0,0) .. controls (-1,1/10) .. (-2,0); 
\draw (0,0) .. controls (-1,-1/10) .. (-2,0); 
\draw (0,0) node[circle,draw,fill=white,minimum size=10pt,inner sep=1pt] {{\tiny1}};
\draw (-2,0) node[circle,draw,fill=white,minimum size=10pt,inner sep=1pt] {{\tiny 0}};
\end{tikzpicture}}\\
\end{center}

\subsection{Toric varieties}

From a tropical curve $f = \min_{i \in I}\{a_iX+b_iY+c_i\}$ we can also construct a $3$-dimensional toric variety $\VV_f$. The fan of this variety can be constructed by taking
a cone over the Newton polygon $\NP(f)$ and look at the subdivision of this cone induced
by the subdivision $\cF(f)$.

Instead of using the fan we can also construct $\VV_f$ as follows.
First we define the 3d polytopes $\PT_{f,r} := \{ (x,y,z) \in \R^3| a_ix+b_iy+z \ge -r c_i\}$. These polytopes are noncompact because we can make $z$ arbitrarily large.
If we don't specify $r$ we assume it is equal to $1$: $\PT_f := \PT_{f,1}$.

We can partition $\PT_{f,r}$ into subsets of points for which the same
inequalities $a_ix+b_iy +z\ge -r c_i$ are equalities. These subsets are called the faces. 
For each face $\sigma$ we define the support $I_{\sigma}$ as the set of indices which give equalities.
\[
 I_{\sigma} := \{ i \in I | \forall (x,y,z) \in \sigma: a_ix+b_iy +z= -r c_i\}.
\]
For all $r>0$ 
$\PT_{f,r}$ has the same number of faces because $\PT_{f,r}=r \PT_{f,1}$. 
On the other hand $\PL_{f,0}=\lim_{r \to 0}\PT_{f,r}$ so $\PL_{f,0}$ is a cone and its faces correspond to the noncompact faces of $\PT_{f,1}$ because
all other's are shrunken to a point.

Because $\PT_f$ is a $3$-dimensional polytope there are $2,1$ and $0$-dimensional faces. The $1$-skeleton of the polytope is the union of all $1$ and $0$-dimensional faces.
The polytope projects onto the $X,Y$-plane and the image of the $1$-skeleton of the polytope $\PT_{f}$ is precisely the tropical curve $\Trop_f$.

Now consider the graded ring
\[
R_f := \bigoplus_{r=0}^\infty \C\{X^uY^vZ^w| (u,v,w) \in \PT_{f,r} \cap \Z^3\}.
\]
The toric variety of $f$ is defined as 
\[
 \VV_f := \Proj R_f.
\]
It is a projective variety over the ring $(R_f)_0 = \C[X^uY^vZ^w| a_iu+b_iv+w\ge 0 ]$. 
The $\Z^3$ grading on $R_f$ coming from the exponents of $X,Y,Z$ gives rise to a $\C^{*3}$-action on $\VV_f$ and this turns $\VV_f$ into a toric variety.

The geometry of the polytope $\PT_f$ and the variety $\VV_f$ are closely related.
Every $k$-dimensional face $\sigma$ will correspond to a torus orbit $\cO_\sigma \subset \VV_f$ which
consists of the points for which the functions $X^uY^vZ^w$ are nonzero if and only if $(u,v,w) \in \sigma$.
Moreover, if we restrict the action to the real torus $U_1^3 \subset \C^{*3}$, the quotient $\VV_f/U_1^3$ can be identified with the polytope 
$\PT_{f}$ such that each face of the polytope is the image of the corresponding orbit. The identification $\VV_f/U_1^3\to \PT_{f}$ is however
not unique: it depends on a moment map and hence the choice of a symplectic structure on $\VV_f$.

If we look at the tropical curve $\Trop f$, each node $p$ will correspond to a torus fixed point in $\VV_f$. All torus orbits that have this
fixed point in its closure form an open subset $\UU_p \subset \VV_f$. This subset is also a toric variety corresponding to the tropical polynomial containing
only the terms in the support of $p$: $f_p = \min_{i \in I_p}\{a_iX+b_iY+c_i\}$. The tropical curve corresponding to $f_p$ is obtained by looking only at a neighborhood of $p$ and extending all edges to infinity. Finally, $\VV_f$ is a smooth variety if and only if $\Trop(f)$ is a smooth tropical curve.

\begin{remark}
We can also define tropical curves for which the coefficients $c_i$ are in $\QQ$ instead of $\Z$.
The curve $S_t$ will still be uniquely defined for $t \in \R_{\>0}$ and we can define the limit
of the Amoeba-map for these $t$ and observe that it is equal to the tropical curve.

On the other hand in the construction of toric variety, it is easy to see that 
if $f = \min_{i \in I}\{a_iX+b_iY+c_i\}$  and  $f' = \min_{i \in I}\{a_iX+b_iY+kc_i\}$
where $k\in \N$ then $\Proj R_f$ and $\Proj R_{f'}$ are related by a veronese map and
hence isomorphic varieties, so it also makes sense to define $\VV_f$ for tropical polynomials
with rational $c$-coefficients.
\end{remark}

\section{Moduli spaces}\label{sectionmoduli}

One main direction of research in mirror symmetry is the idea that two mirror manifolds $M,M^\vee$ should be connected by an SYZ-fibration.
Loosely speaking this means that there is a diagram
\[
 \xymatrix{M\ar[dr]^\pi&&M'\ar[dl]_{\pi'}\\&B&}
\]
such that $B$ is an $n$-dimensional space and the generic fibers of the projection maps $\pi,\pi'$ are $n$-dimensional Lagrangian tori.
The space $B$ is called the base and in general it is not smooth but in many cases it has the structure of a tropical variety.

\erbij{
On the symplectic side this base can be seen as a moduli space that classifies nonintersecting Lagrangian submanifolds, while on the complex side
it could be seen as a quotient space that classifies orbits of points under some $U_1^n$-action.
}

In our setting the base should be a tropical curve and on the A-side the fibers are circles. On the B-side 
the mirror is a noncommutative Landau-Ginzburg model $(\Jac,\ell)$ and the points of $M'$ should be interpreted as representations of the algebra $\Jac$. These representations give rise to matrix factorizations, so we get a moduli space of matrix factorizations. 
In this section we will work out both sides in detail.

\subsection{The A-model: quadratic differentials}

Let $\bar S$ be a compact smooth Riemann surface of genus $g$ and $m_1,\dots, m_n \in \bar S$ an ordered set of $n$ distinct marked points.
Removing the marked points we get a punctured Riemann surface $\dot S = \bar S \setminus \{m_1,\dots, m_n\}$. We will assume that $\dot S$ has at least 3 punctures, to ensure 
that $\dot S$ is a Riemann surface with negative curvature.

A quadratic differential \cite{strebel1984quadratic} on $\dot S$ is a holomorphic section of the square of the canonical bundle: $\Phi \in \Gamma(K_{\bar S}^{\otimes 2}, \dot S)$. In local coordinates we can write $\Phi= f(z)dz^2$ where $f(z)$ is only allowed to have poles at the marked points.
If $\Phi$ is nonzero in a point $p \in \dot S$ we can find a holomorphic coordinate 
such that around $p$ $\Phi= dw^2$ with $w(p)=0$.
In all other points $q \in \bar S$ we can write $\Phi=aw^k dw^2$ where $k$ is the order of the zero/pole of $f$ at $q$ and a is a constant.
If $k\ne -2$ then we can choose $a=1$ but if $k=-2$ we cannot get rid of the coeficient $a$ by base change. The coefficient $a$ is called the residue.

A tangent vector $v \in T_p \dot S$ is called horizontal if $\Phi(v)\ge 0$. Note that in general
$\Phi(v)$ is a complex number, so in a generic point $p$ this condition singles out a one-dimensional real subspace 
in $T_p \dot S$. This means that $\Phi$ defines a foliation on $\dot S$, which we call the horizontal foliation.
Generically, a leave of the foliation can either be a closed curve or a curve that connects two punctures. There are also special leaves: the zeros of $\Phi$ and
the leaves that have a zero of $\Phi$ in their limit. 

In the neighborhood of a regular point $p$ the horizontal leaves are all parallel: they look like the horizontal lines in the $w$-complex plane.
If $p$ is a zero of order $k>0$ the picture of the trajectories has a symmetry of order $k+2$ because a base change $w = \zeta w'$ with $\zeta^{k+2}=1$ will leave 
the differential invariant. The local picture of horizontal leaves looks like this:

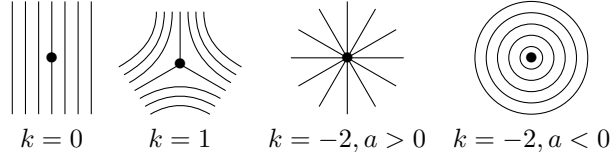
\begin{center}
\begin{tabular}{cccc}
\begin{tikzpicture}[scale=.5]
    \draw (0,0) coordinate (a_1) node {$\bullet$}  -- (0,0);
    \draw (0,-1.5) -- (0,1.5);
 \foreach \i in {1,...,3}
{
    \draw (.35*\i,-1.5) -- (.35*\i,1.5);
    \draw (-.35*\i,-1.5)-- (-.35*\i,1.5);

    }
\end{tikzpicture}
&
\begin{tikzpicture}[scale=.35]
    \draw (0,-.5) coordinate (a_1) node {$\bullet$}  -- (0,1.5);
    \draw (0,-.5) -- (-1.73,-1.5);
    \draw (0,-.5) -- (1.73,-1.5);
 \foreach \i in {1,...,3}
{
\draw (-1.73-.2*\i,-1.5+.3*\i) to[out=30,in=270] (-.3*\i,1.5);
\draw (1.73+.2*\i,-1.5+.3*\i) to[out=150,in=270] (.3*\i,1.5);
\draw (-1.73+.2*\i,-1.5-.3*\i) to[out=30,in=150]  (1.73-.2*\i,-1.5-.3*\i);
}
\end{tikzpicture}
&
\begin{tikzpicture}[scale=.5]
    \draw (0,0) coordinate (a_1) node {$\bullet$}  -- (0,0);
 \foreach \i in {1,...,12}
{
    \draw (0,0) -- (30*\i:1.5);
    }
\end{tikzpicture}
&
\begin{tikzpicture}[scale=.5]
    \draw (0,0) coordinate (a_1) node {$\bullet$}  -- (0,0);
 \foreach \i in {1,...,5}
{
    \draw (.3*\i,0) arc (0:360:.3*\i);
    }
\end{tikzpicture}
\\
$k=0$&$k=1$&$k=-2, a>0$&$k=-2, a<0$
\end{tabular}
\end{center}
In a zero of order $k$ there are $k+2$ special leaves and all other leaves bend away from the zero. 
If we look at poles, the most interesting case is $\Phi = a \frac{dw^2}{w^2}$.
If $a$ is a positive real number the leaves will be rays emanating from the pole. If $a$ is negative the leaves will be circles
around the poles and if $a$ is complex they will be spirals around the poles. Poles of other orders will always have leaves that emanate from the pole.

The special leaves form a graph $\SG(\Phi)$ on the surface $S$, which has internal edges corresponding to trajectories between
two zeros and external edges corresponding to trajectories between a zero and a pole.
The complement of the graph consists of connected components that contain parallel trajectories. In each component all trajectories are isotopic and we have
two possibilities: the leaves are open curves or the leaves are circles. We call these components strips and cylinders.

Each strip can be assigned a width. To do this we define lengths for any curve $\gamma$
\[
 \ell(\gamma) =\int_0^1 \sqrt{|\Phi(\frac {d\gamma}{dt})|} dt ~~~(*)
\]
The width of the strip is the minimal length of a curve that crosses it. The cylinders can be given a width and a circumference, which is the length of a leaf.
Note that cylinders can have an infinite width, if they wrap around one of the punctures. This is not possible for the strips. 

The leaf space of the foliation $\LH(\Phi)$ can be seen as a weighted graph with legs. Its nodes correspond to the connected components of $\SG(\Phi)$ and 
the strips and cylinders are the edges. The weigths are the widths of the strips and cylinders. 
We will now look at two special cases: when there are only strips and when there are only cylinders.
The former is called the ribbon case, while the latter is called the spider case.
\begin{itemize}
 \item 
In the ribbon case the leaves in each strip correspond to an isotopy class of a curve that connects two punctures.
If we draw one of the leaves in each strip and we get a graph on the surface $\bar S$ that splits the surface into polygons.

The leaf space can be identified with the dual graph of this split and hence it has the structure of a ribbon graph. We can assign weights
to this ribbon graph by assigning the widths of a strips to its edge. Note that there are no external edges.
\item
In the spider case each cylinder corresponds to an isotopy class of closed curves. The leaf space has a structure of a spider graph, where
the number of legs equals the number of punctures. To each node we can also assign a genus, which is the genus of a small neighborhood of 
its the connected component in $\SG(\Phi) \subset S$. Finally we can also give weights to the edges which correspond to the widths of the cylinders.
\end{itemize}

If we are in the spider case the differential is called a Jenkins-Strebel differential and if the corresponding spider graph has just one node then
the differential is called a Strebel differential. In the latter case there is just one cylinder for every puncture.
Note that a Jenkins-Strebel differential can only have poles of order $2$ with a negative residue.
A classical result by Strebel \cite{strebel1984quadratic} tells us that if $S$ is a Riemann surface with punctures $p_1,\dots,p_n$ and we specify 
residues $a_1,\dots,a_n<0$ then there is a unique Strebel differential on $S$ with these residues.

This result can be generalized as follows.
\begin{theorem}[Liu]\cite{liu2008jenkins}
Let $\Lambda$ be a weighted spider graph with $n$ legs and fix an $n$-tuple $(a_i) \in \R_{>0}$.

For each complex structure on $\tube \Lambda$ there is a unique Jenkins-Strebel differential $\Phi$ such that
$\LH(\Phi)$ and $\Lambda$ are isomorphic as weighted spider graphs and the maps
projection maps $\pi: \tube \Lambda \to \LH(\Phi)$ and $\pi: \tube \Lambda \to \Lambda$ are homotopic.
\end{theorem}

This theorem implies that we can use Jenkins-Strebel differentials with a fixed weighted spider graph to paramatrize the space of
complex structures on $S$, i.e. Teichmuller space.

\begin{theorem}\label{doublegraph}
Let $\Gamma$ be a weighted ribbon graph of type $\Sigma_{g,n}$, let $\Lambda$ be a weighted spider graph of the same type, and 
choose an identification $S:=\tube \Lambda = \tub \Gamma$.

There exists a unique Jenkins-Strebel differential $\Phi$ on $S$ such that 
\begin{enumerate}
 \item  $\Lambda \cong \LH(\Phi)$ and $S \to \LH(\Phi)$ is homotopic to $\tube \Lambda \to \Lambda$,
 \item  $\Gamma \cong \LH(-\Phi)$ and $S \to \LH(-\Phi)$ is homotopic to $\rib \Gamma \to \Gamma$. 
\end{enumerate}
\end{theorem}
\begin{proof}
Draw the dual graph of the ribbon graph on $S$. For each edge $a$ in the ribbon graph, let $a^\perp$ be the edge of the dual graph. Choose an orientation for $a^\perp$ 
and look at the projection $\pi: \tube \Lambda \to \Lambda$. This gives a sequence of edges of $\Lambda$ starting with a leg and ending with a leg. For each edge $e$ in the sequence
we cut out a strip $[0,w_e]\times [0,w_a]i$ equipped with the differential $dz^2$. Now we glue all these strips together as follows.

Each node of the ribbon graph corresponds to a polygon in the dual graph bounded by dual edges. 
Under $\pi: \tube \Lambda \to \Lambda$ this factors through a tree that is immersed
in $\Lambda$. The preimage of each edge $e$ of the tree is bounded by two dual edges $a^\perp,b^\perp$
and we glue 

Every dual edge $a^\perp$ with an image under $\pi$ that runs through a node of the tree, 
is mapped to two edges $e_1,e_2$ incident with that node. 
We glue $[0,w_{e_1}]\times [0,w_a]i$ and $[0,w_{e_2}]\times [0,w_b]i$ together along the common edge.
In each node of valency $k$, $2k$ strips come together forming an angle of $2k\times \frac{\pi}{2}=k\pi$.
We identify each node with a zero $z^{k-2}(dz)^2$. In this way we get a quadratic differential
on each polygon for which the dual edges are vertical leaves. If we glue all these polygons together 
we get the required $\Phi$.

From the construction it is clear that any Jenkins-Strebel differential with the properties above must have the same strip decomposition, so it is unique.
\end{proof}

\begin{remark}
General quadratic differentials appear in the theory of stability conditions as initial data to construct Bridgeland stability conditions for Fukaya categories.
This has been studied by Bridgeland an Smith for 3-dimensional Fukaya categories in \cite{bridgeland2015quadratic} and by Haiden et al. for topological
Fukaya categories of surfaces \cite{haiden2014flat}. On the other hand, we are interested in the link between Strebel quadratic differentials and King (i.e. GIT) stability \cite{king1994moduli} conditions as they
are more suited to the construction of moduli spaces. 
\end{remark}

\subsection{The B-model: representations of the Jacobi algebra}

\subsubsection{Moduli spaces of representations}
In this section we fix a consistent dimer on a torus $\qpol$ and following \cite{broomhead2012dimer,ishii2009dimer,bocklandt2016dimer}
we will look at the space of representations of its Jacobi algebra $\Jac=\Jac(\qpol)$ 
for the fixed dimension vector $\alpha:\qpol_0\to \N$ that maps every vertex to one. These are the $\kk$-algebra morphisms
$\rho: \Jac \to \Mat_n(\C)$ where $\kk=\C^{\# \qpol_0}$ is identified with subalgebra in $\Jac$ generated by the vertices and 
the subalgebra of diagonal matrices in $\Mat_n(\C)$. 

The space of these representations will be denoted by $\rep(\Jac,\alpha)$ and is in fact an affine scheme. We can present its ring of coordinates as
\[
\C[\rep(\Jac,\alpha)]=\C[x_a| a\in \qpol_1]/\<\prod_{b\in r_a^+ } x_b - \prod_{b\in r_a^- } x_b| a \in \qpol_a\>.
\]
because a representation maps each arrow $a$ to a scalar $x_a$ times the elementary matrix $E_{h(a)t(a)}$.
A representation is called a torus representation if these scalars are nonzero for every arrow. The torus representations
form an open set of $\rep(\Jac,\alpha)$ and its closure is a component $\trep(\Jac,\alpha)\subset\rep(\Jac,\alpha)$ called the scheme of all toric representations.
One can prove \cite{broomhead2012dimer} that a representation $\rho$ is toric if the set of zero arrows $\{a | \rho((a)=0 \}$ is a union of perfect matchings.

On ${\trep}(\Jac,\alpha)$ there is an action of $\GL_\alpha := \kk^*=\C^{*\# \qpol_0}$ by conjugation. The orbits of this action classify the isomorphism classes of representations
of $\Jac$, so if we want to construct a moduli space toric of representations we need to construct a GIT-quotient \cite{king1994moduli}. 

We start with a map $\theta:\qpol_0 \to \Z$,  which can also be seen as a character
$\GL_\alpha \to \C^*: g \mapsto g^\theta:= \prod_{v\in \qpol_0}g_v^{\theta_v}$. From 
this datum we construct a graded ring of semi-invariants
\[
 {\Semi}_\theta(\Jac,\alpha) := \bigoplus_{i=0}^\infty \{\phi \in \C[{\trep}(\Jac,\alpha)]: \phi(x^g) = g^{n\theta}\phi(x) \}
\]
The proj of this ring ${\Mst}_\theta(\Jac,\alpha) := \Proj\, {\Semi}_\theta(\Jac,\alpha)$
is called the \emph{moduli space of $\theta$-semistable toric representations}.
\erbij{Because ${\Mst}_\theta(\Jac,\alpha) \cong {\Mst}_{n\theta}(\Jac,\alpha)$,
it also makes sense to define the moduli space if $\theta: \qpol_0 \to \QQ$:
we set ${\Mst}_\theta := {\Mst}_{n\theta}$ where $n$ is the least common multiple of the denominators in $\theta$. } 

A representation $\rho$ is called \emph{$\theta$-semistable} if there is a nontrivial $f  \in \Semi_\theta(\Jac,\alpha)$ with $f(\rho)\ne 0$. 
One can prove that from a representation theoretic point of view this means that $\theta\cdot \alpha=\sum_{v\in \qpol_0}\alpha_v\theta_v=0$ and
that it has no proper subrepresentations with dimension vector $\beta$ 
such that $\beta\cdot \theta<0$. 
An $\alpha$-dimensional semistable representation $\rho$ is called \emph{$\theta$-stable} 
if it has no proper subrepresentations with dimension vector $\beta$ such that $\beta\cdot \theta\le 0$, and it is called \emph{$\theta$-polystable} 
if it is the direct sum of stable representations. 
A character is called generic if $\beta\cdot \theta\ne 0$ for all dimension vectors $\beta<\alpha$. For generic $\theta$ the three notions of stability 
coincide. For $\theta=0$ stable is the same as simple, polystable is the same as semisimple and every representation is semistable.  

It is well known \cite{king1994moduli} that $\Mst_\theta(\Jac,\alpha)$ is the categorical quotient of the space of all semistable representations by the 
$\GL_\alpha$-action. The closed orbits are the orbits of polystable representations and the points of $\Mst_\theta(\Jac,\alpha)$ 
correspond to the isomorphism classes of $\theta$-polystable representations in $\trep(\Jac,\alpha)$. Note that $\theta\cdot \alpha$ has to be zero otherwise the moduli space $\Mst_\theta(\Jac,\alpha)$ is empty.

The structure of this toric variety for generic $\theta$ has been studied in \cite{ishii2009dimer,mozgovoy2009crepant}, but we are interested in a slightly different approach that can be found in \cite{bocklandt2016dimer}. Instead of starting with a character we start with a weigth function on the arrows $W: \qpol_1 \to \Z$ and associate to it a
character
\[
 \theta_W: \qpol_0 \to \Z : v \mapsto \sum_{h(a)=v}W_a- \sum_{t(a)=v}W_a.
\]
Because every $W_a$ appears twice with opposite sign we have that $\theta_W\cdot \alpha=0$.
Furthermore if $\qpol$ is connected, any $\theta$ for which $\theta\cdot \alpha=0$ is equal to a $\theta_W$ for an appropriate $W$.

Fix a vertex $v \in \qpol_0$ and $3$ cyclic paths $x,y,z:v\ot v$ in $\qpol$ such that $x,y$ span the 
homology of the torus and $z$ is a cycle in $\qpol_2$. These $3$ paths correspond to $\GL_\alpha$-invariant functions and one can show that all
rational $\GL_\alpha$-invariant functions are generated by these. Similarly all rational $\theta_W$-semi-invariants are the product of
$u= \prod x_a^{W_a}$ and an invariant. By \cite{broomhead2012dimer} one can check whether an expression $ux^iy^jz^k$ is polynomial by checking whether
$\deg_\PM ux^iy^jz^k\ge 0$ for all perfect matchings. 

If we define the tropical polynomial
\[
f_W := \min_{\PM \in \mathrm{PM}(\qpol)}  (\deg_{\PM} x) X + (\deg_{\PM} y) Y + \sum_{a \in \PM}W_a.
\]
where $\PM$ runs over all perfect matchings of $\qpol$, then $ux^iy^jz^k$ is polynomial if and only if $(i,j,k) \in \PT_f$.

\begin{theorem}\cite{bocklandt2016dimer}
If $\qpol$ is consistent then the normalization of the moduli space $\Mst_{\theta_W}$ can be identified with $\VV_{f_W}$.
\end{theorem}
\begin{remark}
Again everything makes sense for rational weight $W$ functions, because the conditions
of being (semi,poly)-stable are the same for $\theta_W$ and $\lambda\theta_W$ if $\lambda$
is any positive scalar.
\end{remark}

To describe the torus orbits in $\Mst_{\theta_W}$,
we have to look a bit closer at $f_W$. The Newton polygon of $f_W$ is the same as the matching polygon of $\qpol$. 
Perfect matchings that are on the same lattice point in the matching polygon generate
terms $(\deg_{\PM} x) X + (\deg_{\PM} y) Y + \sum_{a \in \PM}W_a$ that differ only by a constant and hence
only the one with the lowest constant $\sum_{a \in \PM}W_a$ will contribute to the function $f_W$. 

\begin{definition}
A perfect matching $\PM$ is \emph{semistable} if $(\deg_{\PM} x) X + (\deg_{\PM} y) Y + \sum_{a \in \PM}W_a$ defines a face of $\PT_f$. It is called
stable if no other perfect matching defines the same face.
\end{definition}

\begin{lemma}
A perfect matching $\PM$ is (semi)stable if and only if the representations
\[
 \rho_\PM(a)=\begin{cases}
0 & a\in \PM\\
1 & a\not\in \PM
             \end{cases}
\]
is $\theta_W$-(semi)stable.
\end{lemma}
\begin{proof}
The representation $\rho_\PM$ is semistable if and only if there is a $\theta_W$ semi-invariant $s=ux^iy^jz^k$ that is nonzero for $\rho$. This semi-invariant must sit
on the face defined by $\PM$ because if $\deg_\PM s>0$ then $s(\rho)=0$.

If $\rho_\PM$ is semistable then it cannot be polystable because the complement of the perfect matching is connected. Therefore
there must be a polystable representation $\rho$ in the closure of the $\GL_\alpha$-orbit of $\rho_\PM$. The set of zero arrows of this representation
is the union of at least 2 perfect matchings $\PM$ and $\PM'$. Because the invariants of $\rho_\PM$ and $\rho$ are the same $\PM'$ must define
the same face of $\PT_f$.
\end{proof}

\begin{definition}
We say that $W$ is 
\begin{itemize}
 \item \emph{generic} if all semistable toric representations are stable.
 \item \emph{nondegenerate} if all semistable perfect matchings are stable.
\end{itemize}
\end{definition}

If $W$ is nondegenerate every $2d$-torus orbit corresponds precisely to one perfect matching and all representations
in the orbit are stable. This means that the nonstable locus of $\Mst_{\theta_W}$ has dimension at most $1$.
The generic condition is stronger than the nondegenerate condition because in the latter case the nonstable locus is empty.
An example of a nondegenerate $W$ that is not generic is $W=0$ because one can show that the simple locus of $\Mst_\alpha^0$ has dimension at most $1$ \cite{bocklandt2012consistency}.

\subsubsection{Matrix factorizations from representations}
Given an algebra $B$ one can construct its category of singularities. 
This is the Verdier quotient
of its derived category by the subcategory of perfect complexes, 
which are the bounded complexes of projective modules:
$$\Sing B := \frac{D^b\Mod B}{\mathtt{Perf} B}$$
If $B$ is the quotient of an algebra $A$ with finite global dimension by a central element $\ell$ 
then Buchweitz \cite{buchweitz1986maximal} showed that taking the cokernel
$\mathrm{cok}\, d_{01}$ of a matrix factorization $(P,d)\in \MF_{\Z/2\Z}(A,\ell)$ and 
viewing it as a $B$-module, induces an equivalence between the category of matrix factorizations and the category of singularities of $B$.

If we go back to the representation theory of the Jacobi algebra, we see that every representation $\rho: \Jac \to \Mat_n(\C)$ that maps
the central element $\ell$ to zero, can be considered as an element in $\Sing \Jac/\<\ell\> \cong \MF(\Jac,\ell)$. 

Suppose $\rho \in \trep(\Jac,\alpha)$ and $\rho(\ell)=0$. Because $\ell$ is a sum of cycles in $\qpol_2$ there must be at least one arrow $a$ in each cycle of $\qpol_2$ for which $\rho(a)=0$. This means that there is a perfect matching $\PM$ such
that $a \in \PM \implies \rho(a)=0$.

This perfect matching defines a grading on $\Jac$ and $\rho$ can be seen as a graded representation of $\Jac$ concentrated in degree $0$ because it factors 
through $\Jac/\Jac_{>0}$. The matrix factorization corresponding to $\rho$ will also be graded and hence by the result by Haiden, Katzarkov and Kontsevich \cite{haiden2014flat}
it will correspond to a direct sum of string and band objects. We will now determine an explicit decomposition of this. 

The idea is quite simple: we draw curves in the neighborhood of each arrow in the following way and connect the dotted curves. 
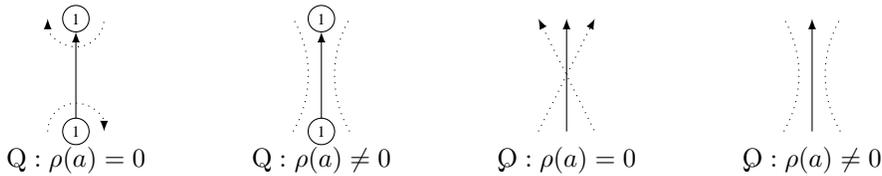
\vspace{.2cm}
\begin{center}
 \begin{tikzpicture}[scale=.75]
\draw [-latex,shorten >=5pt] (0,0)--(0,2) ;
\draw (0,0) node[circle,draw,fill=white,minimum size=10pt,inner sep=1pt] {{\tiny1}};
\draw (0,2) node[circle,draw,fill=white,minimum size=10pt,inner sep=1pt] {{\tiny1}};
\draw[-latex,dotted] (-.5,0) arc (180:0:.5);
\draw[latex-,dotted] (-.5,2) arc (180:360:.5);
\draw (0,-.5) node{$\qpol: \rho(a)=0$};
\end{tikzpicture}
\hspace{1cm}
 \begin{tikzpicture}[scale=.75]
\draw[-latex,shorten >=5pt] (0,0)--(0,2);
\draw (0,0) node[circle,draw,fill=white,minimum size=10pt,inner sep=1pt] {{\tiny1}};
\draw (0,2) node[circle,draw,fill=white,minimum size=10pt,inner sep=1pt] {{\tiny1}};
\draw[dotted] (-.5,0) arc (330:390:2);
\draw[dotted] (.5,0) arc (210:150:2);
\draw (0,-.5) node{$\qpol:\rho(a)\ne 0$};
\end{tikzpicture}
\hspace{1cm}
 \begin{tikzpicture}[scale=.75]
\draw[-latex] (0,0)--(0,2);
\draw[-latex,dotted] (-.5,0)--(.5,2);
\draw[latex-,dotted] (-.5,2)--(.5,0);
\draw (0,-.5) node{$\polq: \rho(a)=0$};
\end{tikzpicture}
\hspace{1cm}
 \begin{tikzpicture}[scale=.75]
\draw[-latex] (0,0)--(0,2);
\draw[dotted] (-.5,0) arc (330:390:2);
\draw[dotted] (.5,0) arc (210:150:2);
\draw (0,-.5) node{$\polq:\rho(a)\ne 0$};
\end{tikzpicture}
\end{center}
The resulting curves are the bands.
Note that some of these curves can be contractible in $\psurf\polq$. If this is the case
this curve will not determine a band, so we omit it.
The decorations of the bands is the product of all $\rho(a)^{\pm 1}$ 
along which the curve passes multiplied by a sign $(-1)^l$ where $2l$ is the number of faces through which the band passes. 
The exponent of $\rho(a)$ is $-1$ if the band goes in the same direction as the arrow and $+1$ if it goes opposite.
We denote this collection of decorated bands by $B_\rho$.

\begin{theorem}
If $\qpol$ is consistent and $\rho\in \rep(\Jac(\qpol),\alpha)$ is a representation that is 
zero for the perfect matching $\PM$ then the object $B_\rho \in \Tw_{\Z_2} \Gtl\polq)$ is 
isomorphic to $\rho \in \Sing \Jac/\<\ell\>$ under the standard identification
\[
 \Tw_{\Z_2} \Gtl(\polq) \cong \Tw_{\Z_2} \mf(\qpol) \subset \H\MF(\Jac,\ell) \cong 
 \Sing \Jac/\<\ell\>.
\]
\end{theorem}
\begin{proof}
First of all $\rho$ does not contain any string objects because
$\Hom_{\Sing \Jac/\<\ell\>}(\rho,\rho)$ is finite while the endomorphism
ring of a string object is always infinite.

To determine the band objects we will look at the hom spaces between $\rho$ and
certain test objects. If we determine the dimensions of these spaces we can single
out the unique combination of bands that has the same hom spaces.
The test object we consider are the objects $M_a$ corresponding to the arrows.

Note that for a given band object $g=(g_0,\dots,g_{2k-1})$ and an $n\times n$ Jordan matrix $\lambda$, the dimension of $\Hom^\bullet(B(g,\lambda), M_a)$ is equal to $n$ times the number of arrows $g_{i}$ that are equal to $a$. Each matching arrow contributes a morphism
of even degree if $i=0 \mod 4$ and of odd degree if $i=2 \mod 4$. This can easily be
seen in the $A$-model because $B(g,\lambda)$ represents curves that runs through
the center of these arrows. 

In $\Sing \Jac\!/\<\ell\>$ the object $M_a$ corresponds to the
periodic complex
\[
\bar M_a : P_{h(a)}\stackrel{a}{\ot} P_{t(a)}\stackrel{\ell a^{-1}}{\ot}  P_{h(a)}\stackrel{a}{\ot} P_{t(a)}\stackrel{\ell a^{-1}}{\ot} \dots
\]
Where $P_v$ stands for the projective module $v\Jac$.
Using the fact that $\Hom(P_{v},S)=\C^{\alpha_v}=\C$ we see that $\Hom_{\Sing}^\bullet (\bar M_a, \rho)$ is the homology of the complex
\[
 \C\stackrel{\rho(a)}{\to} \C \stackrel{\rho(\ell a^{-1})}{\to}  \C\stackrel{\rho(a)}{\to} \C\stackrel{\rho(\ell a^{-1})}{\to} \dots
\]
This complex only has nontrivial homology if $\rho(a)=\rho(\ell a^{-1})=0$ and in that case
$$\Ext^i_{\Jac\!/\<\ell\>}(P_{h(a)}/\<a\>,S)=\C.$$
This implies that the object $\rho$ only contains band objects that run through
arrows that are zero for $\rho$. Furthermore all decorations must be $1$-dimensional. 

To determine the precise bands, note that $B(g,\lambda)$ is a band object that occurs in the decomposition of $\rho$ or $\rho[1]$ if and only if the dimension of the space
\[
 \Hom^0(B(g,\kappa),\rho)
\]
will jump by one if $\kappa$ is equal to $\lambda$ compared to when $\kappa \ne \lambda$.

If we go to the matrix factorization $B(g,\kappa)$ we can calculate 
$\Hom^i(B(g,\kappa),\rho)$ by calculating the homology of the complex 
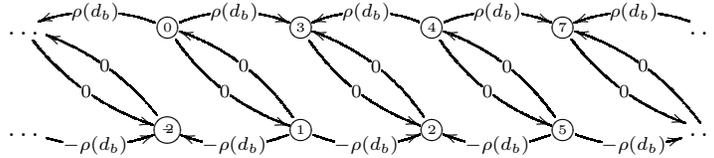
\[
\xymatrix@R=.4cm@C=.6cm{
\dots\ar@{<-}@/^/[ddrr]|{0}\ar@{<-}@/^/[rr]|{\rho(d_b)}&&\vtx{0}\ar@{<-}@/^/[ddrr]|{0}&&\vtx{3}\ar@{<-}@/^/[ddrr]|{0}\ar@{<-}@/_/[ll]|{\rho(d_b)}\ar@{<-}@/^/[rr]|{\rho(d_b)}&&\vtx{4}\ar@{<-}@/^/[ddrr]|{0}&&\vtx{7}\ar@{<-}@/_/[ll]|{\rho(d_b)}\ar@{<-}@/^/[rr]|{\rho(d_b)}\ar@{<-}@/^/[ddrr]|{0}&&\dots\\
&&&&&&&&&&\\
\dots&&\vtx{-\!\!\!2}\ar@{<-}@/^/[uull]|{0}\ar@{<-}@/^/[ll]|{-\rho(d_b)}\ar@{<-}@/_/[rr]|{-\rho(d_b)}&&\vtx{1}\ar@{<-}@/^/[uull]|{0}&&\vtx{2}\ar@{<-}@/^/[uull]|{0}\ar@{<-}@/^/[ll]|{-\rho(d_b)}\ar@{<-}@/_/[rr]|{-\rho(d_b)}&&\vtx{5}\ar@{<-}@/^/[uull]|{0}&&\dots\ar@{<-}@/^/[ll]|{-\rho(d_b)}\ar@{<-}@/^/[uull]|{0}
}
\]
where we changed every vertex to the vector space $\C$ because the dimension vector is 
$(1,\dots,1)$, reversed all arrows and substituted them by their values for $\rho$.

If $g=(g_1,\dots,g_{2k-1})$ is a band for which all the arrows $g_i$ evaluate to zero,
the homology of the complex reduced to that of $d_b$, which splits in an upper and a lower part. If one of the arrows in the upper row 
is zero, the $d_b$-complex becomes contractible, so no jump can take place.
If all arrows in the upper row are nonzero, we can form an element in the homology
by chosing a number $\beta_0$ in vertex $0$, $-\rho(p_2)^{-1}\rho(p_1)\beta_0$ in vertex $4$
(where $p_1$ and $p_2$ are the paths that connect vertices $1$ and $4$ with $3$). Proceeding like this
all the way around the upper row,
it is clear we get an element in the homology if and only if $(-1)^l\prod (\rho(p_i))^{\pm 1}\lambda=1$
where $2l$ is the number of faces of the band. 
Therefore dimension of the hom-space jumps if the decoration is equal to 
$(-1)^l\prod \rho(p_i)^{\mp 1}$.

The same can be said about the lower band but in that case the homology element 
that corresponds to the identity will sit in degree one instead of degree zero. From this we can conclude that 
the bands that occur are precisely those for which the upper row only consists
of nonzero arrows. 
We can draw these bands in $\qpol$ by drawing two half circle paths 
through each zero arrow
and two parallel lines along each nonzero arrow. Because the mirror dimer $\polq$ is twisted around each arrow, the two half circles become a cross.
\end{proof}
\begin{remark}
One can extend this result for representation with a dimension smaller than $\alpha$. To do this
we need to remove the half circles that run around vertices with dimension zero in the $\qpol$-quiver.
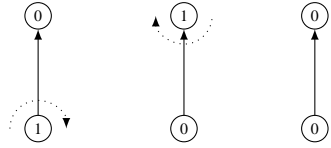
\vspace{.2cm}
\begin{center}
 \begin{tikzpicture}[scale=.75]
\draw [-latex,shorten >=5pt] (0,0)--(0,2) ;
\draw (0,0) node[circle,draw,fill=white,minimum size=10pt,inner sep=1pt] {{\tiny1}};
\draw (0,2) node[circle,draw,fill=white,minimum size=10pt,inner sep=1pt] {{\tiny0}};
\draw[-latex,dotted] (-.5,0) arc (180:0:.5);
\end{tikzpicture}
\hspace{1cm}
 \begin{tikzpicture}[scale=.75]
\draw [-latex,shorten >=5pt] (0,0)--(0,2) ;
\draw (0,0) node[circle,draw,fill=white,minimum size=10pt,inner sep=1pt] {{\tiny0}};
\draw (0,2) node[circle,draw,fill=white,minimum size=10pt,inner sep=1pt] {{\tiny1}};
\draw[latex-,dotted] (-.5,2) arc (180:360:.5);
\end{tikzpicture}
\hspace{1cm}
\begin{tikzpicture}[scale=.75]
\draw [-latex,shorten >=5pt] (0,0)--(0,2) ;
\draw (0,0) node[circle,draw,fill=white,minimum size=10pt,inner sep=1pt] {{\tiny0}};
\draw (0,2) node[circle,draw,fill=white,minimum size=10pt,inner sep=1pt] {{\tiny0}};
\end{tikzpicture}
\end{center}
\end{remark}

\subsubsection{The mirror dimer and the spider graph}

To construct a moduli space of matrix factorizations of $(\Jac,\ell)$, we can look at a moduli space of the form $\Mst_\theta(\Jac,\alpha)$ and restrict to those representations that give a nonzero matrix factorization. 

To find the matrix factorizations, we have to look at $\Mst_\theta(\Jac\!/\<\ell\>,\alpha)$ inside $\Mst_\theta(\Jac,\alpha)$.
Because $\ell$ is zero as soon as one arrow in each cycle of $\qpol_2$ is zero, these are the representation that are zero on at least one perfect matching.
However if $\rho$ is zero for only one perfect matching, 
the construction in the previous paragraph only results in contractible curves that each run in two neighboring faces. This tells us that the matrix factorization is zero.
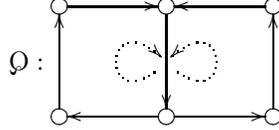
\[
\polq: \vcenter{{\xymatrix@=.5cm{
\vtx{}\ar[rr]&&\vtx{}\ar[dd]&&\vtx{}\ar[ll]\\
&&\ar@(rd,ru)@{.>}\ar@(ld,lu)@{.>}&&\\
\vtx{}\ar[uu]&&\vtx{}\ar[rr]\ar[ll]&&\vtx{}\ar[uu]
}}}
\]
We will now restrict to the case where $\theta=\theta_W$ for a \emph{weight function $W:\qpol_1\to \Z$ that is nondegenerate}.
In this case each two-dimensional orbit of $\Mst_\theta(\Jac,\alpha)$ is the zero locus of precisely one perfect matching. The one-dimensional orbits will be the zero locus
of two perfect matchings and the zero-dimensional orbits of more than two perfect matchings. 
Therefore the representations that give nonzero matrix factorizations correspond to 
union of the zero- and one-dimensional orbits. We denote this union by
\[
 \Msg_{\theta}(\Jac,\alpha) \subset \Mst_\theta(\Jac,\alpha).
\]

\begin{lemma}\label{kbands}
Suppose $\qpol$ is consistent.
Let $\rho=\rho_{\PM_1\cup\PM_2}$ be a representation that is zero 
for two perfect matchings $\PM_1,\PM_2$ located on lattice points $n_{\PM_1},n_{\PM_2}$ 
and let 
$k$ be the elementary length of the line segment between $n_{\PM_1}$ and $n_{\PM_2}$ (i.e. the line segment contains $k+1$ lattice points).
\begin{enumerate}
 \item The representation $\rho$ splits as a direct sum of $k$ indecomposable representations 
 $\rho_1\oplus \dots \oplus \rho_k$.
 \item Each matrix factorization $M_{\rho_i}$ corresponds to the direct sum of a band and its opposite
 \item All the bands of the $M_{\rho_i}$ are (anti)parallel in $\qpol$.
 \end{enumerate}
\end{lemma}
\begin{proof}
Because there are at most two arrows zero in every face the bands look like
$$(a_1,f_1,a_2,f_2,a_3,\dots, a_n,f_n),$$
where $a_i \in \PM_1$ if $i$ is odd and $a_i \in \PM_2$ if $i$ is even. Also if a band
is present then also its opposite is present.
Moreover two bands must either be parallel or antiparallel in because an intersection
would result in at least three zero arrows in a face.

The vector $n_{\PM_1}-n_{\PM_2}=(\deg_{PM_1}x -\deg_{PM_2}x,\deg_{PM_1}y -\deg_{PM_2}y,0)$ measures the intersection
number of parallel bands with $x$ and $y$ so the homology class of the parallel bands is 
$$(-\deg_{PM_1}y +\deg_{PM_2}y,\deg_{PM_1}x -\deg_{PM_2}x).$$ 
If $\gcd(-\deg_{PM_1}y +\deg_{PM_2}y,\deg_{PM_1}x -\deg_{PM_2}x)=k$
this means that the homology class is that of $k$ simple curves so there must be $k$ parallel bands and
$k$ antiparallel bands, so $\rho$ decomposes as a sum of $2k$ indecomposable objects that are pairwise each other's shift.

On the torus $\surf{\qpol}$ the bands split the surface in $k$ cylinders and all arrows
crossing the bands are $0$, so $\rho$ is the direct sum of $k$ subrepresentations, one 
supported on each cylinder. Each of these representations is indecomposable as a 
representation. If it were not indecomposable $M_{\rho_i}$ would split in more than $2$ indecomposable 
matrix factorizations.

\erbij{
To show that the decorations of the bands are the same we look at a one-parameter 
family $\rho_t$ with $\lim_{t\to 0}=\rho$.
If $g$ and $h$ are parallel band then the weak paths $p,q$ that bounds 
the cycles on the left.
Let $r$ be any weak path that connects $h(p)$ with $h(q)$. The paths $rp$ and $qr$ have the same homotopy class and the same degree
for $\PM_1$ so by the consistency of the dimer they represent the same element $\Jac$. 
Therefore
$$\frac{\rho(p)}{\rho(q)} = \lim_{t\to 0}\frac{\rho_t(p)}{\rho_t(q)}=\lim_{t\to 0}\frac{\rho_t(rp)}{\rho_t(qr)}=1$$
so the decorations are the same.}
\end{proof}

Each $k$-fold edge of $\Trop(f_W)$ corresponds to a one-parameter orbit of polystable representations
that decompose into $k$ summands. This implies that there is a correspondence between 
the edges of the spider graph $\Lambda(f_W)$ and the number of bands that occur in the matrix factorizations of the one-dimensional 
torus orbit. Each of these bands corresponds to an unoriented curve in $\psurf{\polq}$. 
To make a moduli space that parametrizes these summands we can take the fibered product
\[
\Mla_{\theta_W} :=  \Msg_{\theta_W}(\Jac,\alpha) \times_{\Trop_{f_W}} \Lambda(f_W)
\]
This moduli space has $k$ 1-parameter family of each indecomposable summand and comes with
a natural projection $\Mla_{\theta_W} \to \Lambda(f_W)$. Note that in the generic 
case every polystable representation is stable and hence indecomposable, 
so $\Mla_{\theta_W} =  \Msg_{\theta_W}(\Jac,\alpha)$.

From theorem \ref{genusmirror} and lemma \ref{genustropical} we know that the $\psurf{\polq}$ 
and $\tub{\Lambda(f_W)}$ are surfaces with the same genus and same number of punctures.
Therefore it is a natural question to ask whether we can draw the dimer quiver $\polq$ on 
$\tub{\Lambda(f_W)}$ such that the bands of the matrix factorizations
correspond to simple curves that wrap around the edges of $\Lambda(f_W)$.
This is indeed possible. We will first show this for generic $W$ and
then extend the result to nondegenerate $W$.

First we need to determine how the projection of each arrow onto 
the spider graph must look. Because an arrow connects two punctures, its projection
will be a sequence of edges, starting with a leg and ending in a leg. 
Which edges are needed can be determined by the subdivision of the Newton Polygon 
$\NP(f_W)$.

For any subpath $p$ of a cycle in $\qpol_2$ 
we can mark all lattice points that contain an arrow of $p$ 
and look at the subcomplex $C_p$ consisting of all edges and triangles 
that contain only marked
lattice points. 

\vspace{.3cm}
\begin{center}
\resizebox{!}{2cm}{\begin{tikzpicture}
\draw (.5,-2) node {{$C_a$}};
\fill [black!25] (-1,0)--(0,1)--(1,1)--(0,0)--(-1,0);
\draw  [black!25] (0,-1) -- (-1,0); 
\draw  [black] (0,0) -- (-1,0); 
\draw  [black] (-1,0) -- (0,1); 
\draw  [black] (0,1) -- (1,1); 
\draw  [black!25] (1,1) -- (2,0); 
\draw  [black!25] (2,0) -- (1,-1); 
\draw  [black!25] (0,0) -- (1,-1); 
\draw  [black!25] (1,1) -- (1,-1); 
\draw  [black] (0,1) -- (0,0); 
\draw  [black!25] (1,-1) -- (0,-1); 
\draw  [black!25] (0,0) -- (0,-1); 
\draw  [black] (1,1) -- (0,0); 
\draw  [black!25] (2,0) -- (0,0); 
\draw (-1,0) node[circle,draw,fill=black,minimum size=10pt,inner sep=1pt]{};
\draw (0,1) node[circle,draw,fill=black,minimum size=10pt,inner sep=1pt]{};
\draw (1,1) node[circle,draw,fill=black,minimum size=10pt,inner sep=1pt]{};
\draw (2,0) node[circle,draw,fill=white,minimum size=10pt,inner sep=1pt]{};
\draw (1,-1) node[circle,draw,fill=white,minimum size=10pt,inner sep=1pt]{};
\draw (0,-1) node[circle,draw,fill=white,minimum size=10pt,inner sep=1pt]{};
\draw (1,0) node[circle,draw,fill=white,minimum size=10pt,inner sep=1pt]{};
\draw (0,0) node[circle,draw,fill=black,minimum size=10pt,inner sep=1pt]{};
\end{tikzpicture}}
\hspace{1cm}
\resizebox{!}{2cm}{\begin{tikzpicture}
\draw (.5,-2) node {{$C_{r_a^+}$}};
\fill [black!25] (1,0)--(2,0)--(1,-1)--(1,0);
\draw  [black!25] (0,-1) -- (-1,0); 
\draw  [black!25] (0,0) -- (-1,0); 
\draw  [black!25] (-1,0) -- (0,1); 
\draw  [black!25] (0,1) -- (1,1); 
\draw  [black!25] (1,1) -- (2,0); 
\draw  [black] (2,0) -- (1,-1); 
\draw  [black!25] (0,0) -- (1,-1); 
\draw  [black] (1,-1) -- (0,-1); 
\draw  [black!25] (1,1) -- (0,0); 
\draw  [black!25] (2,0) -- (0,0); 
\draw  [black!25] (0,1) -- (0,-1);
\draw  [black!25] (1,1) -- (1,0);
\draw  [black] (1,0) -- (1,-1);
\draw  [black] (1,0) -- (2,0);
\draw (-1,0) node[circle,draw,fill=white,minimum size=10pt,inner sep=1pt]{};
\draw (0,1) node[circle,draw,fill=white,minimum size=10pt,inner sep=1pt]{};
\draw (1,1) node[circle,draw,fill=white,minimum size=10pt,inner sep=1pt]{};
\draw (2,0) node[circle,draw,fill=black,minimum size=10pt,inner sep=1pt]{};
\draw (1,-1) node[circle,draw,fill=black,minimum size=10pt,inner sep=1pt]{};
\draw (0,-1) node[circle,draw,fill=black,minimum size=10pt,inner sep=1pt]{};
\draw (1,0) node[circle,draw,fill=black,minimum size=10pt,inner sep=1pt]{};
\draw (0,0) node[circle,draw,fill=white,minimum size=10pt,inner sep=1pt]{};
\end{tikzpicture}}
\end{center}

\erbij{\resizebox{!}{1.2cm}{\begin{tikzpicture} 
\draw [red, thick](-4/3,2/3) -- (-7/3,-1/3); 
\draw [black!25](-4/3,8/3) -- (-4/3,2/3); 
\draw [black!25](-4/3,8/3) -- (-7/3,11/3); 
\draw [black!25](-4/3,8/3) -- (-4/3,11/3); 
\draw [red, thick](-2/3,2) -- (1/3,3); 
\draw [black!25](-2/3,4/3) -- (1/3,1/3); 
\draw [red, thick](-4/3,2/3) -- (-2/3,4/3); 
\draw [black!25](-4/3,2/3) -- (-4/3,-1/3); 
\draw [black!25](-4/3,8/3) -- (-2/3,2); 
\draw [red, thick](-2/3,2) -- (-2/3,4/3); 
\end{tikzpicture}}}

\begin{lemma}\label{cont}
Let $p$ be as subpath of a cycle in $\qpol_2$. If $\qpol$ is consistent 
and $\theta=\theta_W$ generic then the subcomplex $C_p$ is contractible.
\end{lemma}
\begin{proof}
If $\theta$ is generic then Ishii and Ueda proved \cite{ishii2009dimer} 
that $\Mst_\theta(\Jac,\alpha)$ is derived equivalent to $\Jac$.  
They did this by constructing a tilting bundle on $\Mst_\theta(\Jac,\alpha)$.
Pick a base vertex $w$ and let $p$ 
be a path from $w \to v$. Then define
\[
 \cL_{w\to v} := \cL( \sum_{\PM} \deg_\PM(p)D_{\PM})
\]
where the sum is taken over all stable perfect matchings and $D_\PM$ is the toric divisor corresponding to the lattice point of $\PM$.
This line bundle does not depend on the choice of path between $w$ and $v$, only on the choice of the vertices.

For each vertex $w$ the sum
\[
 \oplus_{v \in \qpol_0} \cL_{w \to v}
\]
is a tilting bundle, so we have that $$\Ext^{>0}(\cL_{w\to w},\cL_{w\to v})=\H^{>0}(\cL_{w\to v})=0$$ 
for all $v,w \in \qpol_0$. 

Clearly the arrow $a$ is a path from $t(a)$ to $h(a)$ so
\[
\cL_{t(a)\to h(a)} = \cL(\sum_{\PM} \deg_\PM(a)D_{\PM}).
\]
must have vanishing higher cohomology.
From toric geometry \cite{fulton1993introduction} we know that the higher 
homology of $\cL_{t(a)\to h(a)}$ comes from the homology of 
the subfan marked by the divisors. If $C_p$ is not contractible this homology will be nonzero
and $\oplus_{v \in \qpol_0} \cL_{t(a) \to v}$ cannot be a tilting bundle.
Similarly $C_{r_a^+}$ is also contractible because the higher cohomology of
$\cL_{h(a)\to t(a)}$ vanishes.
\end{proof}

\newcommand{\Line}{\mathsf{line}}
\newcommand{\Tree}{\mathsf{tree}}

For every arrow the subcomplexes $C_a$ and $C_{r_a^+}$ together contain all 
lattice points and are both contractible, therefore we can draw a single curve through
the Newton polygon that separates both complexes. This curve $\Line_a$ can be seen as a sequence 
of edges in the dual graph, which is the spider graph.
\vspace{.3cm}
\begin{center}
\resizebox{!}{1.3cm}{\begin{tikzpicture}
\fill [black!25] (-1,0)--(0,1)--(1,1)--(0,0)--(-1,0);
\draw  [black!25] (0,-1) -- (-1,0); 
\draw  [black] (0,0) -- (-1,0); 
\draw  [black] (-1,0) -- (0,1); 
\draw  [black] (0,1) -- (1,1); 
\draw  [black!25] (1,1) -- (2,0); 
\draw  [black!25] (2,0) -- (1,-1); 
\draw  [black!25] (0,0) -- (1,-1); 
\draw  [black!25] (1,1) -- (1,-1); 
\draw  [black] (0,1) -- (0,0); 
\draw  [black!25] (1,-1) -- (0,-1); 
\draw  [black!25] (0,0) -- (0,-1); 
\draw  [black] (1,1) -- (0,0); 
\draw  [black!25] (2,0) -- (0,0); 
\draw (-1,0) node[circle,draw,fill=black,minimum size=10pt,inner sep=1pt]{};
\draw (0,1) node[circle,draw,fill=black,minimum size=10pt,inner sep=1pt]{};
\draw (1,1) node[circle,draw,fill=black,minimum size=10pt,inner sep=1pt]{};
\draw (2,0) node[circle,draw,fill=white,minimum size=10pt,inner sep=1pt]{};
\draw (1,-1) node[circle,draw,fill=white,minimum size=10pt,inner sep=1pt]{};
\draw (0,-1) node[circle,draw,fill=white,minimum size=10pt,inner sep=1pt]{};
\draw (1,0) node[circle,draw,fill=white,minimum size=10pt,inner sep=1pt]{};
\draw (0,0) node[circle,draw,fill=black,minimum size=10pt,inner sep=1pt]{};
\draw[red] (-1,-1) -- (-.33,-.33)--(.33,-.66)--(.66,-.33)--(.66,.33)--(1.33,.33)--(2,1); 
\end{tikzpicture}}
\end{center}

Consider a face $a_1\dots a_k$ in $\qpol_2$ and look at the edges in
the spider graph. Every lattice point in the Newton polygon correspond to one stable perfect matching and 
hence is contained in precisely one of the $C_{a_i}$. Each edge is either dual to an edge in one of the
$C_{a_i}$ or it is dual to an edge that connects two lattice points
that are in two different $C_{a_i}$. The latter edges form a tree $\Tree(c)$ in the spider
graph. Indeed if $\Tree(c)$ contains a cycle then there is at least one lattice point inside the cycle.
This lattice point must be contained in one of the $C_a$ and hence the full $C_a$ is contained
inside this cycle, but then $C_{r_a^+}$ cannot be contractible.

Because $\Tree(c)$ is embedded in the plane as a subgraph of the tropical curve, we
can also consider a small closed neighborhood $U_c$ of $\Tree(c)$ in the plane. Topologically, this
neigborhood is homeomorphic to a $k$-gon with the corners removed and
there is a map $\retr_c: U_c \to \Trop(f_W)$ that maps this polygon onto the 
tree $\Tree(c)$.

In a positive cycle the $C_{a_i}$ will sit next to each other in an anticlockwise way, while
for a negative cycle they will follow in a clockwise way. This is because the legs of 
$\Line_a$ correspond to the zig ray and the zag ray in $\qpol$. If $ab$ sits in a positive cycle
the zig ray of $a$ is equal to the zag ray of $b$, while it is the other way round if $ab$ sits
in a negative cycle.

Therefore the ends of the lines match up like the edges of the polygons do in $\polq$,
and we can glue the trees together to obtain a surface that is homeomorphic to
$\psurf{\polq}$. Gluing all the maps $\retr_c$ together we get a
map $\retr :\psurf{\polq} \to \Trop(f_w)$. If $\theta$ is generic we can identify
$\Trop(f_w)$ with $\Lambda(f_W)$.

This means we have identified $\psurf{\polq}$ and $\tub{\Lambda(f_W)}$ so we can translate
theorem \ref{doublegraph} to this setting.

\begin{theorem}
Let $\qpol$ be a consistent quiver and consider a generic weight function
$W:\qpol_1 \to \Z$ and a map $B:\qpol_1 \to \R_>0$.
There exists a unique Strebel differential $\Phi$ on $\psurf{\polq}$ such that
\begin{enumerate}
 \item 
the horizontal strips correspond to the arrows $a \in \polq_1$ and have width $B_a$,
 \item
the vertical cylinders correspond to the bands that occur as summands in decomposition
of the matrix factorizations of the $\theta_W$-stable one-dimensional
orbits in $\Mst_{\theta_W}(\Jac,\alpha)$. They have a width equal to the affine length of the corresponding edge in $\Trop(f_L)$.
\end{enumerate}
In other words, we have a diagram
\[
 \xymatrix{\psurf{\polq}\ar[dr]^\pi&&\Mla_{\theta_W}\ar[dl]_{\pi'}\\&\Lambda(f_W)=\LH(\Phi)&}
\]
such that $\pi^{-1}\pi'(\rho)$ is a curve that represents the bands in $M_\rho$.
\end{theorem}
\begin{proof}
We have an identification $\psurf{\polq}=\tub{\Lambda_{f_W}}$ so we can apply 
theorem \ref{doublegraph} to $\Lambda=\Lambda_{f_W}$ and $\Gamma=\polq^\perp$. 
To identify the bands with the cylinders, let $x$ be a point on 
an edge $e$ of $\Trop(f_W)$. The inverse image $\pi^{-1}(x)$ will
in each polygon consist of a line connecting 2 arrows $a,b$ 
with lines $\Line_a,\Line_b$ that contain $e$. If $e$ is dual to the edge
that connects the perfect matchings $\PM_1$ and $\PM_2$ then $a$ and $b$
will each be contained in one of these perfect matchings. Therefore the
total inverse image will be a curve that runs along the band representing
$\rho_{\PM_1\cup\PM_2}$.
\end{proof}

\begin{remark}
One can generalize this statement to the case of nondegenerate weights as well. 
The idea is to approximate the nondegenerate weight $W$ by generic ones $W'$ (using
rational weight functions). 
If we chose $W'$ close enough to $W$, the only thing that can happen is that
a $\theta_{W'}$-stable perfect are $\theta_{W}$-unstable but not the other way round. 
This means that in limit the 
cycles corresponding to the $\theta_{W}$-unstable perfect matchings 
will shrink and disappear in the tropical curve. 
There is a continuous map from $\Trop_{f_{W'}}\to \Trop_{f_{W}}$.
The image of the curve $\Line_a\subset \Trop_{f_{W'}}$ will still be a line in 
$\Trop_{f_{W}}$ 
that separates the lattice points of stable matchings containing $a$ from those not containing $a$.
The image of the tree $\Tree(c)$ will also be a tree in $\Trop_{f_{W}}$. This means
we can still apply the same procedure to map $\psurf{\polq}$ to $\Trop(f_W)$.

The only difference now is that the inverse image of a point on a $k$-fold edge
will now consist of $k$ circles. Therefore we can factorize the map
\[
 \xymatrix{\psurf{\polq}\ar[rr]\ar[rd]&&\Trop(f_W)\\&\Lambda(f_W)\ar[ru]&}
\]
and we have found and identification of $\psurf{\polq}$ with $\tub{\Lambda(f_W)}$.
\end{remark}

\section{Gluing and Localizing}\label{sectionglue}

The Fukaya category of a punctured surface can be constructed by gluing smaller Fukaya categories together. 
There are several ways to make this precise: see for instance \cite{dyckerhoff2013triangulated} and \cite{pascaleff2016topological}.
In this section we will use a method that combines naturally with the theory of stability
conditions and discuss how this can be interpreted in the $A$- and the $B$-model.

Our starting point is again a consistent dimer on a torus $\qpol$ and its mirror dimer $\polq$.
On $\qpol$ we consider a nondegenerate stability condition $\theta=\theta_W$ and
on $\psurf{\polq}$ we consider a quadratic Jenkins-Strebel differential $\Phi$ for which the horizontal
leaf space is $\LH(\Phi)=\Lambda(f_W)$ and the vertical leaves are the arrows of the quiver.

\subsection{The A-model: restriction functors}

Each point $p$ of the spider graph $\LH(\Phi)$ corresponds
to a new punctured surface with a Strebel differential on it.
This surface can be obtained by taking the inverse image of a small open
neighborhood $U_p \subset \LH(\Phi)$. 
If $p$ is an internal point of an edge this surface is a cylinder and if $p$ is a
node of the spider graph it is  a surface with a genus that is equal to 
the genus of the node and a number of holes, that is
equal to the valence of the node. 

The restriction of the quadratic differential gives a quadratic differential on this
surface. We can extend the surface and the differential by gluing infinite cylinders to the holes. These cylinders are of the form $(\frac{\R+\R_{>0}i}{\ell\Z},-dz^2)$
where the circumference $\ell$ is equal to the circumference of the corresponding tube.
In this way we get a new punctured surface with a Strebel differential: $(\dot S_p,\Phi_p)$

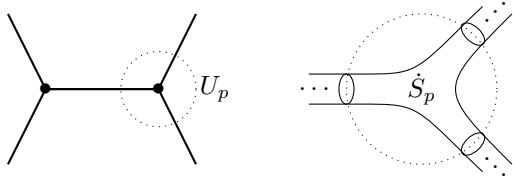
\begin{center}
\begin{tikzpicture}
\begin{scope}
\draw [thick](-.5,-1) -- (0,0) -- (1.5,0) -- (2,-1);
\draw [thick](-.5,1) -- (0,0) -- (1.5,0) -- (2,1);
\draw (0,0) node{$\bullet$};
\draw (1.5,0) node{$\bullet$};
\draw (2.25,0) node{$U_p$};
\draw[dotted] (2,0) arc (0:360:.5);
\end{scope}
 
\begin{scope}[xshift=5cm]
\draw[dotted] (1,0) arc(0:360:1);
\draw (-1.5,.2)--(-.8,.2) ..controls (-.1,.2) and (-.1,.2) .. (.6,.9) -- (.81,1.1);
\draw (-1.5,-.2)--(-.8,-.2) ..controls (-.1,-.2) and (-.1,-.2) .. (.6,-.9) -- (.81,-1.1);
\draw (1.2,1) ..controls (.2,0) and (.2,0) .. (1.2,-1);
\draw (-1.4,0) node{$\dots$};
\draw (1,1.05) node[rotate=45]{$\dots$};
\draw (1,-1.05) node[rotate=-45]{$\dots$};
\draw (-1,0) ellipse (0.1 and 0.2);
\draw[rotate=133] (-1,0) ellipse (0.1 and 0.19);
\draw[rotate=-133] (-1,0) ellipse (0.1 and 0.19);
\draw (0,0) node{$\dot S_p$};
\end{scope}

\end{tikzpicture}
\end{center}
If $p$ is a point of an internal edge we have that
$(S_p,\Phi_p)\cong(\C^*,-a \frac{dz^2}{z^2})$ for some $a \in \R_{>0}$. For all points in the edge $e$ this pair is the same, so we denote it by $(\dot S_e,\Phi_e)$.

If $p$ is a node, the vertical leaves of $(\dot S_p,\Phi_p)$ correspond to arrows for which the image under
$\pi:\psurf{\polq} \to \Lambda(f_W)$ passes through $p$. In fact we can see $S_p$ as the punctured surface of a new dimer
$\polq_p$. The arrows of the dimer $\polq_p$ are in one to one correspondence to the arrows passing through $p$ under $\pi$.
If $a_1\dots a_k$ is a face $\polq_2^\pm$, we construct an analogous face for $\qpol_p$ by removing the arrows that do not pass through $p$.
If none of the arrows remain we omit the face. 

Each arrow $a$ in $\polq$ corresponds to an idempotent $e_a$ in $\Gtl(\polq)$ and each nonzero path $\beta$ in $\Gtl(\polq)$ winds around one of the punctures of $\psurf{\polq}$. There is a natural algebra morphisms $\cF:\Gtl(\polq) \to \Gtl(\polq_p)$ that maps $e_a$ to the corresponding idempotent in $\Gtl(\polq_p)$ if $a$ intersects $S_p$ and zero otherwise. The image of a path $\beta$ is nonzero if it winds around a tube that is also present in $S_p$ and zero otherwise.

This algebra morphism can be extended to an $A_\infty$-morphism. Let $k>1$ and suppose that $\rho_1,\dots,\rho_k$ are nonzero angle paths in $\Gtl(\polq)$ and
\begin{itemize}
 \item[F1] $h(\rho_1)\in \polq_{p1}$ but $h(\rho_i)\not\in \polq_{p1}$ if $i>1$,
 \item[F2] $t(\rho_k)\in \polq_{p1}$ but $t(\rho_i)\not\in \polq_{p1}$ if $i<k$. 
 \item[F3] $\rho_i\rho_{i+1}=0$ for $1<i<k-1$,
\end{itemize}
We define
\[
 \cF(\rho_1,\dots,\rho_k) = (-1)^{|\rho|} \rho
\]
If the path $\rho_1\dots\rho_k$ in $\psurf{\polq}$ is homotopic to the angle path $\rho$ in $\psurf{{\polq}_p}\cong U_p \subset \psurf{\polq}$. 
all other $\cF$-products of paths are zero. 

\begin{lemma}
$\cF$ is an $A_\infty$-morphism.
\end{lemma}
\begin{proof}
First note that by conditions F1 and F2, given a sequence of angle paths $\rho_1,\dots,\rho_k$ in $\Gtl \polq$
there is at most one expression of the form
\[
 \mu(\cF(\rho_1,\dots,\rho_{i_1}),\dots, \cF(\rho_{i_{l+1}},\dots,\rho_{k}))
\]
that is nonzero. This is the expression obtained by cutting the $\rho$-sequence
in bits such that the cuts occur at all the heads $h(\rho_i)$ that are in $\polq_{p1}$.

This means that the left hand side of the $A_\infty$-morphism identity \cite{keller1999introduction}
\se{
&\sum_{l,i_1,\dots, i_l}\pm \mu(\cF(\rho_1,\dots,\rho_{i_1}),\dots, \cF(\rho_{i_{l+1}},\dots,\rho_{k}))\\
=&
\sum_{r<s}\pm \cF(\rho_1,\dots,\rho_{r},\mu(\rho_{r+1},\dots,\rho_s),
\rho_{s+1},\dots,\rho_{k})
}
has at most one term. If that term is nonzero then either the term comes from a $\mu_2$ 
or from a higher order $\mu$. 
In the first case this term is 
$\mu_2(\cF(\dots,\rho_{i_1}),\cF(\rho_{i_1+1}),\dots))$ and 
we have one nonzero term on the right hand side:
\[
\cF(\dots,\mu_2(\rho_{i_1},\rho_{i_1+1}),\dots). 
\]
In the second case, set $\tilde \rho=\mu(\cF(\rho_1,\dots,\rho_{i_1}),\dots, \cF(\rho_{i_{l+1}},\dots,\rho_{k}))$. By lemma \ref{muhomotopy} the angle path 
$\tilde \rho$ in $\polq_p$ is either trivial or a subpath of
$\cF(\rho_1,\dots,\rho_{i_1})$ or $\cF(\rho_{i_{l+1}},\dots,\rho_{k})$.
If $\tilde \rho$ is trivial then $h(\rho_1)=t(\rho_k)=\mu(\rho_1,\dots,\rho_k)$ and
the RHS has one nonzero term $\cF(\mu(\rho_1,\dots,\rho_k))$.

If $\tilde \rho$ is a subpath of $\cF(\rho_1,\dots,\rho_{i_1})$ then $t(\tilde\rho)$ is an arrow $a$ that enters the contractible region bounded by the arrows $h(\rho_1),\dots,t(\rho_k)$. The arrow will end in a puncture and will cut one of the angles
$\rho_i$ in $2$: $\rho_i=\alpha\beta$ with $h(\beta)= a$. 
\begin{center}
\resizebox{!}{2cm}{
 \begin{tikzpicture}
\foreach \s in {1,2,3,4,5,6,7,8,9}
{
\draw (230-\s*25:1.5)node {$\rho_\s$};
}
\draw (0:1.5)--(180:2);
\draw (0,0) node[above] {$a$};
\draw (5:1.2) arc (5:55:1.2);
\draw (65:1.2) arc (65:135:1.2);
\draw (140:1.2) arc (140:205:1.2);
\draw (30:.9) node {$\cF$}; 
\draw (92:.9) node {$\cF$}; 
\draw (170:.9) node {$\cF$}; 
\draw (175:1.9) node {$\beta$}; 
\draw (185:1.9) node {$\alpha$}; 
\draw (-5:1.2) -- (210:1.2);
\draw (285:.8) node{$\mu$};
\draw (2.15,0) node{$=$};
\begin{scope}[xshift=4.5cm]
\foreach \s in {1,2,3,4,5,6,7,8,9}
{
\draw (230-\s*25:1.5)node {$\rho_\s$};
}
\draw (0:1.5)--(180:2);
\draw (0,0) node[above] {$a$};
\draw (5:1.2) arc (5:175:1.2);
\draw (92:.9) node {$\mu$}; 
\draw (175:1.9) node {$\beta$}; 
\draw (185:1.9) node {$\alpha$}; 
\draw (-5:1.2) to[out=185,in=120] (210:1.2);
\draw (285:.5) node{$\cF$};
\end{scope}
\end{tikzpicture}}
\end{center}

Therefore $\mu(\rho_i,\dots,\rho_k)=\alpha$ and $\cF(\rho_1,\dots,\rho_{i-1},\mu(\rho_i,\dots,\rho_k))$ is nonzero. It is the only nonzero term on the RHS because of lemma \ref{muhomotopy}. 
The case when $\tilde\rho$ is a subpath of $\cF(\rho_{i_{l+1}},\dots,\rho_{k})$ is similar. So we can conclude that if the LHS is nonzero, the RHS has one nonzero term, which is equal to the LHS because it is an angle path with the same homotopy.
\vspace{.3cm}

Now let's have a closer look at the right hand side. We distinguish the following cases.
\begin{itemize}
 \item[C1] If one of the $\rho_i$ is an idempotent then the right hand side has two
 terms with a $\mu_2$ that give a possibly nonzero contribution. These two terms have opposite sign and hence they cancel out. All other terms have the idempotent in a higher $\mu$ or $\cF$-term and hence by $F3$ and the definition of $\mu$ they are zero. Similarly the left hand side is zero.
 \item[C2]
 If $\rho_i\rho_{i+1}$ is nonzero and neither $\rho_i$ nor $\rho_{i+1}$ is an idempotent then
 the only terms on the right hand side that are nonzero must contain a $\mu_2(\rho_i,\rho_{i+1})$ or a higher $\mu$ that either contains $\rho_i$ or $\rho_{i+1}$ but not both (again by lemma \ref{muhomotopy}).

 We distinguish two subcases.
 \begin{itemize}
 \item[C2.1] If $h(\rho_{i+1}) \in \polq_{p1}$ then only the $\mu_2$-term on the right hand side can be nonzero because of F1 and F2. If this term is indeed nonzero, F1 and F2 also
 imply that $h(\rho_j) \not \in \polq_{p1}$ for $j\ne {i+1}$. 
 On the left hand side this means that the possibly nonzero term is of the
 form $\mu_2(\cF(\rho_1,\dots,\rho_i),\cF(\rho_{i+1},\dots,\rho_k))$ and by lemma
 \ref{muhomotopy} it is equal to the RHS.
 \item[C2.2] If $a=h(\rho_{i+1}) \not \in \polq_{p1}$ then $h(\rho_{i+1})$ must correspond to
 an arrow that can be drawn completely inside the polygon spanned by the arrows
 $h(\rho_{j})$. This arrow must connect two corners of the polygon: one corresponding
 to the puncture around which $\rho_{i+1}$ winds and another one around which $\rho_j$
 winds. If $j>{i+1}$ then term $\cF(\dots,\mu(\rho_{i+1},\dots,\rho_{j})\dots)$ 
 cancels the term with $\mu_2$ and if $j<i$ then the term 
 $\cF(\dots,\mu(\rho_{j},\dots,\rho_{i+1})\dots)$ does this. The other terms on the RHS are zero. The left hand side is zero because the split has only one term (and $\mu_1=0$).
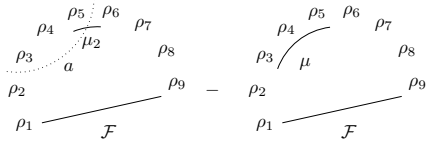
\begin{center}
\resizebox{!}{2cm}{
\begin{tikzpicture}
\foreach \s in {1,2,3,4,5,6,7,8,9}
{
\draw (230-\s*25:1.5)node {$\rho_\s$};
}
\draw (87:1.2) arc (87:112:1.2);
\draw (97:.9) node {$\mu_2$}; 
\draw (140:.7) node {$a$}; 
\draw[dotted] (92:1.8) to[out=-88,in=-12] (168:1.8); 
\draw (-5:1.2) -- (210:1.2);
\draw (285:.8) node{$\cF$};
\draw (2.15,0) node{$-$};
\begin{scope}[xshift=4.5cm]
\foreach \s in {1,2,3,4,5,6,7,8,9}
{
\draw (230-\s*25:1.5)node {$\rho_\s$};
}
\draw (97:1.2) arc (97:160:1.2);
\draw (140:.8) node {$\mu$}; 
\draw (-5:1.2) -- (210:1.2);
\draw (285:.8) node{$\cF$};
\end{scope}
\end{tikzpicture}}
\end{center}
\end{itemize}
\item[C3] If there is a higher $\mu$-term $\mu(\rho_i,\dots,\rho_j)$ that is nonzero on the RHS then
\begin{itemize}
\item[C3.1]
If $i=1$ and $j=k$ then no other term on the RHS can be nonzero and the term on the LHS
corresponds to the same contractible disk.
\item[C3.2]
If $i>1,j<k$ then $\beta = \mu(\rho_i,\dots,\rho_j)$ must be a nontrivial angle path. Therefore it must be a subpath of either $\rho_i$ or $\rho_j$. 
If it is a subpath of $\rho_i=\beta\gamma$ then $\rho_j\rho_{j+1}\ne 0$, so we are in C2.
Similarly if $\rho_j=\alpha\beta$ we can show that $\rho_{i-1}\rho_{i}\ne 0$.
\item[C3.3]
If $i>1$ and $j=k$ we can apply the same reasoning as C3.2 but not when $\rho_j=\alpha\beta$. In that case all consecutive product $\rho_i\rho_{i+1}$ are zero. This implies that
we can group them into subsequences for which $\cF(\rho_u,\dots,\rho_v)\ne 0$ and 
because of lemma \ref{muhomotopy} applied in $\polq_p$, there is a nonzero term on the LHS.
\item[C3.4]
If $i=1$ and $j<k$ we can apply the same reasoning as C3.3.
\end{itemize}
\end{itemize}
This shows that if the RHS has a nonzero term then either there is another RHS term that cancels it or the LHS has a nonzero term.
\end{proof}

By going to the twisted completion, the $A_\infty$-morphism gives rise to an $A_\infty$-functor 
\[
\wfuk(\dot S) \to \wfuk(\dot S_p).
\]
What does this functor do geometrically? If at curves that can be drawn inside
$\pi^{-1}(U_{p})$ it is clear that the corresponding band object will be mapped to a band object with the same curve. 
From this point of view it is clear that this functor is a special case of the restriction functors constructed by Dyckerhoff in \cite{dyckerhoff2013triangulated} and Pascaleff and Sibilla in \cite{pascaleff2016topological}.

All the categories we get out of these restriction maps can be glued together to obtain the original
Fukaya category. For each edge $e$ between nodes $n_1$, $n_2$ we have a diagram
\[
  \wfuk S_{n_1} \to \wfuk S_{p} \ot \wfuk S_{n_2},
\]
which all together make a big diagram of $A_\infty$-categories. We can consider this diagram as a diagram inside
the category $\mathsf{dgcat}_{\Z_2}$, which is the category of $\Z_2$-graded dg-categories localized by the quasi-equivalences.  
In this category we can take the colimit, which is known as the homotopy colimit.

\begin{theorem}[Dyckerhoff\cite{dyckerhoff2015a1},Pascaleff-Sibilla\cite{pascaleff2016topological}]
There is an equivalence between the homotopy colimit of 
\[
\cup_{e \in \Lambda_1(f_W)}  \wfuk S_{n_1} \to \wfuk S_{p} \ot \wfuk S_{n_2},
\]
and the Fukaya category $\wfuk S$.
\end{theorem}

\subsection{The B-model: matrix factorizations and sheaves}

Given a consistent dimer $\qpol$ and a nondegenerate stability condition
$\theta =\theta_W$ we can look at the representation space $\Mst_{\theta}$.
Each point in this space corresponds to a semistable representation $\rho$ and we can look at
the universal localization 
\[
 \Jac_\rho := \Jac\< a^{-1}| \rho(a)\ne 0\>.
\]
Because $\qpol$ is consistent, by theorem \ref{consprop} this algebra can be seen as a subalgebra of
$$\wJac := \Mat_n(\C[X^{\pm 1},Y^{\pm 1},Z^{\pm 1}]).$$

The representation classes $\sigma \in \Mst_{\theta}$ such that $\rho(a)\ne 0 \implies \sigma(a)\ne 0$ form 
an open subset $U_\rho \subset \Mst_{\theta}$ and following \cite{adriaenssens2003local} this space can be seen as $\Mst_0(\Jac_\rho)$.
This means that $\Mst_{\theta}$ can be covered by a sheaf of algebras, such that each affine open part is a representation space of semistable representations of a new stabilitity condition. 

In general these algebras can be quite complicated but in the case of a nondegenerate stability condition they have nice properties.
\begin{theorem}
If $\theta_W$ is a nondegenerate stability condition then all 
$\Jac_{\rho}$ are 3-Calabi-Yau algebras with center $\C[U_{\rho}]$, which implies that $\Mst_{\theta}$ can be covered 
with a sheaf of Calabi Yau algebras.
\begin{itemize}
\item If $\rho$ sits in a $2$-dimensional torus orbit then $\Jac_{\rho}$ is Morita equivalent to $$\C[\C^*\times \C^*\times \C].$$
\item If $\rho$ sits in a $1$-dimensional torus orbit corresponding to an edge with length $k$ then 
$\Jac_{\rho}$ is Morita equivalent to $\C[\C^*]\times \C[X,Y]\star \Z_k$, where $gXg^{-1} =\zeta X$ and
$gYg^{-1} =\zeta^{-1} Y$ for a $k^{th}$ root of unity $\zeta$.
\item If $\rho$ sits in a $0$-dimensional torus orbit then $\Jac_{\rho}$ is Morita equivalent to the Jacobi algebra of a consistent dimer $\qpol_\rho$
with a matching polygon equal to the subpolygon associated to $U_\rho$.
\end{itemize}
\end{theorem}
\begin{proof}
We say that two vertices in $\qpol$ are equivalent if there is a path $p$ between them
with $\rho(p)\ne 0$. In $\Jac_\rho$ these paths have inverses and therefore if we chose one vertex $v_i$ for every equivalence class and take the sum $w=\sum_i v_i$
we see that $\Jac_\rho$ is Morita equivalent to $w \Jac_\rho w$. 

\begin{itemize}
 \item 
If $\rho$ is in a $2$-dimensional torus orbit then the nondegeneracy of $\theta$ implies that there is precisely one perfect matching $\PM$ on which $\rho$ is zero.
Because $\qpol \setminus \PM$ is connected all vertices are in the same equivalence class.
This implies that $\Jac_\rho$ is Morita equivalent to $v\Jac_\rho v\cong \C[U_{\rho}]$  which is Calabi-Yau-3 because $U_\rho=\C^*\times \C^*\times \C$.
\item
If $\rho$ is in a 1-dimensional torus orbit then the nondegeneracy of $\theta$ implies that there are precisely two perfect matchings $\PM_1,\PM_2$ on which $\rho$ is zero.
We can draw the curves according to $\rho_{\PM_1\cup \PM_2}$ on the surface $\surf{\qpol}$. These curves split the torus into strips and all vertices in the same strip are
equivalent, because there is a weak path that goes around along the strip. This weak path gives
rise to a central element $f$.

There are $3$ important types of generators in $w \Jac_\rho w$.
Arrows $x_i$ that connect two strips from the left to the right, arrows $y_i$ that connect two strips from the right to the left
and the central elements $f,f^{-1}$.
Two arrows connecting the same strips in the same direction are equivalent because viewed from fixed vertices they correspond
to weak paths with the same homology class.
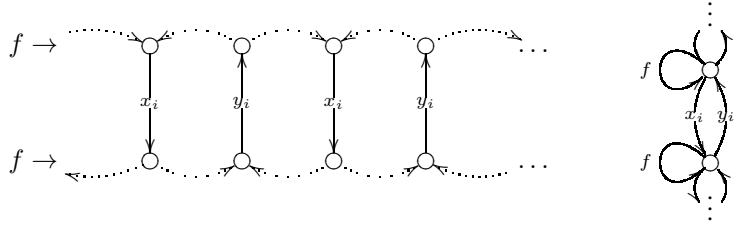
\[
\hspace{2cm}
\vcenter{
\xymatrix@=.4cm{
f\to\ar@{.>}@/^/[rr]&&\vtx{}\ar[dd]|{x_i}&&\vtx{}\ar@{.>}@/_/[ll]\ar@{.>}@/^/[rr]&&\vtx{}\ar[dd]|{x_i}&&\vtx{}\ar@{.>}@/_/[ll]\ar@{.>}@/^/[rr]&&\dots\\
&&&&&&&&&&\\
f\to&&\vtx{}\ar@{.>}@/^/[ll]\ar@{.>}@/_/[rr]&&\vtx{}\ar[uu]|{y_i}&&\vtx{}\ar@{.>}@/^/[ll]\ar@{.>}@/_/[rr]&&\vtx{}\ar[uu]|{y_i}&&\dots \ar@{.>}@/^/[ll]
}
}
\hspace{1cm}
\vcenter{
\xymatrix@=.2cm{
\vdots\ar@/_/[d]\\
\vtx{}\ar@/_/[ddd]|{x_i}\ar@/_/[u]\ar@(lu,ld)_f\\
\\~\\
\vtx{}\ar@/_/[d]\ar@/_/[uuu]|{y_i}\ar@(lu,ld)_f\\
\vdots\ar@/_/[u]
}}
\]

The quiver is the double of the extended dynkin quiver $\tilde A_n$ with relations $x_iy_{i}=y_{i+1}x_{i+1}$,
and a loop in each vertex corresponding to $f$. 
The algebra $w \Jac_\rho w$ is the preprojective algebra tensored with $\C[f,f^{-1}]$. 
Because the first factor is CY2 \cite{bocklandt2008graded} and the 
second CY1 the whole is CY3.
\item

If $\rho$ is zero dimensional then there are at least $3$ perfect matchings which are zero for $\rho$. These perfect matchings, $\PM_1,\dots,\PM_k$ lie on different points in the matching polygon, so there is no nontrivial weak path that has zero degree for all these perfect matchings. Therefore the vertices in an equivalence class and the invertible arrows connecting them form a contractible subset on the torus.

If we contract all the invertible arrows we end up with a quiver on a surface.
By lemma \ref{existspath}, the algebra $w \Jac_\rho w$ is in fact the path algebra of this new quiver where two paths are identified if they have the same homology and $\PM_i$-degree. 
We do not need all arrows of this quiver to generate the whole algebra. If $a$ is contracted
to a contractible loop then this path will be equivalent to any cycle in $\qpol_2$ starting in the same vertex. Therefore we can remove loops and we denote the remaining quiver $\qpol_\rho$.

The mirror of $\qpol_\rho$ is the dimer $\polq_p$ where $p$
is the point in the spider graph corresponding to $\rho$. Indeed $a$ is an arrow
that passes through $p$ under the projection $\dot S(\polq) \to \Trop(f_W)$ if and only if 
$p$ is on the boundary between the perfect matchings that contain $p$ and those that do not.
If $a$ does not pass through $p$ then either $a \not \in \PM_1\cup\dots\cup\PM_k$ or
$a \in \PM_1\cap\dots\cap\PM_k$. In the former case we contracted $a$ to obtain $\qpol_\rho$ and in the latter cases we removed it because it was a contractible loop.

A dimer $\qpol$ is \emph{well-ordered} if the cyclic order in which the zigzag cycles meet any positive cycle is the same as the cyclic order of their homology classes.
This property was introduced by Gulotta in \cite{gulotta2008properly} and Ishii and Ueda proved that a dimer is well-ordered if and only if it is zigzag-consistent \cite{ishii2010note}. 

We will now show that $\qpol_\rho$ is well-ordered and hence consistent.
The image of a positive or negative cycle in $\polq_2$ under $\pi_{\polq}$ is a tree because these cycles are contractible.  Two arrows of $\polq_p$
that follow each other in the cyclic order of the cycle enter and leave the node $p$ via the same edge in the spider graph. This implies zigzag paths in $\qpol_\rho$
consist of all arrows entering/leaving the node via the same edge. 
These are the arrows of the garland corresponding to that edge.
Hence there is a one to one correspondence between bands coming from edges connected to $p$ and zigzag paths in $\qpol_\rho$. 

If we look at a face in $\qpol_{red2}^+$ we see that the zigzag paths that are incident
with that face, meet that face in the arrows that enter and leave the neighborhood of $p$
by the edge corresponding to that zigzag path. Moreover the homology class of the zigzag path in $\qpol_{red2}^+$ is the same as the direction of this edge. Therefore
the cyclic order of the zigzag paths around the cycle is the same as the cyclic
order of the homology classes of the zigzag paths.

Now we show that the matching polygon of $\qpol_\rho$ is the subpolygon associated to $U_\rho$. The perfect matchings $\PM_1,\dots,\PM_k$ of $\qpol$ give positive gradings on $\Jac_\rho$ and hence also on $w \Jac_\rho w$. Therefore they induce perfect matchings on $\qpol_\rho$.
Vice versa if $\PM$ is a perfect matching on $\qpol_\rho$, we get a positive
grading on $w \Jac_\rho w$, which we can extend to a positive grading $\Jac_\rho$ by giving
the paths we contracted degree zero. This grading gives a perfect matching supported by $\rho$.

Finally because $\qpol_\rho$ is consistent the Jacobi algebra is the path algebra of $\qpol_\rho$ where two paths are identified if they have the same homotopy class
and the same degree for at least one perfect matching. As every perfect matching on $\qpol_\rho$ gives a perfect matching on $\qpol$ we see that $w \Jac_\rho w$ is
isomorphic to $\Jac(\qpol_\rho)$
\end{itemize}
\end{proof}

\begin{remark}
The operation $\qpol\to \qpol_\rho$ can be performed for every collection of arrows,
for which the underlying subgraph is a union of trees. It is in general not true
that the resulting reduced quiver is consistent if the original dimer is. Certain special 
subsets of arrows for which this is the case have been studied by Ishii and Ueda in
\cite{ishii2009dimer}. The result above gives a different way to generate such subsets of arrows. 
\end{remark}

The embedding $\Jac \subset \Jac_\rho$ gives a functor $-\otimes \Jac_\rho:\MF(\Jac,\ell) \to \MF(\Jac_\rho,\ell)$ that maps
$(P,d)$ to $(P \otimes_{\Jac} \Jac_\rho,d)$. Because $\Jac_\rho$ is Morita-equivalent to $\Jac(\qpol_\rho)$, the category
$\MF(\Jac_\rho,\ell)$ will be equivalent to $\MF(\Jac(\qpol_\rho),\ell)$. The shortening lemma implies that
$M_a$ is mapped to zero if $\rho(a)\ne 0$. It is also easy to check that if $\rho(a)=0$ then $M_a$ is mapped to $M_a'$, 
where $a'$ is the arrow corresponding to $a$ in $\qpol_\rho$.
Furthermore if $M_a\otimes \Jac_\rho=M_{a'}$,  $M_b\otimes \Jac_\rho=M_{b'}$ and there is a nonzero morphism $M_a\to M_b$, the tensored version
will also give a nonzero morphism in $\MF(\Jac_\rho,\ell)$ and $\MF(\Jac(\qpol_\rho),\ell)$. This implies that the $\cF_0$ and $\cF_1$ part of the functor we constructed on
the A-side is the same as the tensor functor restricted to $\cD_{\Z_2}\mf(\qpol)$. Hence, they give isomorphic $A_\infty$-functors and
just like in the $A$-model we have a diagram of functors for which the homotopy colimit is equivalent to  $\cD_{\Z_2}\mf(\qpol)$.

\section{An example}\label{sectionexample}

Consider the following dimer $\qpol$ with weight function $W:\qpol_1\to\Z$ such that $W_{a}=1$ if $a\in \{1,4,6,16\}$ and $W_a=0$ otherwise. The corresponding stability condition $\theta_W$ has $7$ stable perfect matchings.
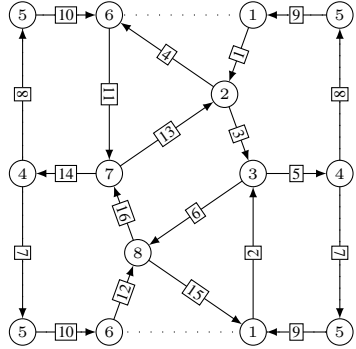
\begin{center}
\begin{tikzpicture}
\begin{scope}[scale=.4]
\begin{scope}
\draw[loosely dotted] (102pt,20pt) rectangle (402pt,320pt);
\draw [-latex,shorten >=5pt] (320pt,320pt) to node [rectangle,draw,fill=white,sloped,inner sep=1pt] {{\tiny 1}} (293pt,245pt); 
\draw [-latex,shorten >=5pt] (184pt,320pt) to node [rectangle,draw,fill=white,sloped,inner sep=1pt] {{\tiny 11}} (184pt,170pt); 
\draw [-latex,shorten >=5pt] (320pt,170pt) to node [rectangle,draw,fill=white,sloped,inner sep=1pt] {{\tiny 6}} (211pt,95pt); 
\draw [-latex,shorten >=5pt] (320pt,20pt) to node [rectangle,draw,fill=white,sloped,inner sep=1pt] {{\tiny 2}} (320pt,170pt); 
\draw [-latex,shorten >=5pt] (293pt,245pt) to node [rectangle,draw,fill=white,sloped,inner sep=1pt] {{\tiny 3}} (320pt,170pt); 
\draw [-latex,shorten >=5pt] (293pt,245pt) to node [rectangle,draw,fill=white,sloped,inner sep=1pt] {{\tiny 4}} (184pt,320pt); 
\draw [-latex,shorten >=5pt] (320pt,170pt) to node [rectangle,draw,fill=white,sloped,inner sep=1pt] {{\tiny 5}} (402pt,170pt); 
\draw [-latex,shorten >=5pt] (102pt,170pt) to node [rectangle,draw,fill=white,sloped,inner sep=1pt] {{\tiny 7}} (102pt,20pt); 
\draw [-latex,shorten >=5pt] (402pt,170pt) to node [rectangle,draw,fill=white,sloped,inner sep=1pt] {{\tiny 7}} (402pt,20pt); 
\draw [-latex,shorten >=5pt] (102pt,170pt) to node [rectangle,draw,fill=white,sloped,inner sep=1pt] {{\tiny 8}} (102pt,320pt); 
\draw [-latex,shorten >=5pt] (402pt,170pt) to node [rectangle,draw,fill=white,sloped,inner sep=1pt] {{\tiny 8}} (402pt,320pt); 
\draw [-latex,shorten >=5pt] (402pt,20pt) to node [rectangle,draw,fill=white,sloped,inner sep=1pt] {{\tiny 9}} (320pt,20pt); 
\draw [-latex,shorten >=5pt] (402pt,320pt) to node [rectangle,draw,fill=white,sloped,inner sep=1pt] {{\tiny 9}} (320pt,320pt); 
\draw [-latex,shorten >=5pt] (102pt,20pt) to node [rectangle,draw,fill=white,sloped,inner sep=1pt] {{\tiny 10}} (184pt,20pt); 
\draw [-latex,shorten >=5pt] (102pt,320pt) to node [rectangle,draw,fill=white,sloped,inner sep=1pt] {{\tiny 10}} (184pt,320pt); 
\draw [-latex,shorten >=5pt] (184pt,20pt) to node [rectangle,draw,fill=white,sloped,inner sep=1pt] {{\tiny 12}} (211pt,95pt); 
\draw [-latex,shorten >=5pt] (184pt,170pt) to node [rectangle,draw,fill=white,sloped,inner sep=1pt] {{\tiny 13}} (293pt,245pt); 
\draw [-latex,shorten >=5pt] (184pt,170pt) to node [rectangle,draw,fill=white,sloped,inner sep=1pt] {{\tiny 14}} (102pt,170pt); 
\draw [-latex,shorten >=5pt] (211pt,95pt) to node [rectangle,draw,fill=white,sloped,inner sep=1pt] {{\tiny 15}} (320pt,20pt); 
\draw [-latex,shorten >=5pt] (211pt,95pt) to node [rectangle,draw,fill=white,sloped,inner sep=1pt] {{\tiny 16}} (184pt,170pt); 
\node at (320pt,20pt) [circle,draw,fill=white,minimum size=10pt,inner sep=1pt] {\mbox{\tiny $1$}}; 
\node at (320pt,320pt) [circle,draw,fill=white,minimum size=10pt,inner sep=1pt] {\mbox{\tiny $1$}}; 
\node at (293pt,245pt) [circle,draw,fill=white,minimum size=10pt,inner sep=1pt] {\mbox{\tiny $2$}}; 
\node at (320pt,170pt) [circle,draw,fill=white,minimum size=10pt,inner sep=1pt] {\mbox{\tiny $3$}}; 
\node at (102pt,170pt) [circle,draw,fill=white,minimum size=10pt,inner sep=1pt] {\mbox{\tiny $4$}}; 
\node at (402pt,170pt) [circle,draw,fill=white,minimum size=10pt,inner sep=1pt] {\mbox{\tiny $4$}}; 
\node at (102pt,20pt) [circle,draw,fill=white,minimum size=10pt,inner sep=1pt] {\mbox{\tiny $5$}}; 
\node at (402pt,20pt) [circle,draw,fill=white,minimum size=10pt,inner sep=1pt] {\mbox{\tiny $5$}}; 
\node at (402pt,320pt) [circle,draw,fill=white,minimum size=10pt,inner sep=1pt] {\mbox{\tiny $5$}}; 
\node at (102pt,320pt) [circle,draw,fill=white,minimum size=10pt,inner sep=1pt] {\mbox{\tiny $5$}}; 
\node at (184pt,20pt) [circle,draw,fill=white,minimum size=10pt,inner sep=1pt] {\mbox{\tiny $6$}}; 
\node at (184pt,320pt) [circle,draw,fill=white,minimum size=10pt,inner sep=1pt] {\mbox{\tiny $6$}}; 
\node at (184pt,170pt) [circle,draw,fill=white,minimum size=10pt,inner sep=1pt] {\mbox{\tiny $7$}}; 
\node at (211pt,95pt) [circle,draw,fill=white,minimum size=10pt,inner sep=1pt] {\mbox{\tiny $8$}}; 
\end{scope}\end{scope}
\draw (250pt,64pt) node {
\begin{minipage}{4cm}
Stable perfect matchings:
\begin{enumerate}
\item $\{ 9, 6, 4, 14 \}$
\item $\{ 8, 2, 4, 16 \}$
\item $\{ 8, 2, 13, 12 \}$
\item $\{ 5, 15, 13, 10 \}$
\item $\{ 3, 15, 11, 7 \}$
\item $\{ 1, 6, 11, 7 \}$
\item $\{ 9, 15, 13, 14 \}$
\end{enumerate}
\end{minipage}
};
\end{tikzpicture}
\end{center}

The matching polygon is a hexagon and it is subdivided in three quadrangles. The numbers of the lattice points in the hexagons correspond to the stable perfect matchings. 
The dual spider graph has $3$ nodes, two with genus zero and one with genus one because the square has an internal lattice point.
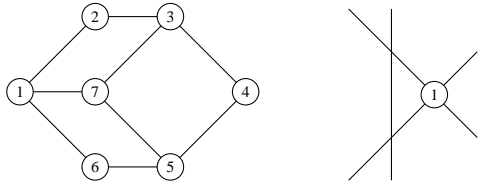
\begin{center}
\begin{tikzpicture}
\draw (0,-1) -- (-1,0); 
\draw (0,0) -- (-1,0); 
\draw (-1,0) -- (0,1); 
\draw (0,1) -- (1,1); 
\draw (1,1) -- (2,0); 
\draw (2,0) -- (1,-1); 
\draw (0,0) -- (1,-1); 
\draw (1,-1) -- (0,-1); 
\draw (1,1) -- (0,0); 
\draw (-1,0) node[circle,draw,fill=white,minimum size=10pt,inner sep=1pt] {{\tiny1}};
\draw (0,1) node[circle,draw,fill=white,minimum size=10pt,inner sep=1pt] {{\tiny2}};
\draw (1,1) node[circle,draw,fill=white,minimum size=10pt,inner sep=1pt] {{\tiny3}};
\draw (2,0) node[circle,draw,fill=white,minimum size=10pt,inner sep=1pt] {{\tiny4}};
\draw (1,-1) node[circle,draw,fill=white,minimum size=10pt,inner sep=1pt] {{\tiny5}};
\draw (0,-1) node[circle,draw,fill=white,minimum size=10pt,inner sep=1pt] {{\tiny6}};
\draw (0,0) node[circle,draw,fill=white,minimum size=10pt,inner sep=1pt] {{\tiny7}};
\end{tikzpicture} 
\hspace{1cm}
\begin{tikzpicture}[scale=.57] 
\draw (-1,-1) -- (-2,-2); 
\draw (-1,1) -- (-1,-1); 
\draw (-1,1) -- (-2,2); 
\draw (-1,1) -- (-1,2); 
\draw (0,0) -- (1,1); 
\draw (0,0) -- (1,-1); 
\draw (-1,-1) -- (0,0); 
\draw (-1,-1) -- (-1,-2); 
\draw (-1,1) -- (0,0); 
\draw (0,0) node[circle,draw,fill=white,minimum size=10pt,inner sep=1pt] {{\tiny1}};
\end{tikzpicture}
\end{center}
Every arrow gives rise to a connected subset of the subdivision of the matching polygon 
which is supported on the lattice points of the stable perfect matchings 
that contain this arrow. The boundary of this subset corresponds 
to a line in the spider graph.
\vspace{.3cm}

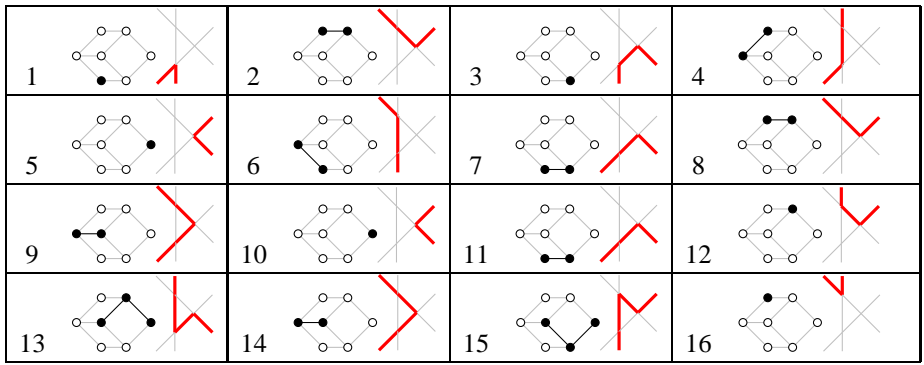
\newcommand{\ZARb}{\begin{tikzpicture}\begin{scope}[scale=.33]\draw  [black!25] (0,-1) -- (-1,0); 
\draw  [black!25] (0,0) -- (-1,0); 
\draw  [black!25] (-1,0) -- (0,1); 
\draw  [black!25] (0,1) -- (1,1); 
\draw  [black!25] (1,1) -- (2,0); 
\draw  [black!25] (2,0) -- (1,-1); 
\draw  [black!25] (0,0) -- (1,-1); 
\draw  [black!25] (1,-1) -- (0,-1); 
\draw  [black!25] (1,1) -- (0,0); 
\draw (-1,0) node[circle,draw,fill=white,minimum size=3pt,inner sep=1pt]{};
\draw (0,1) node[circle,draw,fill=white,minimum size=3pt,inner sep=1pt]{};
\draw (1,1) node[circle,draw,fill=white,minimum size=3pt,inner sep=1pt]{};
\draw (2,0) node[circle,draw,fill=white,minimum size=3pt,inner sep=1pt]{};
\draw (1,-1) node[circle,draw,fill=white,minimum size=3pt,inner sep=1pt]{};
\draw (0,-1) node[circle,draw,fill=black,minimum size=3pt,inner sep=1pt]{};
\draw (0,0) node[circle,draw,fill=white,minimum size=3pt,inner sep=1pt]{};
\end{scope}\end{tikzpicture}{\begin{tikzpicture}\begin{scope}[scale=.25] 
\draw [red, very thick](-1,-1) -- (-2,-2); 
\draw [black!25](-1,1) -- (-1,-1); 
\draw [black!25](-1,1) -- (-2,2); 
\draw [black!25](-1,1) -- (-1,2); 
\draw [black!25](0,0) -- (1,1); 
\draw [black!25](0,0) -- (1,-1); 
\draw [black!25](-1,-1) -- (0,0); 
\draw [red, very thick](-1,-1) -- (-1,-2); 
\draw [black!25](-1,1) -- (0,0); 
\end{scope}\end{tikzpicture}}}
\newcommand{\ZARc}{\begin{tikzpicture}\begin{scope}[scale=.33]\draw  [black!25] (0,-1) -- (-1,0); 
\draw  [black!25] (0,0) -- (-1,0); 
\draw  [black!25] (-1,0) -- (0,1); 
\draw  [black] (0,1) -- (1,1); 
\draw  [black!25] (1,1) -- (2,0); 
\draw  [black!25] (2,0) -- (1,-1); 
\draw  [black!25] (0,0) -- (1,-1); 
\draw  [black!25] (1,-1) -- (0,-1); 
\draw  [black!25] (1,1) -- (0,0); 
\draw (-1,0) node[circle,draw,fill=white,minimum size=3pt,inner sep=1pt]{};
\draw (0,1) node[circle,draw,fill=black,minimum size=3pt,inner sep=1pt]{};
\draw (1,1) node[circle,draw,fill=black,minimum size=3pt,inner sep=1pt]{};
\draw (2,0) node[circle,draw,fill=white,minimum size=3pt,inner sep=1pt]{};
\draw (1,-1) node[circle,draw,fill=white,minimum size=3pt,inner sep=1pt]{};
\draw (0,-1) node[circle,draw,fill=white,minimum size=3pt,inner sep=1pt]{};
\draw (0,0) node[circle,draw,fill=white,minimum size=3pt,inner sep=1pt]{};
\end{scope}\end{tikzpicture}{\begin{tikzpicture}\begin{scope}[scale=.25] 
\draw [black!25](-1,-1) -- (-2,-2); 
\draw [black!25](-1,1) -- (-1,-1); 
\draw [red, very thick](-1,1) -- (-2,2); 
\draw [black!25](-1,1) -- (-1,2); 
\draw [red, very thick](0,0) -- (1,1); 
\draw [black!25](0,0) -- (1,-1); 
\draw [black!25](-1,-1) -- (0,0); 
\draw [black!25](-1,-1) -- (-1,-2); 
\draw [red, very thick](-1,1) -- (0,0); 
\end{scope}\end{tikzpicture}}}
\newcommand{\ZARd}{\begin{tikzpicture}\begin{scope}[scale=.33]\draw  [black!25] (0,-1) -- (-1,0); 
\draw  [black!25] (0,0) -- (-1,0); 
\draw  [black!25] (-1,0) -- (0,1); 
\draw  [black!25] (0,1) -- (1,1); 
\draw  [black!25] (1,1) -- (2,0); 
\draw  [black!25] (2,0) -- (1,-1); 
\draw  [black!25] (0,0) -- (1,-1); 
\draw  [black!25] (1,-1) -- (0,-1); 
\draw  [black!25] (1,1) -- (0,0); 
\draw (-1,0) node[circle,draw,fill=white,minimum size=3pt,inner sep=1pt]{};
\draw (0,1) node[circle,draw,fill=white,minimum size=3pt,inner sep=1pt]{};
\draw (1,1) node[circle,draw,fill=white,minimum size=3pt,inner sep=1pt]{};
\draw (2,0) node[circle,draw,fill=white,minimum size=3pt,inner sep=1pt]{};
\draw (1,-1) node[circle,draw,fill=black,minimum size=3pt,inner sep=1pt]{};
\draw (0,-1) node[circle,draw,fill=white,minimum size=3pt,inner sep=1pt]{};
\draw (0,0) node[circle,draw,fill=white,minimum size=3pt,inner sep=1pt]{};
\end{scope}\end{tikzpicture}{\begin{tikzpicture}\begin{scope}[scale=.25] 
\draw [black!25](-1,-1) -- (-2,-2); 
\draw [black!25](-1,1) -- (-1,-1); 
\draw [black!25](-1,1) -- (-2,2); 
\draw [black!25](-1,1) -- (-1,2); 
\draw [black!25](0,0) -- (1,1); 
\draw [red, very thick](0,0) -- (1,-1); 
\draw [red, very thick](-1,-1) -- (0,0); 
\draw [red, very thick](-1,-1) -- (-1,-2); 
\draw [black!25](-1,1) -- (0,0); 
\end{scope}\end{tikzpicture}}}
\newcommand{\ZARe}{\begin{tikzpicture}\begin{scope}[scale=.33]\draw  [black!25] (0,-1) -- (-1,0); 
\draw  [black!25] (0,0) -- (-1,0); 
\draw  [black] (-1,0) -- (0,1); 
\draw  [black!25] (0,1) -- (1,1); 
\draw  [black!25] (1,1) -- (2,0); 
\draw  [black!25] (2,0) -- (1,-1); 
\draw  [black!25] (0,0) -- (1,-1); 
\draw  [black!25] (1,-1) -- (0,-1); 
\draw  [black!25] (1,1) -- (0,0); 
\draw (-1,0) node[circle,draw,fill=black,minimum size=3pt,inner sep=1pt]{};
\draw (0,1) node[circle,draw,fill=black,minimum size=3pt,inner sep=1pt]{};
\draw (1,1) node[circle,draw,fill=white,minimum size=3pt,inner sep=1pt]{};
\draw (2,0) node[circle,draw,fill=white,minimum size=3pt,inner sep=1pt]{};
\draw (1,-1) node[circle,draw,fill=white,minimum size=3pt,inner sep=1pt]{};
\draw (0,-1) node[circle,draw,fill=white,minimum size=3pt,inner sep=1pt]{};
\draw (0,0) node[circle,draw,fill=white,minimum size=3pt,inner sep=1pt]{};
\end{scope}\end{tikzpicture}{\begin{tikzpicture}\begin{scope}[scale=.25] 
\draw [red, very thick](-1,-1) -- (-2,-2); 
\draw [red, very thick](-1,1) -- (-1,-1); 
\draw [black!25](-1,1) -- (-2,2); 
\draw [red, very thick](-1,1) -- (-1,2); 
\draw [black!25](0,0) -- (1,1); 
\draw [black!25](0,0) -- (1,-1); 
\draw [black!25](-1,-1) -- (0,0); 
\draw [black!25](-1,-1) -- (-1,-2); 
\draw [black!25](-1,1) -- (0,0); 
\end{scope}\end{tikzpicture}}}
\newcommand{\ZARf}{\begin{tikzpicture}\begin{scope}[scale=.33]\draw  [black!25] (0,-1) -- (-1,0); 
\draw  [black!25] (0,0) -- (-1,0); 
\draw  [black!25] (-1,0) -- (0,1); 
\draw  [black!25] (0,1) -- (1,1); 
\draw  [black!25] (1,1) -- (2,0); 
\draw  [black!25] (2,0) -- (1,-1); 
\draw  [black!25] (0,0) -- (1,-1); 
\draw  [black!25] (1,-1) -- (0,-1); 
\draw  [black!25] (1,1) -- (0,0); 
\draw (-1,0) node[circle,draw,fill=white,minimum size=3pt,inner sep=1pt]{};
\draw (0,1) node[circle,draw,fill=white,minimum size=3pt,inner sep=1pt]{};
\draw (1,1) node[circle,draw,fill=white,minimum size=3pt,inner sep=1pt]{};
\draw (2,0) node[circle,draw,fill=black,minimum size=3pt,inner sep=1pt]{};
\draw (1,-1) node[circle,draw,fill=white,minimum size=3pt,inner sep=1pt]{};
\draw (0,-1) node[circle,draw,fill=white,minimum size=3pt,inner sep=1pt]{};
\draw (0,0) node[circle,draw,fill=white,minimum size=3pt,inner sep=1pt]{};
\end{scope}\end{tikzpicture}{\begin{tikzpicture}\begin{scope}[scale=.25] 
\draw [black!25](-1,-1) -- (-2,-2); 
\draw [black!25](-1,1) -- (-1,-1); 
\draw [black!25](-1,1) -- (-2,2); 
\draw [black!25](-1,1) -- (-1,2); 
\draw [red, very thick](0,0) -- (1,1); 
\draw [red, very thick](0,0) -- (1,-1); 
\draw [black!25](-1,-1) -- (0,0); 
\draw [black!25](-1,-1) -- (-1,-2); 
\draw [black!25](-1,1) -- (0,0); 
\end{scope}\end{tikzpicture}}}
\newcommand{\ZARg}{\begin{tikzpicture}\begin{scope}[scale=.33]\draw  [black] (0,-1) -- (-1,0); 
\draw  [black!25] (0,0) -- (-1,0); 
\draw  [black!25] (-1,0) -- (0,1); 
\draw  [black!25] (0,1) -- (1,1); 
\draw  [black!25] (1,1) -- (2,0); 
\draw  [black!25] (2,0) -- (1,-1); 
\draw  [black!25] (0,0) -- (1,-1); 
\draw  [black!25] (1,-1) -- (0,-1); 
\draw  [black!25] (1,1) -- (0,0); 
\draw (-1,0) node[circle,draw,fill=black,minimum size=3pt,inner sep=1pt]{};
\draw (0,1) node[circle,draw,fill=white,minimum size=3pt,inner sep=1pt]{};
\draw (1,1) node[circle,draw,fill=white,minimum size=3pt,inner sep=1pt]{};
\draw (2,0) node[circle,draw,fill=white,minimum size=3pt,inner sep=1pt]{};
\draw (1,-1) node[circle,draw,fill=white,minimum size=3pt,inner sep=1pt]{};
\draw (0,-1) node[circle,draw,fill=black,minimum size=3pt,inner sep=1pt]{};
\draw (0,0) node[circle,draw,fill=white,minimum size=3pt,inner sep=1pt]{};
\end{scope}\end{tikzpicture}{\begin{tikzpicture}\begin{scope}[scale=.25] 
\draw [black!25](-1,-1) -- (-2,-2); 
\draw [red, very thick](-1,1) -- (-1,-1); 
\draw [red, very thick](-1,1) -- (-2,2); 
\draw [black!25](-1,1) -- (-1,2); 
\draw [black!25](0,0) -- (1,1); 
\draw [black!25](0,0) -- (1,-1); 
\draw [black!25](-1,-1) -- (0,0); 
\draw [red, very thick](-1,-1) -- (-1,-2); 
\draw [black!25](-1,1) -- (0,0); 
\end{scope}\end{tikzpicture}}}
\newcommand{\ZARh}{\begin{tikzpicture}\begin{scope}[scale=.33]\draw  [black!25] (0,-1) -- (-1,0); 
\draw  [black!25] (0,0) -- (-1,0); 
\draw  [black!25] (-1,0) -- (0,1); 
\draw  [black!25] (0,1) -- (1,1); 
\draw  [black!25] (1,1) -- (2,0); 
\draw  [black!25] (2,0) -- (1,-1); 
\draw  [black!25] (0,0) -- (1,-1); 
\draw  [black] (1,-1) -- (0,-1); 
\draw  [black!25] (1,1) -- (0,0); 
\draw (-1,0) node[circle,draw,fill=white,minimum size=3pt,inner sep=1pt]{};
\draw (0,1) node[circle,draw,fill=white,minimum size=3pt,inner sep=1pt]{};
\draw (1,1) node[circle,draw,fill=white,minimum size=3pt,inner sep=1pt]{};
\draw (2,0) node[circle,draw,fill=white,minimum size=3pt,inner sep=1pt]{};
\draw (1,-1) node[circle,draw,fill=black,minimum size=3pt,inner sep=1pt]{};
\draw (0,-1) node[circle,draw,fill=black,minimum size=3pt,inner sep=1pt]{};
\draw (0,0) node[circle,draw,fill=white,minimum size=3pt,inner sep=1pt]{};
\end{scope}\end{tikzpicture}{\begin{tikzpicture}\begin{scope}[scale=.25] 
\draw [red, very thick](-1,-1) -- (-2,-2); 
\draw [black!25](-1,1) -- (-1,-1); 
\draw [black!25](-1,1) -- (-2,2); 
\draw [black!25](-1,1) -- (-1,2); 
\draw [black!25](0,0) -- (1,1); 
\draw [red, very thick](0,0) -- (1,-1); 
\draw [red, very thick](-1,-1) -- (0,0); 
\draw [black!25](-1,-1) -- (-1,-2); 
\draw [black!25](-1,1) -- (0,0); 
\end{scope}\end{tikzpicture}}}
\newcommand{\ZARi}{\begin{tikzpicture}\begin{scope}[scale=.33]\draw  [black!25] (0,-1) -- (-1,0); 
\draw  [black!25] (0,0) -- (-1,0); 
\draw  [black!25] (-1,0) -- (0,1); 
\draw  [black] (0,1) -- (1,1); 
\draw  [black!25] (1,1) -- (2,0); 
\draw  [black!25] (2,0) -- (1,-1); 
\draw  [black!25] (0,0) -- (1,-1); 
\draw  [black!25] (1,-1) -- (0,-1); 
\draw  [black!25] (1,1) -- (0,0); 
\draw (-1,0) node[circle,draw,fill=white,minimum size=3pt,inner sep=1pt]{};
\draw (0,1) node[circle,draw,fill=black,minimum size=3pt,inner sep=1pt]{};
\draw (1,1) node[circle,draw,fill=black,minimum size=3pt,inner sep=1pt]{};
\draw (2,0) node[circle,draw,fill=white,minimum size=3pt,inner sep=1pt]{};
\draw (1,-1) node[circle,draw,fill=white,minimum size=3pt,inner sep=1pt]{};
\draw (0,-1) node[circle,draw,fill=white,minimum size=3pt,inner sep=1pt]{};
\draw (0,0) node[circle,draw,fill=white,minimum size=3pt,inner sep=1pt]{};
\end{scope}\end{tikzpicture}{\begin{tikzpicture}\begin{scope}[scale=.25] 
\draw [black!25](-1,-1) -- (-2,-2); 
\draw [black!25](-1,1) -- (-1,-1); 
\draw [red, very thick](-1,1) -- (-2,2); 
\draw [black!25](-1,1) -- (-1,2); 
\draw [red, very thick](0,0) -- (1,1); 
\draw [black!25](0,0) -- (1,-1); 
\draw [black!25](-1,-1) -- (0,0); 
\draw [black!25](-1,-1) -- (-1,-2); 
\draw [red, very thick](-1,1) -- (0,0); 
\end{scope}\end{tikzpicture}}}
\newcommand{\ZARj}{\begin{tikzpicture}\begin{scope}[scale=.33]\draw  [black!25] (0,-1) -- (-1,0); 
\draw  [black] (0,0) -- (-1,0); 
\draw  [black!25] (-1,0) -- (0,1); 
\draw  [black!25] (0,1) -- (1,1); 
\draw  [black!25] (1,1) -- (2,0); 
\draw  [black!25] (2,0) -- (1,-1); 
\draw  [black!25] (0,0) -- (1,-1); 
\draw  [black!25] (1,-1) -- (0,-1); 
\draw  [black!25] (1,1) -- (0,0); 
\draw (-1,0) node[circle,draw,fill=black,minimum size=3pt,inner sep=1pt]{};
\draw (0,1) node[circle,draw,fill=white,minimum size=3pt,inner sep=1pt]{};
\draw (1,1) node[circle,draw,fill=white,minimum size=3pt,inner sep=1pt]{};
\draw (2,0) node[circle,draw,fill=white,minimum size=3pt,inner sep=1pt]{};
\draw (1,-1) node[circle,draw,fill=white,minimum size=3pt,inner sep=1pt]{};
\draw (0,-1) node[circle,draw,fill=white,minimum size=3pt,inner sep=1pt]{};
\draw (0,0) node[circle,draw,fill=black,minimum size=3pt,inner sep=1pt]{};
\end{scope}\end{tikzpicture}{\begin{tikzpicture}\begin{scope}[scale=.25] 
\draw [red, very thick](-1,-1) -- (-2,-2); 
\draw [black!25](-1,1) -- (-1,-1); 
\draw [red, very thick](-1,1) -- (-2,2); 
\draw [black!25](-1,1) -- (-1,2); 
\draw [black!25](0,0) -- (1,1); 
\draw [black!25](0,0) -- (1,-1); 
\draw [red, very thick](-1,-1) -- (0,0); 
\draw [black!25](-1,-1) -- (-1,-2); 
\draw [red, very thick](-1,1) -- (0,0); 
\end{scope}\end{tikzpicture}}}
\newcommand{\ZARba}{\begin{tikzpicture}\begin{scope}[scale=.33]\draw  [black!25] (0,-1) -- (-1,0); 
\draw  [black!25] (0,0) -- (-1,0); 
\draw  [black!25] (-1,0) -- (0,1); 
\draw  [black!25] (0,1) -- (1,1); 
\draw  [black!25] (1,1) -- (2,0); 
\draw  [black!25] (2,0) -- (1,-1); 
\draw  [black!25] (0,0) -- (1,-1); 
\draw  [black!25] (1,-1) -- (0,-1); 
\draw  [black!25] (1,1) -- (0,0); 
\draw (-1,0) node[circle,draw,fill=white,minimum size=3pt,inner sep=1pt]{};
\draw (0,1) node[circle,draw,fill=white,minimum size=3pt,inner sep=1pt]{};
\draw (1,1) node[circle,draw,fill=white,minimum size=3pt,inner sep=1pt]{};
\draw (2,0) node[circle,draw,fill=black,minimum size=3pt,inner sep=1pt]{};
\draw (1,-1) node[circle,draw,fill=white,minimum size=3pt,inner sep=1pt]{};
\draw (0,-1) node[circle,draw,fill=white,minimum size=3pt,inner sep=1pt]{};
\draw (0,0) node[circle,draw,fill=white,minimum size=3pt,inner sep=1pt]{};
\end{scope}\end{tikzpicture}{\begin{tikzpicture}\begin{scope}[scale=.25] 
\draw [black!25](-1,-1) -- (-2,-2); 
\draw [black!25](-1,1) -- (-1,-1); 
\draw [black!25](-1,1) -- (-2,2); 
\draw [black!25](-1,1) -- (-1,2); 
\draw [red, very thick](0,0) -- (1,1); 
\draw [red, very thick](0,0) -- (1,-1); 
\draw [black!25](-1,-1) -- (0,0); 
\draw [black!25](-1,-1) -- (-1,-2); 
\draw [black!25](-1,1) -- (0,0); 
\end{scope}\end{tikzpicture}}}
\newcommand{\ZARbb}{\begin{tikzpicture}\begin{scope}[scale=.33]\draw  [black!25] (0,-1) -- (-1,0); 
\draw  [black!25] (0,0) -- (-1,0); 
\draw  [black!25] (-1,0) -- (0,1); 
\draw  [black!25] (0,1) -- (1,1); 
\draw  [black!25] (1,1) -- (2,0); 
\draw  [black!25] (2,0) -- (1,-1); 
\draw  [black!25] (0,0) -- (1,-1); 
\draw  [black] (1,-1) -- (0,-1); 
\draw  [black!25] (1,1) -- (0,0); 
\draw (-1,0) node[circle,draw,fill=white,minimum size=3pt,inner sep=1pt]{};
\draw (0,1) node[circle,draw,fill=white,minimum size=3pt,inner sep=1pt]{};
\draw (1,1) node[circle,draw,fill=white,minimum size=3pt,inner sep=1pt]{};
\draw (2,0) node[circle,draw,fill=white,minimum size=3pt,inner sep=1pt]{};
\draw (1,-1) node[circle,draw,fill=black,minimum size=3pt,inner sep=1pt]{};
\draw (0,-1) node[circle,draw,fill=black,minimum size=3pt,inner sep=1pt]{};
\draw (0,0) node[circle,draw,fill=white,minimum size=3pt,inner sep=1pt]{};
\end{scope}\end{tikzpicture}{\begin{tikzpicture}\begin{scope}[scale=.25] 
\draw [red, very thick](-1,-1) -- (-2,-2); 
\draw [black!25](-1,1) -- (-1,-1); 
\draw [black!25](-1,1) -- (-2,2); 
\draw [black!25](-1,1) -- (-1,2); 
\draw [black!25](0,0) -- (1,1); 
\draw [red, very thick](0,0) -- (1,-1); 
\draw [red, very thick](-1,-1) -- (0,0); 
\draw [black!25](-1,-1) -- (-1,-2); 
\draw [black!25](-1,1) -- (0,0); 
\end{scope}\end{tikzpicture}}}
\newcommand{\ZARbc}{\begin{tikzpicture}\begin{scope}[scale=.33]\draw  [black!25] (0,-1) -- (-1,0); 
\draw  [black!25] (0,0) -- (-1,0); 
\draw  [black!25] (-1,0) -- (0,1); 
\draw  [black!25] (0,1) -- (1,1); 
\draw  [black!25] (1,1) -- (2,0); 
\draw  [black!25] (2,0) -- (1,-1); 
\draw  [black!25] (0,0) -- (1,-1); 
\draw  [black!25] (1,-1) -- (0,-1); 
\draw  [black!25] (1,1) -- (0,0); 
\draw (-1,0) node[circle,draw,fill=white,minimum size=3pt,inner sep=1pt]{};
\draw (0,1) node[circle,draw,fill=white,minimum size=3pt,inner sep=1pt]{};
\draw (1,1) node[circle,draw,fill=black,minimum size=3pt,inner sep=1pt]{};
\draw (2,0) node[circle,draw,fill=white,minimum size=3pt,inner sep=1pt]{};
\draw (1,-1) node[circle,draw,fill=white,minimum size=3pt,inner sep=1pt]{};
\draw (0,-1) node[circle,draw,fill=white,minimum size=3pt,inner sep=1pt]{};
\draw (0,0) node[circle,draw,fill=white,minimum size=3pt,inner sep=1pt]{};
\end{scope}\end{tikzpicture}{\begin{tikzpicture}\begin{scope}[scale=.25] 
\draw [black!25](-1,-1) -- (-2,-2); 
\draw [black!25](-1,1) -- (-1,-1); 
\draw [black!25](-1,1) -- (-2,2); 
\draw [red, very thick](-1,1) -- (-1,2); 
\draw [red, very thick](0,0) -- (1,1); 
\draw [black!25](0,0) -- (1,-1); 
\draw [black!25](-1,-1) -- (0,0); 
\draw [black!25](-1,-1) -- (-1,-2); 
\draw [red, very thick](-1,1) -- (0,0); 
\end{scope}\end{tikzpicture}}}
\newcommand{\ZARbd}{\begin{tikzpicture}\begin{scope}[scale=.33]\draw  [black!25] (0,-1) -- (-1,0); 
\draw  [black!25] (0,0) -- (-1,0); 
\draw  [black!25] (-1,0) -- (0,1); 
\draw  [black!25] (0,1) -- (1,1); 
\draw  [black] (1,1) -- (2,0); 
\draw  [black!25] (2,0) -- (1,-1); 
\draw  [black!25] (0,0) -- (1,-1); 
\draw  [black!25] (1,-1) -- (0,-1); 
\draw  [black] (1,1) -- (0,0); 
\draw (-1,0) node[circle,draw,fill=white,minimum size=3pt,inner sep=1pt]{};
\draw (0,1) node[circle,draw,fill=white,minimum size=3pt,inner sep=1pt]{};
\draw (1,1) node[circle,draw,fill=black,minimum size=3pt,inner sep=1pt]{};
\draw (2,0) node[circle,draw,fill=black,minimum size=3pt,inner sep=1pt]{};
\draw (1,-1) node[circle,draw,fill=white,minimum size=3pt,inner sep=1pt]{};
\draw (0,-1) node[circle,draw,fill=white,minimum size=3pt,inner sep=1pt]{};
\draw (0,0) node[circle,draw,fill=black,minimum size=3pt,inner sep=1pt]{};
\end{scope}\end{tikzpicture}{\begin{tikzpicture}\begin{scope}[scale=.25] 
\draw [black!25](-1,-1) -- (-2,-2); 
\draw [red, very thick](-1,1) -- (-1,-1); 
\draw [black!25](-1,1) -- (-2,2); 
\draw [red, very thick](-1,1) -- (-1,2); 
\draw [black!25](0,0) -- (1,1); 
\draw [red, very thick](0,0) -- (1,-1); 
\draw [red, very thick](-1,-1) -- (0,0); 
\draw [black!25](-1,-1) -- (-1,-2); 
\draw [black!25](-1,1) -- (0,0); 
\end{scope}\end{tikzpicture}}}
\newcommand{\ZARbe}{\begin{tikzpicture}\begin{scope}[scale=.33]\draw  [black!25] (0,-1) -- (-1,0); 
\draw  [black] (0,0) -- (-1,0); 
\draw  [black!25] (-1,0) -- (0,1); 
\draw  [black!25] (0,1) -- (1,1); 
\draw  [black!25] (1,1) -- (2,0); 
\draw  [black!25] (2,0) -- (1,-1); 
\draw  [black!25] (0,0) -- (1,-1); 
\draw  [black!25] (1,-1) -- (0,-1); 
\draw  [black!25] (1,1) -- (0,0); 
\draw (-1,0) node[circle,draw,fill=black,minimum size=3pt,inner sep=1pt]{};
\draw (0,1) node[circle,draw,fill=white,minimum size=3pt,inner sep=1pt]{};
\draw (1,1) node[circle,draw,fill=white,minimum size=3pt,inner sep=1pt]{};
\draw (2,0) node[circle,draw,fill=white,minimum size=3pt,inner sep=1pt]{};
\draw (1,-1) node[circle,draw,fill=white,minimum size=3pt,inner sep=1pt]{};
\draw (0,-1) node[circle,draw,fill=white,minimum size=3pt,inner sep=1pt]{};
\draw (0,0) node[circle,draw,fill=black,minimum size=3pt,inner sep=1pt]{};
\end{scope}\end{tikzpicture}{\begin{tikzpicture}\begin{scope}[scale=.25] 
\draw [red, very thick](-1,-1) -- (-2,-2); 
\draw [black!25](-1,1) -- (-1,-1); 
\draw [red, very thick](-1,1) -- (-2,2); 
\draw [black!25](-1,1) -- (-1,2); 
\draw [black!25](0,0) -- (1,1); 
\draw [black!25](0,0) -- (1,-1); 
\draw [red, very thick](-1,-1) -- (0,0); 
\draw [black!25](-1,-1) -- (-1,-2); 
\draw [red, very thick](-1,1) -- (0,0); 
\end{scope}\end{tikzpicture}}}
\newcommand{\ZARbf}{\begin{tikzpicture}\begin{scope}[scale=.33]\draw  [black!25] (0,-1) -- (-1,0); 
\draw  [black!25] (0,0) -- (-1,0); 
\draw  [black!25] (-1,0) -- (0,1); 
\draw  [black!25] (0,1) -- (1,1); 
\draw  [black!25] (1,1) -- (2,0); 
\draw  [black] (2,0) -- (1,-1); 
\draw  [black] (0,0) -- (1,-1); 
\draw  [black!25] (1,-1) -- (0,-1); 
\draw  [black!25] (1,1) -- (0,0); 
\draw (-1,0) node[circle,draw,fill=white,minimum size=3pt,inner sep=1pt]{};
\draw (0,1) node[circle,draw,fill=white,minimum size=3pt,inner sep=1pt]{};
\draw (1,1) node[circle,draw,fill=white,minimum size=3pt,inner sep=1pt]{};
\draw (2,0) node[circle,draw,fill=black,minimum size=3pt,inner sep=1pt]{};
\draw (1,-1) node[circle,draw,fill=black,minimum size=3pt,inner sep=1pt]{};
\draw (0,-1) node[circle,draw,fill=white,minimum size=3pt,inner sep=1pt]{};
\draw (0,0) node[circle,draw,fill=black,minimum size=3pt,inner sep=1pt]{};
\end{scope}\end{tikzpicture}{\begin{tikzpicture}\begin{scope}[scale=.25] 
\draw [black!25](-1,-1) -- (-2,-2); 
\draw [red, very thick](-1,1) -- (-1,-1); 
\draw [black!25](-1,1) -- (-2,2); 
\draw [black!25](-1,1) -- (-1,2); 
\draw [red, very thick](0,0) -- (1,1); 
\draw [black!25](0,0) -- (1,-1); 
\draw [black!25](-1,-1) -- (0,0); 
\draw [red, very thick](-1,-1) -- (-1,-2); 
\draw [red, very thick](-1,1) -- (0,0); 
\end{scope}\end{tikzpicture}}}
\newcommand{\ZARbg}{\begin{tikzpicture}\begin{scope}[scale=.33]\draw  [black!25] (0,-1) -- (-1,0); 
\draw  [black!25] (0,0) -- (-1,0); 
\draw  [black!25] (-1,0) -- (0,1); 
\draw  [black!25] (0,1) -- (1,1); 
\draw  [black!25] (1,1) -- (2,0); 
\draw  [black!25] (2,0) -- (1,-1); 
\draw  [black!25] (0,0) -- (1,-1); 
\draw  [black!25] (1,-1) -- (0,-1); 
\draw  [black!25] (1,1) -- (0,0); 
\draw (-1,0) node[circle,draw,fill=white,minimum size=3pt,inner sep=1pt]{};
\draw (0,1) node[circle,draw,fill=black,minimum size=3pt,inner sep=1pt]{};
\draw (1,1) node[circle,draw,fill=white,minimum size=3pt,inner sep=1pt]{};
\draw (2,0) node[circle,draw,fill=white,minimum size=3pt,inner sep=1pt]{};
\draw (1,-1) node[circle,draw,fill=white,minimum size=3pt,inner sep=1pt]{};
\draw (0,-1) node[circle,draw,fill=white,minimum size=3pt,inner sep=1pt]{};
\draw (0,0) node[circle,draw,fill=white,minimum size=3pt,inner sep=1pt]{};
\end{scope}\end{tikzpicture}{\begin{tikzpicture}\begin{scope}[scale=.25] 
\draw [black!25](-1,-1) -- (-2,-2); 
\draw [black!25](-1,1) -- (-1,-1); 
\draw [red, very thick](-1,1) -- (-2,2); 
\draw [red, very thick](-1,1) -- (-1,2); 
\draw [black!25](0,0) -- (1,1); 
\draw [black!25](0,0) -- (1,-1); 
\draw [black!25](-1,-1) -- (0,0); 
\draw [black!25](-1,-1) -- (-1,-2); 
\draw [black!25](-1,1) -- (0,0); 
\end{scope}\end{tikzpicture}}}
\vspace{.2cm}
\begin{tabular}{|cc|cc|cc|ccc|}\hline1&\ZARb&2&\ZARc&3&\ZARd&4&\ZARe&\\ \hline 
5&\ZARf&6&\ZARg&7&\ZARh&8&\ZARi&\\ \hline 
9&\ZARj&10&\ZARba&11&\ZARbb&12&\ZARbc&\\ \hline 
13&\ZARbd&14&\ZARbe&15&\ZARbf&16&\ZARbg&\\ \hline 
\end{tabular}
\vspace{.3cm}

To find the dimers corresponding to the three nodes of the spider graph, 
we need to contract
the arrows of which the line does not run through that node. 
For the upper left node these are 
$\{1, 3, 5, 7, 10, 11\}$, for the lower left node $\{2, 5, 8, 10, 12, 16\}$ and for the right node
$\{1,4,12,16\}$. This gives us the following reduced dimers. The $2$-cycles that can be
removed are drawn in light grey.

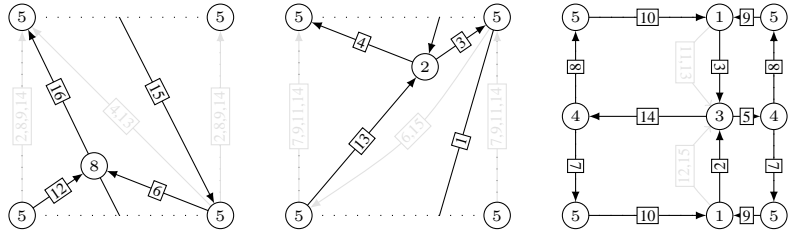
\begin{center}
\begin{tikzpicture}
\begin{scope}[scale=.25]
\begin{scope}
\draw[loosely dotted] (102pt,20pt) rectangle (402pt,320pt);
\draw [-latex,shorten >=5pt] (402pt,20pt) to node [rectangle,draw,fill=white,sloped,inner sep=1pt] {{\tiny 6}} (211pt,95pt); 
\draw [latex-latex,shorten >=5pt,black!12] (402pt,20pt) to node [rectangle,draw,fill=white,sloped,inner sep=1pt] {{\tiny 2,8,9,14}} (402pt,320pt); 
\draw [latex-latex,shorten >=5pt,black!12] (102pt,20pt) to node [rectangle,draw,fill=white,sloped,inner sep=1pt] {{\tiny 2,8,9,14}} (102pt,320pt); 
\draw [latex-latex,shorten >=5pt,black!12] (402pt,20pt) to node [rectangle,draw,fill=white,sloped,inner sep=1pt] {{\tiny 4,13}} (102pt,320pt); 
\draw [-latex,shorten >=5pt] (102pt,20pt) to node [rectangle,draw,fill=white,sloped,inner sep=1pt] {{\tiny 12}} (211pt,95pt); 
\begin{scope}\clip (102pt,20pt) rectangle (402pt,320pt);
\draw [-latex,shorten >=5pt] (211pt,95pt) to node [rectangle,draw,fill=white,sloped,inner sep=1pt] {{\tiny 15}} (402pt,-280pt); 
\draw [-latex,shorten >=5pt] (211pt,395pt) to node [rectangle,draw,fill=white,sloped,inner sep=1pt] {{\tiny 15}} (402pt,20pt); 
\end{scope}
\draw [-latex,shorten >=5pt] (211pt,95pt) to node [rectangle,draw,fill=white,sloped,inner sep=1pt] {{\tiny 16}} (102pt,320pt); 

\node at (102pt,20pt) [circle,draw,fill=white,minimum size=10pt,inner sep=1pt] {\mbox{\tiny $5$}}; 
\node at (402pt,20pt) [circle,draw,fill=white,minimum size=10pt,inner sep=1pt] {\mbox{\tiny $5$}}; 
\node at (402pt,320pt) [circle,draw,fill=white,minimum size=10pt,inner sep=1pt] {\mbox{\tiny $5$}}; 
\node at (102pt,320pt) [circle,draw,fill=white,minimum size=10pt,inner sep=1pt] {\mbox{\tiny $5$}}; 
\node at (211pt,95pt) [circle,draw,fill=white,minimum size=10pt,inner sep=1pt] {\mbox{\tiny $8$}}; 
\end{scope}\end{scope}
\end{tikzpicture}
\hspace{.5cm}
\begin{tikzpicture}
\begin{scope}[scale=.25]
\begin{scope}
\draw[loosely dotted] (102pt,20pt) rectangle (402pt,320pt);
\begin{scope}\clip (102pt,20pt) rectangle (402pt,320pt);
\draw [-latex,shorten >=5pt] (402pt,620pt) to node [rectangle,draw,fill=white,sloped,inner sep=1pt] {{\tiny 1}} (293pt,245pt); 
\draw [-latex,shorten >=5pt] (402pt,320pt) to node [rectangle,draw,fill=white,sloped,inner sep=1pt] {{\tiny 1}} (293pt,-55pt); 
\end{scope}
\draw [latex-latex,shorten >=5pt,black!12] (402pt,20pt) to node [rectangle,draw,fill=white,sloped,inner sep=1pt] {{\tiny 7,9,11,14}} (402pt,320pt); 
\draw [latex-latex,shorten >=5pt,black!12] (102pt,20pt) to node [rectangle,draw,fill=white,sloped,inner sep=1pt] {{\tiny 7,9,11,14}} (102pt,320pt); 
\draw [latex-latex,shorten >=5pt,black!12] (402pt,320pt) .. controls (282pt,140pt) and (282pt,140pt) .. node [rectangle,draw,fill=white,sloped,inner sep=1pt] {{\tiny 6,15}} (102pt,20pt); 

\draw [-latex,shorten >=5pt] (293pt,245pt) to node [rectangle,draw,fill=white,sloped,inner sep=1pt] {{\tiny 3}} (402pt,320pt); 
\draw [-latex,shorten >=5pt] (293pt,245pt) to node [rectangle,draw,fill=white,sloped,inner sep=1pt] {{\tiny 4}} (102pt,320pt); 
\draw [-latex,shorten >=5pt] (102pt,20pt) to node [rectangle,draw,fill=white,sloped,inner sep=1pt] {{\tiny 13}} (293pt,245pt); 

\node at (293pt,245pt) [circle,draw,fill=white,minimum size=10pt,inner sep=1pt] {\mbox{\tiny $2$}}; 
\node at (102pt,20pt) [circle,draw,fill=white,minimum size=10pt,inner sep=1pt] {\mbox{\tiny $5$}}; 
\node at (402pt,20pt) [circle,draw,fill=white,minimum size=10pt,inner sep=1pt] {\mbox{\tiny $5$}}; 
\node at (402pt,320pt) [circle,draw,fill=white,minimum size=10pt,inner sep=1pt] {\mbox{\tiny $5$}}; 
\node at (102pt,320pt) [circle,draw,fill=white,minimum size=10pt,inner sep=1pt] {\mbox{\tiny $5$}}; 
\end{scope}\end{scope}
\end{tikzpicture}
\hspace{.5cm}
\begin{tikzpicture}
\begin{scope}[scale=.25]
\begin{scope}
\draw[loosely dotted] (102pt,20pt) rectangle (402pt,320pt);
\draw [-latex,shorten >=5pt] (320pt,320pt) to node [rectangle,draw,fill=white,sloped,inner sep=1pt] {{\tiny 3}} (320pt,170pt); 
\draw [<->,shorten >=5pt,black!12] (320pt,320pt) .. controls (250pt,254pt) and (250pt,254pt) .. node [rectangle,draw,fill=white,sloped,inner sep=1pt] {{\tiny 11,13}} (320pt,170pt); 
\draw [-latex,shorten >=5pt] (320pt,20pt) to node [rectangle,draw,fill=white,sloped,inner sep=1pt] {{\tiny 2}} (320pt,170pt); 
\draw [<->,shorten >=5pt,black!12] (320pt,20pt) .. controls (250pt,95pt) and (250pt,95pt) .. node [rectangle,draw,fill=white,sloped,inner sep=1pt] {{\tiny 12,15}} (320pt,170pt); 
\draw [-latex,shorten >=5pt] (320pt,170pt) to node [rectangle,draw,fill=white,sloped,inner sep=1pt] {{\tiny 5}} (402pt,170pt); 
\draw [-latex,shorten >=5pt] (102pt,170pt) to node [rectangle,draw,fill=white,sloped,inner sep=1pt] {{\tiny 7}} (102pt,20pt); 
\draw [-latex,shorten >=5pt] (402pt,170pt) to node [rectangle,draw,fill=white,sloped,inner sep=1pt] {{\tiny 7}} (402pt,20pt); 
\draw [-latex,shorten >=5pt] (102pt,170pt) to node [rectangle,draw,fill=white,sloped,inner sep=1pt] {{\tiny 8}} (102pt,320pt); 
\draw [-latex,shorten >=5pt] (402pt,170pt) to node [rectangle,draw,fill=white,sloped,inner sep=1pt] {{\tiny 8}} (402pt,320pt); 
\draw [-latex,shorten >=5pt] (402pt,20pt) to node [rectangle,draw,fill=white,sloped,inner sep=1pt] {{\tiny 9}} (320pt,20pt); 
\draw [-latex,shorten >=5pt] (402pt,320pt) to node [rectangle,draw,fill=white,sloped,inner sep=1pt] {{\tiny 9}} (320pt,320pt); 
\draw [-latex,shorten >=5pt] (102pt,20pt) to node [rectangle,draw,fill=white,sloped,inner sep=1pt] {{\tiny 10}} (320pt,20pt); 
\draw [-latex,shorten >=5pt] (102pt,320pt) to node [rectangle,draw,fill=white,sloped,inner sep=1pt] {{\tiny 10}} (320pt,320pt);  
\draw [-latex,shorten >=5pt] (320pt,170pt) to node [rectangle,draw,fill=white,sloped,inner sep=1pt] {{\tiny 14}} (102pt,170pt); 
\node at (320pt,20pt) [circle,draw,fill=white,minimum size=10pt,inner sep=1pt] {\mbox{\tiny $1$}}; 
\node at (320pt,320pt) [circle,draw,fill=white,minimum size=10pt,inner sep=1pt] {\mbox{\tiny $1$}}; 
\node at (320pt,170pt) [circle,draw,fill=white,minimum size=10pt,inner sep=1pt] {\mbox{\tiny $3$}}; 
\node at (102pt,170pt) [circle,draw,fill=white,minimum size=10pt,inner sep=1pt] {\mbox{\tiny $4$}}; 
\node at (402pt,170pt) [circle,draw,fill=white,minimum size=10pt,inner sep=1pt] {\mbox{\tiny $4$}}; 
\node at (102pt,20pt) [circle,draw,fill=white,minimum size=10pt,inner sep=1pt] {\mbox{\tiny $5$}}; 
\node at (402pt,20pt) [circle,draw,fill=white,minimum size=10pt,inner sep=1pt] {\mbox{\tiny $5$}}; 
\node at (402pt,320pt) [circle,draw,fill=white,minimum size=10pt,inner sep=1pt] {\mbox{\tiny $5$}}; 
\node at (102pt,320pt) [circle,draw,fill=white,minimum size=10pt,inner sep=1pt] {\mbox{\tiny $5$}}; 
\end{scope}\end{scope}
\end{tikzpicture}
\end{center}
If an arrow $a$ is part of a $2$-cycle $ab$, the path that runs in the same direction but on the other side of the arrow $b$ has the same homotopy class and $\PM_i$-degree 
as $a$ and must hence be equal to $a$. Therefore we can remove $2$-cycles from the dimer quiver without changing the Jacobi algebra. We have drawn these $2$-cycles in grey.
After removing the $2$-cycles, the first two dimer quivers can be recognized as the conifold quiver (up to a Dehn twist of the torus), while the third is the dimer for local
$\PP_2$, which is in agreement with the shapes of the polygons in the subdivision.

\bibliography{qdiff}{}

\begin{thebibliography}{10}

\bibitem{abouzaid2013homological}
Mohammed Abouzaid, Denis Auroux, Alexander Efimov, Ludmil Katzarkov, and Dmitri
  Orlov.
\newblock Homological mirror symmetry for punctured spheres.
\newblock {\em Journal of the American Mathematical Society}, 26(4):1051--1083,
  2013.

\bibitem{adriaenssens2003local}
Jan Adriaenssens and Lieven~Le Bruyn.
\newblock Local quivers and stable representations.
\newblock {\em Communications in Algebra}, 31(4):1777--1797, 2003.

\bibitem{bocklandt2008graded}
Raf Bocklandt.
\newblock Graded calabi yau algebras of dimension 3.
\newblock {\em Journal of pure and applied algebra}, 212(1):14--32, 2008.

\bibitem{bocklandt2012consistency}
Raf Bocklandt.
\newblock Consistency conditions for dimer models.
\newblock {\em Glasgow Mathematical Journal}, 54(02):429--447, 2012.

\bibitem{bocklandt2016dimer}
Raf Bocklandt.
\newblock A dimer abc.
\newblock {\em Bulletin of the London Mathematical Society}, 48(3):387--451,
  2016.

\bibitem{bocklandt2016noncommutative}
Raf Bocklandt.
\newblock Noncommutative mirror symmetry for punctured surfaces (with an
  appendix by mohammed abouzaid).
\newblock {\em Transactions of the American Mathematical Society},
  368(1):429--469, 2016.

\bibitem{bridgeland2015quadratic}
Tom Bridgeland and Ivan Smith.
\newblock Quadratic differentials as stability conditions.
\newblock {\em Publications math{\'e}matiques de l'IH{\'E}S}, 121(1):155--278,
  2015.

\bibitem{broomhead2012dimer}
Nathan Broomhead.
\newblock {\em Dimer models and Calabi-Yau algebras}, volume 211.
\newblock American Mathematical Soc., 2012.

\bibitem{buchweitz1986maximal}
Ragnar-Olaf Buchweitz.
\newblock Maximal cohen-macaulay modules and tate-cohomology over gorenstein
  rings.
\newblock 1986.

\bibitem{burban2010maximal}
Igor Burban and Yuriy Drozd.
\newblock Maximal cohen-macaulay modules over non-isolated surface
  singularities and matrix problems.
\newblock {\em arXiv preprint arXiv:1002.3042}, 2010.

\bibitem{davison2011consistency}
Ben Davison.
\newblock Consistency conditions for brane tilings.
\newblock {\em Journal of Algebra}, 338(1):1--23, 2011.

\bibitem{dyckerhoff2015a1}
Tobias Dyckerhoff.
\newblock A1-homotopy invariants of topological fukaya categories of surfaces.
\newblock {\em arXiv preprint arXiv:1505.06941}, 2015.

\bibitem{dyckerhoff2013triangulated}
Tobias Dyckerhoff and Mikhail Kapranov.
\newblock Triangulated surfaces in triangulated categories.
\newblock {\em arXiv preprint arXiv:1306.2545}, 2013.

\bibitem{feng2008dimer}
Bo~Feng, Yang-Hui He, Kristian~D Kennaway, Cumrun Vafa, et~al.
\newblock Dimer models from mirror symmetry and quivering amoebae.
\newblock {\em Advances in Theoretical and Mathematical Physics},
  12(3):489--545, 2008.

\bibitem{franco2006brane}
Sebasti{\'a}n Franco, Amihay Hanany, David Vegh, Brian Wecht, and Kristian~D
  Kennaway.
\newblock Brane dimers and quiver gauge theories.
\newblock {\em Journal of High Energy Physics}, 2006(01):096, 2006.

\bibitem{fulton1993introduction}
William Fulton.
\newblock {\em Introduction to toric varieties}.
\newblock Number 131. Princeton University Press, 1993.

\bibitem{gulotta2008properly}
Daniel~R Gulotta.
\newblock Properly ordered dimers, r-charges, and an efficient inverse
  algorithm.
\newblock {\em Journal of High Energy Physics}, 2008(10):014, 2008.

\bibitem{haiden2014flat}
Fabian Haiden, Ludmil Katzarkov, and Maxim Kontsevich.
\newblock Flat surfaces and stability structures.
\newblock {\em arXiv preprint arXiv:1409.8611}, 2014.

\bibitem{hanany2012brane}
Amihay Hanany and Rak-Kyeong Seong.
\newblock Brane tilings and specular duality.
\newblock {\em Journal of High Energy Physics}, 2012(8):1--36, 2012.

\bibitem{ishii2009dimer}
Akira Ishii and Kazushi Ueda.
\newblock Dimer models and the special mckay correspondence.
\newblock {\em arXiv preprint arXiv:0905.0059}, 2009.

\bibitem{ishii2010note}
Akira Ishii and Kazushi Ueda.
\newblock A note on consistency conditions on dimer models.
\newblock {\em RIMS K\^oky\^uroku Bessatsu}, B24:143--164, 2011.

\bibitem{kadeishvili1980homology}
Tornike~V Kadeishvili.
\newblock On the homology theory of fibre spaces.
\newblock {\em Russian Mathematical Surveys}, 35(3):231--238, 1980.

\bibitem{keller1999introduction}
Bernhard Keller.
\newblock Introduction to a-infinity algebras and modules.
\newblock {\em arXiv preprint math.RA/9910179}, 1999.

\bibitem{keller2006infinity}
Bernhard Keller.
\newblock A-infinity algebras, modules and functor categories.
\newblock {\em Contemporary Mathematics}, 406:67--94, 2006.

\bibitem{king1994moduli}
Alastair~D King.
\newblock Moduli of representations of finite dimensional algebras.
\newblock {\em The Quarterly Journal of Mathematics}, 45(4):515--530, 1994.

\bibitem{kontsevich606241notes}
Maxim Kontsevich and Yan Soibelman.
\newblock Notes on a-infinity algebras, a-infinity categories and
  noncommutative geometry. i, arxive.
\newblock {\em arXiv preprint math.RA/0606241}.

\bibitem{liu2008jenkins}
Jinsong Liu.
\newblock Jenkins--strebel differentials with poles.
\newblock {\em Commentarii Mathematici Helvetici}, 83(1):211--240, 2008.

\bibitem{mikhalkin2004amoebas}
Grigory Mikhalkin.
\newblock Amoebas of algebraic varieties and tropical geometry.
\newblock In {\em Different faces of geometry}, pages 257--300. Springer, 2004.

\bibitem{mozgovoy2009crepant}
Sergey Mozgovoy.
\newblock Crepant resolutions and brane tilings i: Toric realization.
\newblock {\em arXiv preprint arXiv:0908.3475}, 2009.

\bibitem{mozgovoy2010noncommutative}
Sergey Mozgovoy and Markus Reineke.
\newblock On the noncommutative donaldson--thomas invariants arising from brane
  tilings.
\newblock {\em Advances in mathematics}, 223(5):1521--1544, 2010.

\bibitem{mulase1998ribbon}
Motohico Mulase and Michael Penkava.
\newblock Ribbon graphs, quadratic differentials on riemann surfaces, and
  algebraic curves defined over $\bar q$.
\newblock {\em arXiv preprint math-ph/9811024}, 1998.

\bibitem{pascaleff2016topological}
James Pascaleff and Nicol{\`o} Sibilla.
\newblock Topological fukaya category and mirror symmetry for punctured
  surfaces.
\newblock {\em arXiv preprint arXiv:1604.06448}, 2016.

\bibitem{strebel1984quadratic}
Kurt Strebel.
\newblock Quadratic differentials.
\newblock In {\em Quadratic Differentials}, pages 16--26. Springer, 1984.

\end{thebibliography}
\bibliographystyle{plain}

\end{document}